\documentclass{article}

\usepackage{etex}
\usepackage{graphicx}
\usepackage{latexsym}
\usepackage{amsfonts}
\usepackage{amssymb}
\usepackage{amsmath}
\usepackage{verbatim}
\usepackage{url}
\usepackage[standard, thmmarks]{ntheorem}
\usepackage[margin=3cm]{geometry}
\usepackage{fancyhdr}
\usepackage{enumerate}

\newtheorem{thm}{Theorem}[section]
\newtheorem{cor}[thm]{Corollary}
\newtheorem{lem}[thm]{Lemma}
\newtheorem{prop}[thm]{Proposition}

\newtheorem{defn}[thm]{Definition}

\newcommand{\R}{\mathbb R}

\newcommand{\CP}{\mathbb C P}

\newcommand{\Z}{\mathbb Z}
\newcommand{\C}{\mathbb C}

\newcommand{\To}{\longrightarrow}

\newcommand{\N}{\mathbb N}

\newcommand{\E}{\mathcal{E}}
\newcommand{\z}{{\bf z}}

\newcommand{\D}{\mathcal{D}}
\newcommand{\Pent}{\mathcal{P}}
\newcommand{\V}{\mathcal{V}}

\newcommand{\hyp}{\mathbb H}
\newcommand{\Out}{\text{Out}\,}

\DeclareMathOperator{\Par}{Par} \DeclareMathOperator{\Hyp}{Hyp}
\DeclareMathOperator{\Ell}{Ell} 

 \DeclareMathOperator{\mcg}{MCG}
\DeclareMathOperator{\Aut}{Aut} \DeclareMathOperator{\Inn}{Inn}
\DeclareMathOperator{\Cone}{Cone} \DeclareMathOperator{\Isom}{Isom}
\DeclareMathOperator{\Axis}{Axis} \DeclareMathOperator{\Twist}{Tw}
\DeclareMathOperator{\Tr}{Tr} 
\DeclareMathOperator{\Fix}{Fix}

\begin{document}

\title{Hyperbolic cone-manifold structures with prescribed holonomy I: punctured tori}

\author{Daniel V. Mathews}%

\date{}

\maketitle

\begin{abstract}
We consider the relationship between hyperbolic cone-manifold structures on surfaces, and algebraic representations of the fundamental group into a group of isometries. A hyperbolic cone-manifold structure on a surface, with all interior cone angles being integer multiples of $2\pi$, determines a holonomy representation of the fundamental group. We ask, conversely, when a representation of the fundamental group is the holonomy of a hyperbolic cone-manifold structure. In this paper we prove results for the punctured torus; in the sequel, for higher genus surfaces.

We show that a representation of the fundamental group of a punctured torus is a holonomy representation of a hyperbolic cone-manifold structure with no interior cone points and a single corner point if and only if it is not virtually abelian. We construct a pentagonal fundamental domain for hyperbolic structures, from the geometry of a representation. Our techniques involve the universal covering group $\widetilde{PSL_2\R}$ of the group of orientation-preserving isometries of $\hyp^2$ and Markoff moves arising from the action of the mapping class group on the character variety.
\end{abstract}

\tableofcontents

\section{Introduction}

\subsection{Background and results}

A geometric $(X, \Isom X)$ structure on an orientable manifold $M$ induces a holonomy representation $\rho: \pi_1(M) \To \Isom X$. Geometric data about the manifold is encapsulated in this representation. A natural question is: how much information? From a geometric structure it is an easy matter to obtain $\rho$; but the reverse direction is much longer and uphill. For any geometry, we may ask: given a representation $\rho: \pi_1(M) \To \Isom X$, is $\rho$ the holonomy of a geometric structure on $M$? The answer varies between different types of geometry, and depends on how broadly we define ``geometric structure": closed or bordered manifolds, boundary conditions, permissible singularities, and so on.

If we allow cone singularities, we obtain a \emph{cone-manifold structure}. A representation $\rho$ can only make sense as a holonomy of a cone-manifold structure, however, if every interior cone point has a cone angle which is an integer multiple of $2\pi$. This broadening of the notion of geometric
structure is quite a natural one to make: see, e.g., \cite{Leleu,Gallo_Kapovich_Marden}. 

In this series of papers we consider 2-dimensional hyperbolic geometry, but the question can be asked for any geometry. The present paper proves the following theorem about the punctured torus.
\begin{thm}
\label{punctured_torus_theorem}
Let $S$ be a punctured torus, and $\rho: \pi_1(S) \To PSL_2\R$ a homomorphism. The following are equivalent:
\begin{enumerate}
\item
$\rho$ is the holonomy representation of a hyperbolic cone-manifold structure on $S$ with geodesic boundary, except for at most one corner point, and no interior cone points;
\item
$\rho$ is not virtually abelian.
\end{enumerate}
\end{thm}
The corner angles in the cone-manifold structures described by this theorem range over all of $(0, 3\pi)$. Recall a representation is \emph{virtually abelian} if its image contains an abelian subgroup of finite index.

The sequel \cite{Me10MScPaper2} uses this result, and other considerations, to prove results about higher-genus surfaces. The result in the present paper is a complete result, describing exactly which representations are holonomy representations of the desired type. For higher-genus surfaces our results are not so complete, but still may be of interest; they apply only for certain values of the Euler class and with some additional restrictions.

By way of background, we recount some known results for various geometries, hyperbolic and not.
\begin{itemize}
\item 
\textbf{Three-dimensional hyperbolic, Euclidean, spherical geometry; manifold with boundary.} Let $M$ be a 3-manifold with nonempty boundary, and let $(X, \Isom X)$ denote 3-dimensional hyperbolic, spherical or Euclidean geometry. In \cite{Leleu} Leleu proved that a representation $\rho: \pi_1(M) \To \Isom X$ is the holonomy of an $(X, \Isom X)$-structure on $M$ if and only if $\rho$ lifts to the universal covering group $\widetilde{\Isom X}$. No cone
points are required. However the boundary need not be totally geodesic, and will in general be complicated.

\item
\textbf{Three-dimensional hyperbolic geometry; cusped manifold; cone-manifold structures.} Recall that a finite volume complete orientable hyperbolic 3-manifold $M$ is diffeomorphic to the interior of a compact 3-manifold $\bar{M}$ whose boundary consists of tori. There are various ways the question has been attacked in this case. In \cite{Mathews03} I investigated the representation varieties of simple hyperbolic knot complements $S^3
- K$ via the A-polynomial $A_K(x,y)$ of $K$. The A-polynomial encodes information about the restriction of a representation $\pi_1(S^3 - K) \To \Isom^+ \hyp^3$ to the peripheral subgroup (see \cite{CCGLS}). I found in several examples that each branch of the variety defined by $A_K(x,y)$ had a geometric interpretation describing the holonomy of hyperbolic cone-manifold structures on $S^3 - K$. It is known \cite{Hoste-Shanahan,Mathews03} that this is true for twist knots.

In other directions, one can show that there is a well-defined ``volume" associated to a representation $\pi_1(M) \To PSL_2\C$, and that it is maximised at the representation of the unique complete hyperbolic structure --- it is unique by Mostow rigidity. See, e.g. \cite{Thurston_notes,Dunfield99,Francaviglia}.

\item
\textbf{Two-dimensional complex projective geometry; closed surface; cone-manifold structures.} For an oriented closed surface $S$, Gallo--Kapovich--Marden proved in \cite{Gallo_Kapovich_Marden} that a representation $\rho: \pi_1(S) \To PSL_2\C$ is the holonomy of a complex projective cone-manifold structure if and only if $\rho$ is nonelementary. If $\rho$ lifts to a representation into $SL_2\C$, then a complete complex projective structure is possible. Otherwise a single cone point of angle $4\pi$ is sufficient.

\item
\textbf{Two-dimensional hyperbolic geometry; complete hyperbolic structures with totally geodesic boundary.} This question was answered by Goldman \cite{Goldman_thesis}. For a closed surface $S$ with $\chi(S) < 0$, a representation $\rho: \pi_1(S) \To PSL_2\R$ determines an \emph{Euler class} $\E(\rho)$. The Euler class is a 2-dimensional cohomology class on $S$, hence a multiple of the fundamental class. The Euler class may be any multiple of the fundamental class between $\chi(S)$ and $-\chi(S)$, and it parametrises the connected components of the representation space (\cite{Goldman88}). Goldman proved that $\rho$ is the holonomy of a hyperbolic structure on $S$ if and only if the Euler class is $\pm \chi(S)$ times the fundamental cohomology class. If $S$ has boundary, then the same machinery applies, and the same theorem holds, provided that each boundary curve is sent to a non-elliptic isometry. In this case we obtain a \emph{relative} Euler class. In the sequel \cite[sec. 4]{Me10MScPaper2} we discuss these ideas and in fact reprove Goldman's theorem.
\end{itemize}

If $S$ has boundary, then we may require that the boundary be totally geodesic, or piecewise geodesic with a small number of corners. Allowing arbitrarily many corners rapidly trivialises the problem, giving us great freedom to construct a developing map and hence a geometric structures.

We might also allow folding of our hyperbolic structure: allowing the developing map sometimes to preserve and sometimes to reverse orientation, with changes of orientation along geodesic folds. But with more folds allowed there is more freedom to construct a developing map, and we must restrict the number of folds tightly to avoid trivialising the problem; see e.g. \cite{Tan94}. Another way to broaden the question is to relinquish control over the boundary of a surface, not requiring it to be totally or even piecewise geodesic. Then the question is easier, but the boundary may be very complicated. Thus, the type of structure in theorem \ref{punctured_torus_theorem} --- one corner point only --- is quite natural to consider.

\subsection{Structure of this paper}

This paper is organised as follows. 

In section \ref{sec:Preliminaries} we briefly recall some preliminaries regarding geometric structures and cone-manifolds. We develop some results in hyperbolic geometry we shall need. We examine the group $\widetilde{PSL_2\R}$, using the notion of the ``twist" of a hyperbolic isometry at a point; see also \cite{Me10MScPaper0}.

In section \ref{sec:geometry_punctured_tori} we analyse the geometry of punctured tori with hyperbolic cone-manifold structures with one corner point. We show how they can be decomposed into a pentagonal fundamental domain, and conversely give a method for constructing such domains. We simply require a certain pentagon to bound an immersed disc in $\hyp^2$. We establish a relationship between our notion of ``twist" and the corner angle which arises.

In section \ref{sec:Euler_class_representation_space} we examine representation and character varieties. We describe characters of the fundamental group of the punctured torus precisely. Nielsen's theorem shows just how closely algebra and geometry are related. Changes of basis in the fundamental group have a simple description in terms of Markoff moves. We characterise virtually abelian representations in terms of the character variety, and classify reducible representations.

In section \ref{sec:construction_tori} we prove the main theorem \ref{punctured_torus_theorem}, constructing hyperbolic cone-manifold structures on punctured tori with no interior cone points and at most one corner point, for all representations $\rho$ except the virtually abelian ones. We have several cases, corresponding to the values of a single parameter, namely the trace of the holonomy of the loop around the puncture, which is natural in light of Nielsen's theorem. In the most difficult case, we must apply an algorithm to change basis in our fundamental group, using Markoff moves, until we obtain a good geometric arrangement of isometries. 

Finally, in section \ref{rigidity section}, we examine the lack of rigidity in these geometric structures --- one representation can be the holonomy for a continuous family of geometric structures.

\section{Preliminaries}
\label{sec:Preliminaries}

Throughout, let $S$ be an orientable surface. We recall some basic notions. 

\subsection{Geometric structures on manifolds}
\label{fibre bundle introduction section}

Recall that, given a model geometry $X$, with $\Isom X$ acting transitively on $X$, a geometric $(X, \Isom X)$ structure (see \cite{Thurston_book,Thurston_notes} for details) is a metric on $S$ so that every point of $S$ has a neighbourhood isometric to a standard ball neighbourhood in $X$. A geometric structure gives an atlas of coordinate charts to $X$, with transition maps in $\Isom X$, and a developing map $\D: \tilde{S} \to X$, unique up to conjugation by isometries of $X$. Taking a based loop $C$ in $S$, lifting to $\tilde{S}$, considering the developing map, composing transition maps we obtain an isometry $\rho(C) \in \Isom X$, the \emph{holonomy} of $C$. This isometry depends only on the homotopy class of $C$ and describes the action within $X$ as we walk along the developing image of this curve. Thus we obtain the \emph{holonomy map} or \emph{holonomy representation} $\rho: \pi_1(S) \To \Isom X$.

Recall $\pi_1(S)$ acts on $\tilde{S}$ by deck transformations and on $X$ (via $\rho)$ by isometries. This action is equivariant with respect to the
developing map $\D$. Now for $\alpha, \beta \in \pi_1(S)$ we have $T_\alpha \circ T_\beta = T_{\alpha \beta}$ (where $T_\cdot$ are deck transformations), provided that a composition of loops in $\pi_1(S)$ is traversed \emph{left to right} and a composition of functions is (as usual!) applied \emph{right to left}. Then
\[
    \rho(\alpha \beta) \circ \D = \D \circ T_{\alpha \beta} = \D
    \circ T_\alpha \circ T_\beta = \rho(\alpha) \circ \D \circ
    T_\beta = \rho(\alpha) \rho(\beta) \circ \D
\]
so $\rho$ is a homomorphism.

\subsection{Cone-manifolds}
\label{sec:cone-manifolds}

Recall the notion of cone-manifolds (see e.g. \cite{CHK} for details). For our purposes we only need hyperbolic cone-manifolds, and only in dimension 2. In this case a hyperbolic cone-manifold is simply a surface obtained by piecing together geodesic triangles in $\hyp^2$. Points $p$ in the interior of $S$ have neighbourhoods locally isometric to $\hyp^2$, except possibly at some vertices of the triangulation, around which the angles sum to $\theta \neq 2\pi$. Such points are called (interior) \emph{cone points}. The neighbourhood of such a cone point is isometric to a wedge of angle $\theta$ in $\hyp^2$, with sides glued (i.e. a cone). The angle $\theta$ is called the \emph{cone angle} at $p$. Letting $\theta = 2\pi(1+s)$, we call the number $s$ the \emph{order} of the cone point, following \cite{Troyanov}. If $S$ has boundary then this boundary will be piecewise geodesic. There may be vertices on the boundary around which the angles sum to $\theta \neq \pi$. Such a point is called a \emph{corner point} and $\theta$ is the \emph{corner angle}. Letting $\theta = 2 \pi (\frac{1}{2} + s )$, then $s$ is the \emph{order} of the corner point. A corner point has a neighbourhood isometric to a wedge of angle $\theta$ in $\hyp^2$ (without sides glued). A \emph{singular point} is a cone or corner point. The set of singular points is called the \emph{singular locus}. In general the singular locus of an $n$-dimensional cone-manifold is a union of totally geodesic closed simplices of dimension $n-2$. Other points are called \emph{regular points}.

Note a cone or corner angle can be any positive real number --- it can be more than $2\pi$. We will be dealing with many large cone angles. 

Cone points can be considered as ``concentrated curvature''; topology imposes limits on the curvature concentrated in cone and corner angles in a 2-dimensional hyperbolic cone-manifold. Taking a triangulation of $S$ as above, recalling that the area of a triangle with angles $\alpha, \beta, \gamma$ is $\pi - \alpha - \beta - \gamma$, and using Euler's formula, the positivity of area gives a bound on the $s_i$.
\begin{lem}
\label{Gauss-Bonnet}
    Let $S$ be a surface (with or without boundary). A hyperbolic cone-manifold structure on $S$ with cone and corner points having orders $s_i$ satisfies
    \[
        \sum s_i < - \chi(S),
    \]
    and in fact their difference $-\chi(S) - \sum s_i$ is the hyperbolic area of $S$, divided by $2\pi$.
    \qed
\end{lem}

In any hyperbolic cone-manifold, a sufficiently small loop around an interior cone point $v$ is homotopically trivial. Therefore, if a holonomy map is to be well-defined, the corresponding isometry of $\hyp^2$ must be the identity. This can only occur if the cone angle is an integer multiple of $2\pi$, in which case a loop about $v$, under our developing map, winds around some $\D(\tilde{v})$ a number of times but forms a closed loop. However there is no such problem with corner points, which \emph{a priori} may have any corner angle, subject to the bounds discussed above.

We recall some basic properties of curves on hyperbolic cone-manifolds (see \cite{BH} for details). Between any two points of a hyperbolic cone-manifold $S$ there is a geodesic, even though it may not be smooth and may pass through cone points; $S$ is therefore  a \emph{geodesic space}. Amongst them there are shortest geodesics. The distance between two points, defined as the infimum of the lengths of curves between them, makes $S$ into a metric space: $S$ is therefore a \emph{length space}. This distance is achieved by shortest geodesics.

Restricting to dimension $2$, a singular point in a hyperbolic cone surface has a standard neighbourhood isometric to a \emph{hyperbolic open cone} on an arc or circle $M$. The cone $\Cone(M,R)$ of radius $R$ on $M$ is $M \times [0,R)$ with $M \times \{0\}$ collapsed to a vertex. Let $[0,R)$ and $M$ have Riemannian metrics $dr$ and $d\theta$ respectively; then infinitesimal distance on $\Cone(M,R)$ is $ds^2 = dr^2 + \sinh^2 r \; d\theta^2$ (the standard form for hyperbolic distance in polar coordinates). A unit speed geodesic $C(t)$ through the vertex $x$ at $t=0$ has the form $C(t) = (m,-t)$ for $t<0$ and $C(t)=(m', t)$ for $t>0$, for some $m,m' \in M$. Such a curve is a geodesic if and only if $d_M (m,m') \geq \pi$. That is, $C$ is a geodesic if and only if it makes an angle of at least $\pi$ at $x$. (At a regular point the condition that a geodesic must make an angle of $\pi$ is well known!)

It follows that geodesics must avoid cone points with cone angles under $2\pi$. However, there are many geodesics through a cone point $x$ with cone angle over $2\pi$. Thus, unlike the situation at regular points, a geodesic segment with an endpoint at $x$ extends in infinitely many directions: see figure \ref{fig:1}. This argument applies equally if $x$ is an interior cone point or a corner point of $S$.

\begin{figure}[tbh]
\centering
\includegraphics[scale=0.5]{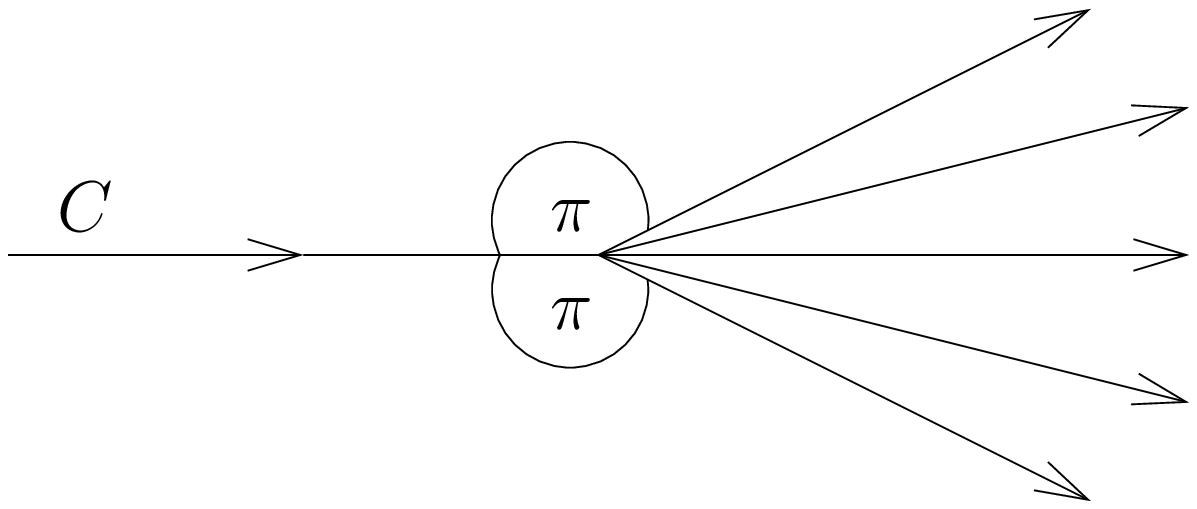}
\caption{Extensions of a geodesic segment at a large angle cone
point.} \label{fig:1}
\end{figure}

\subsection{Hyperbolic isometries}
\label{sec:hyp_geom}

We now make some preliminary considerations in plane hyperbolic geometry. For computations we work in the upper half plane. 

We will need \emph{Fermi coordinates}. Given an oriented line $l$ in $\hyp^2$ and basepoint $q$ on $l$, a point $p \in \hyp^2$ has coordinates $(x,h)$ where $x$ denotes ``distance along $l$" and $h$ denotes ``height above $l$". Precisely, from $p$ we drop a perpendicular to meet $l$ at $p'$. Then $x$ is the signed distance from $q$ to $p'$, and $h$ is the signed length of the perpendicular dropped. In this way the hyperbolic plane is identified with $\R^2$. The distance between $p_1 = (x_1, h_1)$ and $p_2 = (x_2, h_2)$ is then given by (see e.g. \cite{Buser} p. 38):
\begin{eqnarray}
\label{eqn:Fermi}
    \cosh d(p_1, p_2) = \cosh h_1 \; \cosh h_2 \; \cosh (x_2 -
    x_1) - \sinh h_1 \; \sinh h_2.
\end{eqnarray}

\label{compositions of isometries}

We shall need to consider the effect of composing several isometries; and to characterise the geometric arrangement of isometries based on the algebra of matrices in $PSL_2\R$. For the remainder of this section we have some lemmata about commutators $[g,h] = ghg^{-1}h^{-1}$ of isometries.

A proof of the following lemma may be found in \cite{Goldman03}, by computations after conjugating matrices to a simple standard form; or see our more geometric approach using the notion of hyperbolic twisting in \cite{Me10MScPaper0}.
\begin{lem}
\label{axes crossing}
    Let $g,h \in PSL_2\R$. The following are equivalent:
    \begin{enumerate}
        \item
            $g,h$ are hyperbolic and their axes cross;
        \item
            $\Tr[g,h] < 2$.
    \end{enumerate}
\qed
\end{lem}

Note that although $g,h$ are only defined up to sign in $SL_2\R$, the commutator \emph{is} a well-defined element of $SL_2\R$, and has a well-defined trace. (In fact it is well-defined in the universal cover $\widetilde{PSL_2\R}$.) Denote by $a_\alpha, r_\alpha$ the attractive and repulsive fixed points of a hyperbolic isometry $\alpha$.

\begin{lem}
\label{hyperbolic commutator}
    Suppose $g,h \in PSL_2\R$ and $\Tr [g,h] < -2$, so that $g,h$
    are hyperbolic and their axes intersect, and $[g,h]$ is also
    hyperbolic.
    Then $\Axis [g,h]$ does not intersect the
    axis of $g$ or $h$. The fixed points of $[g,h]$ lie on the
    segment of the circle at infinity between $a_g$ and $a_h$: $a_{[g,h]}$ is closer to $a_g$, and $r_{[g,h]}$ is closer to $a_h$.
\end{lem}

The lemma states that the order of the fixed points on the circle at infinity is
\[
    a_h, \quad r_{[g,h]}, \quad a_{[g,h]}, \quad a_g, \quad r_h, \quad r_g
\]
up to cyclic permutation and reflection. See figure \ref{fig:2}. Denote by $R_l$ the reflection in the line $l$.

\begin{figure}[tbh]
 \centering
\includegraphics[scale=0.4, angle=0]{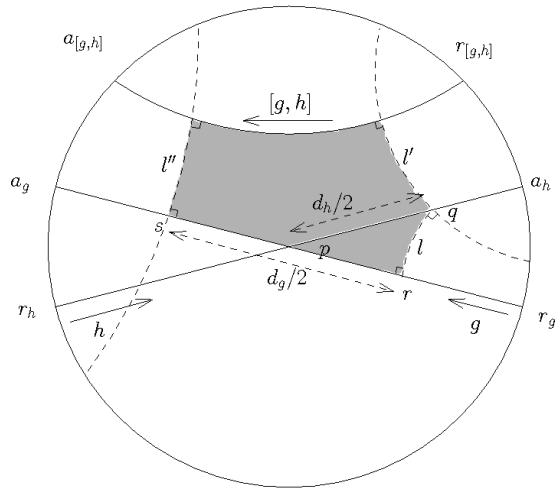}
\caption{The location of $\Axis [g,h]$ in the case $\Tr[g,h]<-2$}
\label{fig:2}
\end{figure}

\begin{proof}
We use an elegant argument of Matelski in \cite{Matelski}; there is also a proof by computation. Let $d_g, d_h$ denote the translation distance of $g,h$. Let $p \in \hyp^2$ be the point of intersection of axes of $g$ and $h$, and let $e \in PSL_2\R$ be a half-turn about $p$. Thus we have $ege = g^{-1}$ and $ehe = h^{-1}$. Consider $he$: this preserves $\Axis h$ but reverses its sense; it is a half-turn about a point $q \in \Axis h$, where $q$ lies on the same side of $p$ as $a_h$, at a distance $d_h/2$ from $p$.

Now consider $ghe$. We have $(ghe)^2 = gh(ege)(ehe) = ghg^{-1}h^{-1} = [g,h]$, which is hyperbolic. Thus $ghe$ is hyperbolic and has the same axis as $[g,h]$. So we only need show that the axis of $ghe$ lies in the desired position.

Let $l$ be the perpendicular from $q$ to $\Axis g$, and $r$ its foot. Let $l'$ be the line through $q$ perpendicular to $l$. Let $s$ be a point along $\Axis g$ on the same side of $r$ as $a_g$, and distance $d_g/2$ from $r$. Let $l''$ be the line through $s$ perpendicular to $\Axis g$. So $R_l R_{l'} = he$; the composition of two reflections in two perpendicular lines meeting at $q$ is a half-turn about $q$. And $R_{l''} R_l = g$; the composition of two reflections in lines perpendicular to $\Axis g$ being $d_g/2$ apart is a translation along $\Axis g$ by $d_g$. Thus $ghe = R_{l''} R_l^2 R_{l'} = R_{l''} R_{l'}$. So $l'$ and $l''$ do not intersect, even at infinity (otherwise $ghe$ would be elliptic or parabolic), and $\Axis ghe = \Axis [g,h]$ is the common perpendicular of $l'$ and $l''$.

Now $l, l', \Axis[g,h], l'', \Axis g$ form a right-angled pentagon, as shown in figure \ref{fig:2}. Thus $\Axis[g,h]$ must lie on the same side of $\Axis g$ as $a_h$, and on the same side of $\Axis h$ as $a_g$; and $[g,h]$ translates in the desired direction.
\qed
\end{proof}

Repeating the same argument when $[g,h]$ is parabolic, or elliptic, we have similar results.
\begin{lem}
\label{parabolic_commutator}
    Suppose $g,h \in PSL_2\R$ and $\Tr [g,h] = -2$, so that $g,h$
    are hyperbolic and their axes intersect, and $[g,h]$ is parabolic.
    Then $\Fix [g,h]$ lies on the
    segment of the circle at infinity between $a_g$ and $a_h$. The
    sense of the rotation is as shown in figure \ref{fig:3} (left).
\qed
\end{lem}

\begin{lem}
\label{elliptic_commutator}
    Suppose $\Tr[g,h] \in (-2,2)$. Then $\Fix [g,h]$ lies in the region of $\hyp^2$ determined by $\Axis(g)$,
    $\Axis(h)$ which is
    bounded by the arc on the circle at infinity between $a_g$ and
    $a_h$. See figure \ref{fig:3} (right).
\qed
\end{lem}

\begin{figure}[tbh]
\begin{center}
$\begin{array}{c}
\includegraphics[scale=0.3]{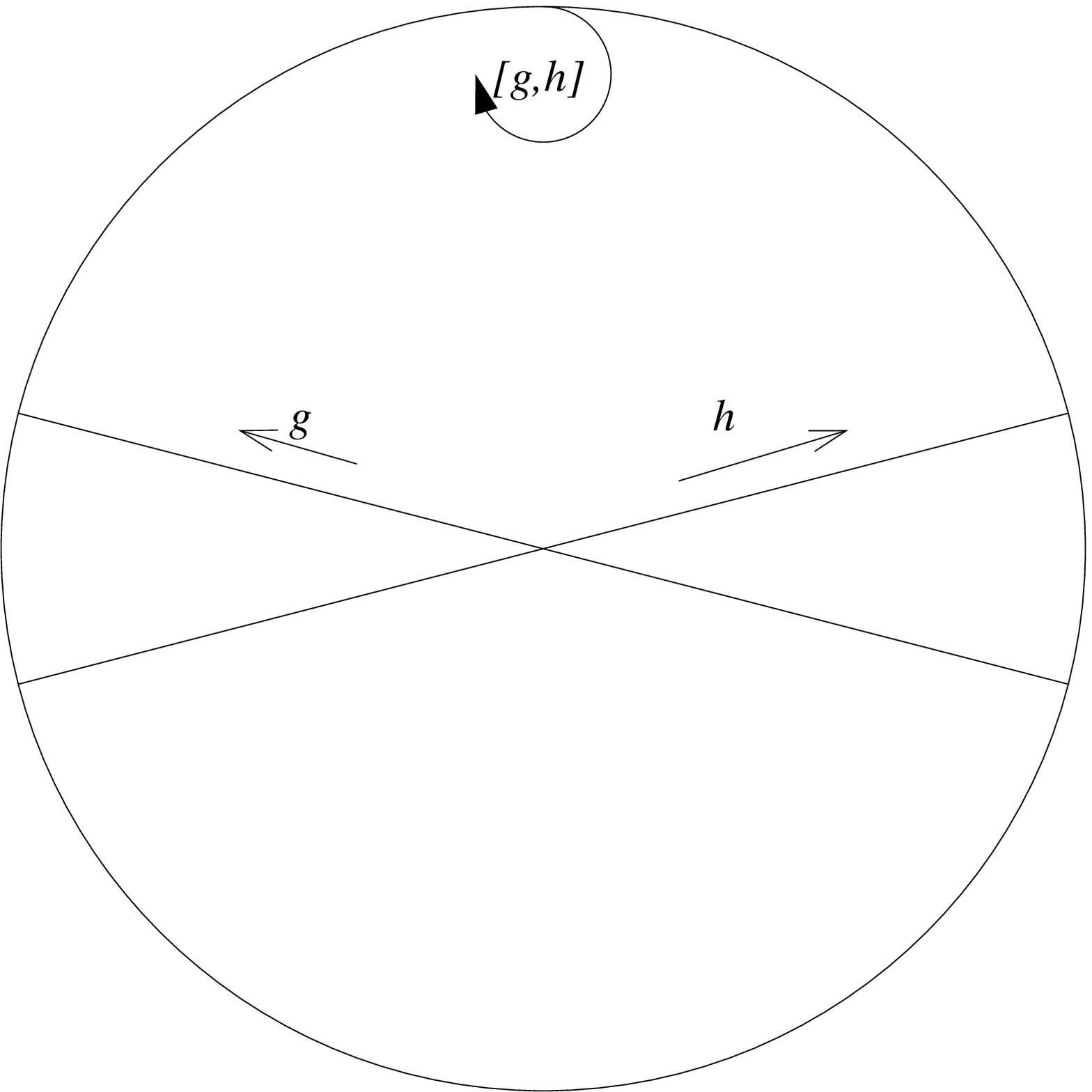}
\end{array}$
\hspace{2cm}
$\begin{array}{c}
\includegraphics[scale=0.3]{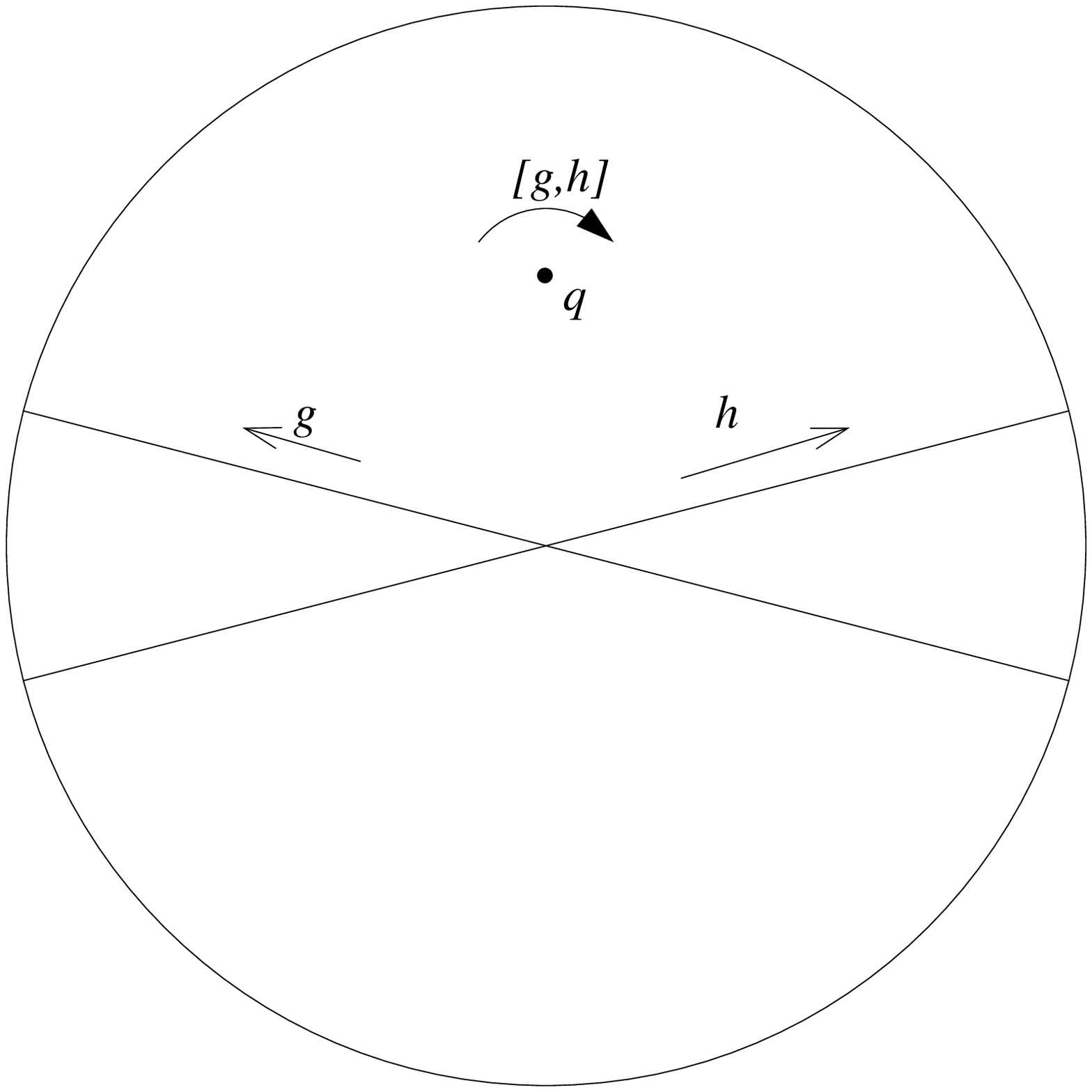}
\end{array}$
\caption{The location of $[g,h]$ in the cases $\Tr[g,h] = -2$ (left) and $\Tr[g,h] \in (-2,2)$ (right).} \label{fig:3}
\end{center}
\end{figure}

\subsection{$PSL_2\R$ and $\widetilde{PSL_2\R}$}

Fixing an arbitrary unit tangent vector $u_0$ at an arbitrary basepoint $y_0$ in $\hyp^2$, we see that a hyperbolic isometry is uniquely determined by the image of this unit tangent vector; thus we may identify the unit tangent bundle $UT\hyp^2$ with $PSL_2\R$. The universal cover of $PSL_2\R$ is denoted $\widetilde{PSL_2\R}$.

We recall some properties of $\widetilde{PSL_2\R}$; see also \cite{Goldman_thesis,Goldman88,Me10MScPaper0} for details. Topologically $PSL_2\R \cong \R^2 \times S^1$, and $\pi_1(PSL_2\R) \cong \Z$. An element $\tilde{x} \in \widetilde{PSL_2\R}$ is hyperbolic, elliptic or parabolic accordingly as is its projection in $PSL_2\R$; $\widetilde{PSL_2\R}$ can be regarded as the hyperbolic plane, with a line attached to each point, covering the circle of unit tangent vectors at that point.

Elements of $\widetilde{PSL_2\R}$ can be considered as homotopy classes of paths in $UT\hyp^2$; since the basepoint is arbitrary, every path $c: [0,1] \To UT\hyp^2$ determines an element of $\widetilde{PSL_2\R}$; the projection of $c$ to $PSL_2\R$ is the isometry sending $c(0)$ to $c(1)$.

The lifts of $1 \in PSL_2\R$ form an infinite cyclic group $\{\z^n :  n \in \Z \}$, where $\z$ is a rotation by $2\pi$. This $\z$ generates the centre of $\widetilde{PSL_2\R}$. The lifts of a general $\alpha \in  PSL_2\R$ differ by powers of $\z$ and represent paths in $UT\hyp^2$ between the same start and end
tangent vectors.

Some elements of $PSL_2\R$ have ``simplest'' lifts to $\widetilde{PSL_2\R}$. The identity in $\widetilde{PSL_2\R}$ is the simplest lift of the identity in $PSL_2\R$. For hyperbolic $\alpha \in PSL_2\R$ there exists a unique homomorphism $c: \R \To PSL_2\R$ with $c(1) = \alpha$; this gives a path which is a simplest lift. The same applies to parabolics. For elliptic $\alpha$ however there are infinitely many such homomorphisms. Suppose $\alpha$ rotates by angle $\theta$ (mod $2\pi$); then the lifts of $\alpha$ are rotations by angles $\theta + 2\pi \Z$. There are two simplest lifts of $\alpha$, one anticlockwise and one clockwise, with rotation angle lying in $(0, 2\pi)$ and $(-2\pi,0)$ respectively.

Simplest lifts of hyperbolics and parabolics are denoted $\Hyp_0, \Par_0$; then $\Hyp_n = \z^n \Hyp_0$ and $\Par_n = \z^n \Par_0$; the hyperbolic (resp. parabolic) elements of $\widetilde{PSL_2\R}$ are the disjoint union of the $\Hyp_n$ (resp. $\Par_n$). We further distinguish $\Par_n^+$ and $\Par_n^-$, the rotations about points at infinity whose projections to $PSL_2\R$ are anticlockwise and clockwise respectively. For elliptics, simplest anticlockwise and clockwise lifts are defined to lie in $\Ell_1$ and $\Ell_{-1}$ respectively. For $n>0$ let $\Ell_n = \z^{n-1} \Ell_1$ and $\Ell_{-n} = \z^{-n+1} \Ell_{-1}$. (So $\Ell_0$ is not defined and $\z \Ell_{-1} = \Ell_1$.) For $n>0$ (resp. $n<0$), $\Ell_n$ consists of all rotations through angles between $2\pi(n-1)$ and $2\pi n$ (resp. between $2\pi n$ and $2\pi (n+1)$).

We have a schematic diagram of $\widetilde{PSL_2\R}$ in figure \ref{fig:4}.
\begin{figure}[tbh]
\centering
\includegraphics[scale=0.6]{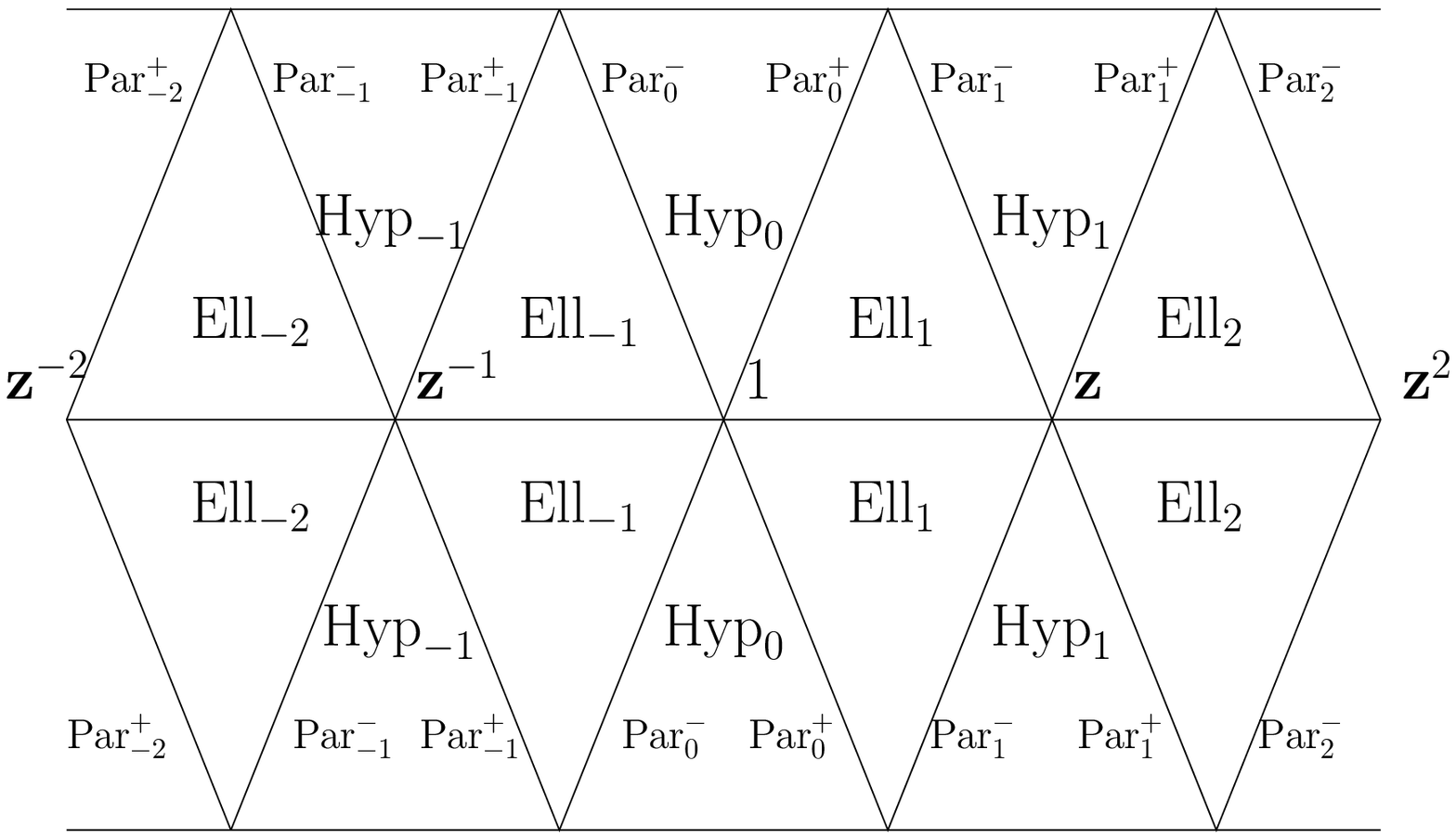}
\caption[Schematic diagram of $\widetilde{PSL_2\R}$]{Schematic
diagram of $\widetilde{PSL_2\R}$.}
\label{fig:4}
\end{figure}

Since lifts of $\alpha \in PSL_2\R$ differ by powers of $\z$, the following lemma is clear.
\begin{lem}
\label{commutator_lift}
    Let $\alpha, \beta \in PSL_2\R$. Then $[\alpha,\beta]$
    has a
    well-defined lift to $\widetilde{PSL_2\R}$. That is, any two
    sets of lifts $\tilde{\alpha}_1, \tilde{\beta}_1$ and $\tilde{\alpha}_2,
    \tilde{\beta}_2$ satisfy $[\tilde{\alpha}_1, \tilde{\beta}_1] =
    [\tilde{\alpha}_2, \tilde{\beta}_2]$.
\qed
\end{lem}

\subsection{Derivatives of Isometries of $\hyp^2$}

\label{sec:deriv_isoms}

An isometry $\alpha \in PSL_2\R$ has a derivative $D\alpha$ which we may consider as a map $UT\hyp^2 \To UT\hyp^2$. We define a notion of the \emph{twist} of an $\tilde{\alpha} \in \widetilde{PSL_2\R}$ at a point $y \in \hyp^2$; see \cite{Me10MScPaper0} for details.

First, given a unit tangent vector field $\V$ along a smooth curve $c: [0,1] \To \hyp^2$, we define the \emph{twist of $\V$ along $c$}. At $c(t)$ we consider the angle $\theta(t)$ (measured anticlockwise) from the velocity vector, to $\V(t)$. There are many choices for $\theta(0)$ (differing by $2 \pi \Z$) but choosing $\theta(0)$ arbitrarily determines continuous $\theta$ completely; $\theta(1) - \theta(0)$ is independent of this choice, and is the twist of $\V$ along $c$.

Now we define the \emph{twist of $\tilde{\alpha} \in \widetilde{PSL_2\R}$ at $y \in \hyp^2$}, denoted $\Twist(\tilde{\alpha},y)$. Let $\tilde{\alpha}$ project to $\alpha \in PSL_2\R$. Let $c: [0,1] \To \hyp^2$ be a constant speed (possibly $0$) geodesic from  $y$ to $\alpha(y)$. There is a vector field $\V: [0,1] \To UT\hyp^2$ along $c$ which lies in the homotopy class of $\tilde{\alpha}$. Then $\Twist(\tilde{\alpha},y)$ is the twist of $\V$ along $c$; this does not depend on the choice of $\V$. For $\alpha \in PSL_2\R$, $\Twist(\alpha, y)$ is defined similarly with angles modulo $2\pi$.

Thus, $\Twist(\tilde{\alpha},y)$ describes how the tangent vector at $y$ is moved by $\tilde{\alpha}$, compared to parallel translation along the geodesic from $y$ to $\tilde{\alpha}(y)$. 

In \cite{Me10MScPaper0} we prove various properties of the twist; one can easily verify the following.
\begin{itemize}
\item 
For hyperbolic $\tilde{\alpha} \in \Hyp_0$, $\Twist(\tilde{\alpha},y) = 0$ for $y \in \Axis \alpha$, and for general $y \in \hyp^2$, $\Twist(\tilde{\alpha},y) \in (-\pi,\pi)$. The twist is constant along curves of constant distance $h$ from $\Axis \alpha$. For each $\theta \in (-\pi,\pi)$ there is precisely one $h$ for which the curve at distance $h$ is the locus of points $y$ with $\Twist(\tilde{\alpha},y) = \theta$.

\item
For parabolic $\tilde{\alpha} \in \Par_0$, $\Twist(\tilde{\alpha},y)$ is constant along horocycles about $\Fix \alpha$. If $\tilde{\alpha} \in \Par_0^+$ (resp. $\Par_0^-$) then $\Twist(\tilde{\alpha},y) \in (0,\pi)$ (resp. $(-\pi,0)$). On horocycles close to $\Fix \alpha$, the twist is close to $0$. For each $\theta \in (0, \pi)$ (resp. $(-\pi, 0)$) there is precisely one horocycle which is the locus of points $y$ with $\Twist(\tilde{\alpha},y) = \theta$.

\item
For elliptic $\tilde{\alpha}$, $\Twist(\tilde{\alpha},y)$ is constant along hyperbolic circles centred at $\Fix \alpha$. Take $\tilde{\alpha} \in \Ell_1$ for convenience, so $\tilde{\alpha}$ rotates by angle $\psi \in (0, 2\pi)$. So $\Twist(\tilde{\alpha}, \Fix \alpha) = \psi$. If $\psi \in (0,\pi)$ then $\Twist(\tilde{\alpha},y)$ always lies in $[\psi, \pi)$; for each $\theta \in (\psi, \pi)$ there is precisely one hyperbolic circle centred at $\Fix \alpha$ which is the locus of $y$ with $\Twist(\tilde{\alpha},y) = \theta$. If $\psi = \pi$ then $\alpha$ is a half turn and $\Twist(\tilde{\alpha}, y) = \pi$ for all $y$. If $\psi \in (\pi, 2\pi)$ then $\Twist(\tilde{\alpha},y)$ always lies in $(\pi,\psi]$ and for each $\theta \in (\psi, \pi)$ there is precisely one hyperbolic circle centred at $\Fix \alpha$ which is the locus of $y$ with $\Twist(\tilde{\alpha},y) = \theta$.
\end{itemize}

\begin{prop}\
\label{twist_bounds}
    \begin{align*}
        \Twist(\Hyp_n, \hyp^2) &= \left( \left( 2n-1 \right) \pi,
        \left( 2n + 1 \right) \pi \right) \\
        \Twist(\Par_n, \hyp^2) &= \left( \left( 2n-1 \right) \pi,
        \left( 2n + 1 \right) \pi \right) \\
        \Twist(\Ell_n, \hyp^2) &= \left\{ \begin{array}{ll} \Bigl( (2n-2)\pi, 2n\pi \Bigr) & \text{
        for $n>0$} \\
        \Bigl( -2|n|\pi, (-2|n|+1)\pi \Bigr) &
        \text{ for $n<0$} \end{array} \right.
    \end{align*}\
    \qed
\end{prop}

\subsection{Traces and commutators in $\widetilde{PSL_2\R}$}

Note that $\widetilde{PSL_2\R}$ covers $SL_2\R$, so there is a well-defined trace on $\widetilde{PSL_2\R}$. For all elliptic regions, the trace lies in $(-2,2)$; in the various other regions of $\widetilde{PSL_2\R}$ the value of the trace follows from considering the regions in figure \ref{fig:4}; see \cite{Me10MScPaper0}.
\begin{lem}
\label{trace_regions}
    \begin{align*}
        \Tr \left( \z^n \right) &= (-1)^n \cdot 2 \\
        \Tr \left( \Par_n \right) &= (-1)^n \cdot 2 \\
        \Tr \left( \Hyp_n \right) &= \left\{ \begin{array}{ll} (2,
        \infty) & \text{$n$ even} \\ (-\infty, -2) & \text{$n$ odd.}
        \end{array} \right.
    \end{align*} \
\qed
\end{lem}

We now consider commutators in $PSL_2\R$. The following theorem is well-known (e.g. \cite{Milnor,Wood,Eisenbud-Hirsch-Neumann,Goldman88}; we also give proofs in \cite{Me10MScPaper0} and \cite[sections 3.5--3.7]{Mathews05}).

\begin{thm}
\label{commutator_regions}
    If $g,h \in PSL_2\R$ then (noting $[g,h]$ is well-defined in
    $\widetilde{PSL_2\R})$
    \[
        [g,h] \in \{1\} \cup \left( \bigcup_{n=-1}^1 \Hyp_n \cup
        \Ell_n \right) \cup \Par_0 \cup \Par_{-1}^+ \cup \Par_1^-.
    \]
\qed
\end{thm}
(We take $\Ell_0 = \emptyset$ for convenience.)

\begin{figure}[tbh]
\centering
\includegraphics[scale=0.6]{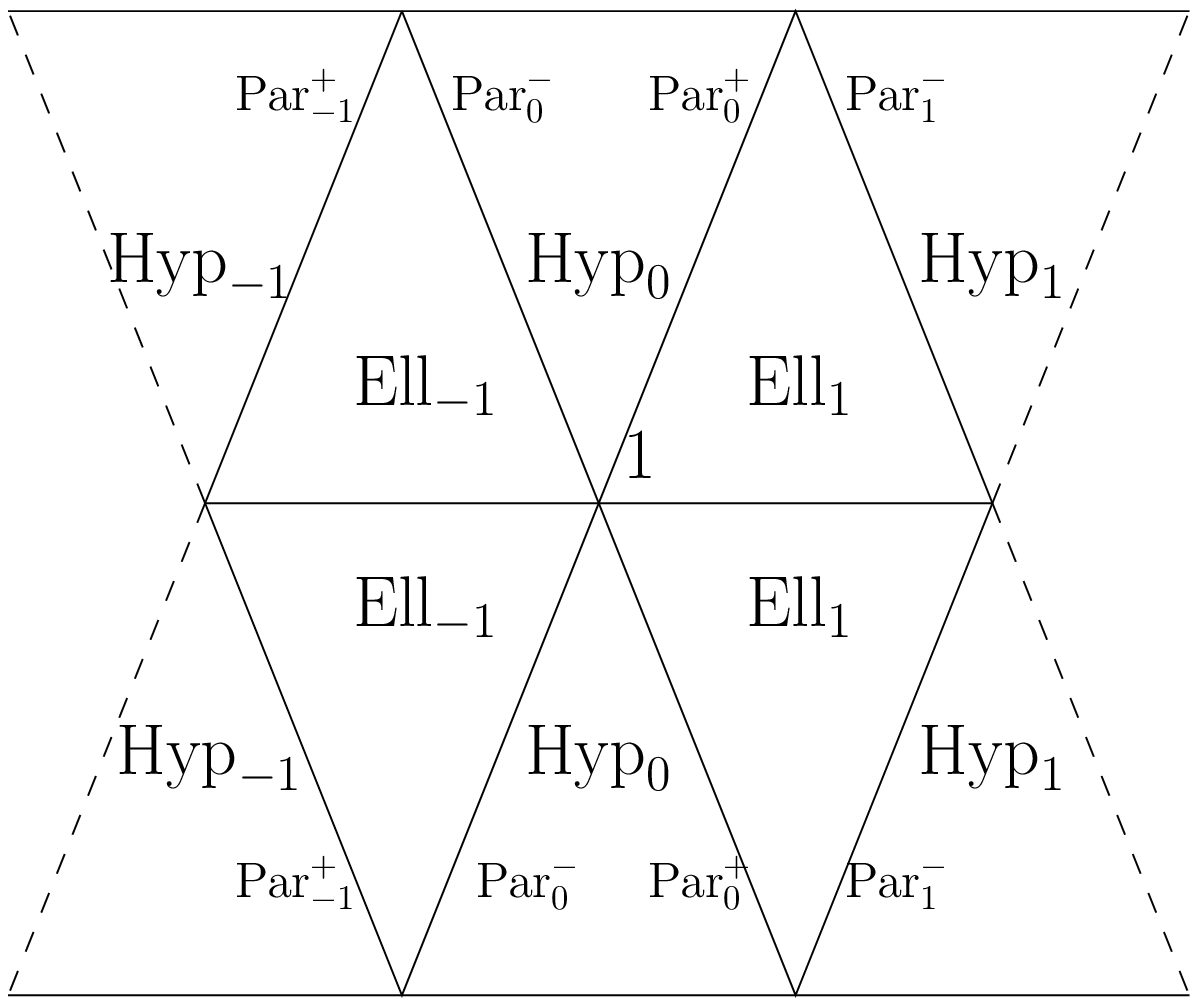}
\caption{Possible commutators in $\widetilde{PSL_2\R}$.}
\label{fig:5}
\end{figure}

\begin{cor}
\label{commutator_trace_info}
    If $g,h \in PSL_2\R$ then
    \begin{enumerate}
        \item
            $\Tr[g,h] > 2$ implies $[g,h] \in \Hyp_0$;
        \item
            $\Tr[g,h] = 2$ implies $[g,h] \in \{1\} \cup \Par_0$;
        \item
            $\Tr[g,h] \in (-2,2)$ implies $[g,h] \in
            \Ell_{-1} \cup \Ell_1$;
        \item
            $\Tr[g,h] = -2$ implies $[g,h] \in
            \Par_{-1}^+ \cup \Par_1^-$;
        \item
            $\Tr[g,h] < -2$ implies $[g,h] \in \Hyp_{-1} \cup
            \Hyp_1$.
    \end{enumerate}\
\qed
\end{cor}

\section{The Geometry of Punctured Tori}
\label{sec:geometry_punctured_tori}

Let $S$ denote a punctured torus with a hyperbolic cone-manifold structure; we saw above (lemma \ref{Gauss-Bonnet}) that $\sum s_i < -\chi(S) = 1$. We are interested in interior cone points with angles which are multiples of $2\pi$, i.e. $s_i \in \N$; but there cannot be any such cone points $s_i \geq 1$. Hence we only consider corner points. We allow $S$ to have at most one corner point $p_0$, with corner angle $\theta$; $s_i < 1$ implies $\theta \in (0, 3\pi)$.

\subsection{Pentagon decomposition}
\label{pentagon decomposition}

We demonstrate a standard decomposition of a punctured torus as described above into a hyperbolic geodesic pentagon. 

We need to be careful with the behaviour of geodesics at corner points. We have mentioned that between any $p,q$ in a hyperbolic cone-manifold there is a shortest curve joining them, which is a geodesic (section \ref{sec:cone-manifolds}). Such a shortest geodesic $C$ joining two points $p$ and $q$ must be simple (i.e. non-self-intersecting), and if $C$ intersects the boundary $\partial S$ then $C \cap \partial S$ is a disjoint union of closed segments whose endpoints are corner points with corner angles greater than $\pi$, or $p$ or $q$.

We also need the following lemma. Recall that a curve $C$ is
\emph{boundary-parallel} to a boundary component $A$ if $C$ can be
homotoped to lie entirely on $A$. In particular a null-homotopic
curve is boundary-parallel to $A$.

\begin{lem}
\label{simple non-boundary-parallel geodesic}
    Let $S$ be a 2-dimensional hyperbolic cone-manifold with no interior cone points and connected piecewise geodesic boundary with exactly one corner point $q$. Then there is a shortest closed curve $C$ based at $q$ which is not boundary-parallel. The curve $C$ is a geodesic arc, intersects no singular points in its interior, and is simple.
\end{lem}

\begin{proof}
    Since a sufficiently small neighbourhood of $q$ is
    contractible, the quantity
    \[
        d = \inf \left\{ l(\gamma) : \text{$\gamma$ a loop based at $q$
        and not boundary-parallel} \right\}
    \]
    is positive. Thus we find curves $\gamma_n$ based at $q$, not boundary-parallel, such that $l(\gamma_n) < d + 1/n$. We can apply the Arzel\`{a}-Ascoli theorem to find a subsequence of $\gamma_n$ converging uniformly to a curve $C$, based at $q$, with $l(C)=d$, homotopic to $\gamma_n$ for $n$ sufficiently large, hence not boundary-parallel. (See e.g. \cite[prop I.3.16]{BH}.)

    Since $q$ is the only singular point of $S$, $C$ can only consist of geodesic arcs from $q$ to $q$. If any arcs are boundary parallel, then $C$ can be shortened, a contradiction; so every arc is not boundary parallel. If there is more than one such arc, again $C$ can be shortened; so $C$ is a single geodesic arc and intersects $\partial S$ only at $q$ at its endpoints.

    Suppose $C$ is not simple. Then $C$ intersects itself at some point $y$ in the interior of $S$. Denote the three segments $q \to y \to y \to q$ by $\alpha, \beta, \gamma$ respectively, so $l(C) = l(\alpha) + l(\beta) + l(\gamma)$. The intersection at $y$ must be transverse: if not, the geodesic segments would coincide.

    Now $\alpha.\gamma$ is boundary parallel, else we contradict the minimality of $C$. Thus $\alpha.\beta.\alpha^{-1}$  is not boundary parallel, and the free loop $\beta$ is not boundary parallel either. Similarly, $\gamma^{-1}.\beta.\gamma$ is not boundary parallel. The minimality of $C$ then implies that $l(\alpha) + l(\beta) + l(\gamma) \leq 2 l(\alpha) + l(\beta)$ and $l(\alpha) + l(\beta) + l(\gamma) \leq l(\beta) + 2 l(\gamma)$, hence $l(\gamma) \leq l(\alpha) \leq l(\gamma)$, so $l(\alpha) = l(\gamma)$. But then $\alpha.\beta.\alpha^{-1}$ has the same length as $C$, is not boundary
    parallel, but is not a geodesic. Thus $C$ can be shortened, contradicting the minimality of $C$.
\qed
\end{proof}

Now we obtain our decomposition of a punctured torus.
\begin{prop}
    Let $S$ be a punctured torus with a hyperbolic cone-manifold
    structure with no interior cone points and at most one corner point
    $q$ with corner angle $\theta$ (let $\theta = \pi$ if $q$ is a regular
    point).
    There exist two geodesic arcs $G, H$ on $S$ based at $q$, intersecting only at $q$, such
    that cutting along $G$ and $H$ produces a topological disc which
    is isometric to an immersed disc in $\hyp^2$ bounded by a geodesic pentagon.
\end{prop}

\begin{proof}
Let $G$ denote a shortest closed curve through $q$ which is not boundary-parallel, guaranteed by lemma \ref{simple non-boundary-parallel geodesic}, which shows that $G$ is a geodesic arc and is simple. We cut along $G$, forming a cylinder with a hyperbolic cone-manifold structure. The two boundary components $\partial_1, \partial_2$ are piecewise geodesic. There is one corner point $q_1$ on $\partial_1$, and two corner points $q_2, q_3$ on $\partial_2$. Gluing $\partial_1$ to one of the two geodesic segments of $\partial_2$ recovers the initial surface.

Now consider the shortest curve $\gamma$ from $q_1$ to $q_2$. This curve is piecewise geodesic, with possible corners at $q_1, q_2, q_3$. It cannot pass through $q_1$ in its interior, by minimality. Nor can it pass through $q_2$ in its interior. If it consists of one geodesic segment from $q_1$ to $q_2$, then we let this curve be $H$. Otherwise $\gamma$ passes through $q_3$ on the way to $q_2$; in this case we take $H$ to be the
initial segment from $q_1$ to $q_3$.

\begin{figure}[tbh]
\begin{center}
$\begin{array}{c}
\includegraphics[scale=0.5]{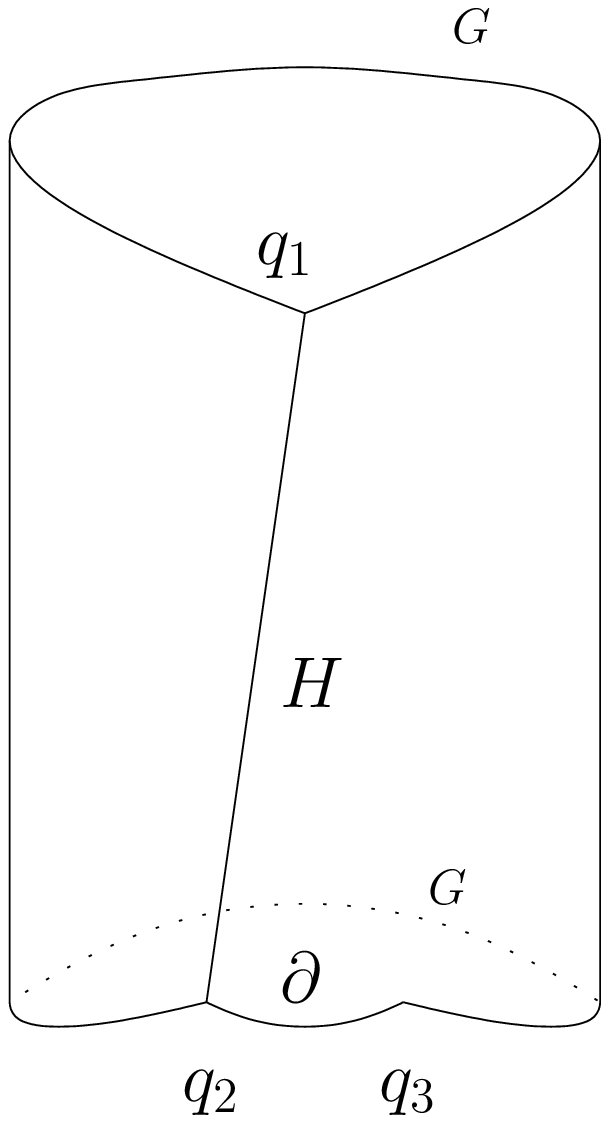}
\end{array}$
\hspace{2 cm}
$\begin{array}{c}
\includegraphics[scale=0.3]{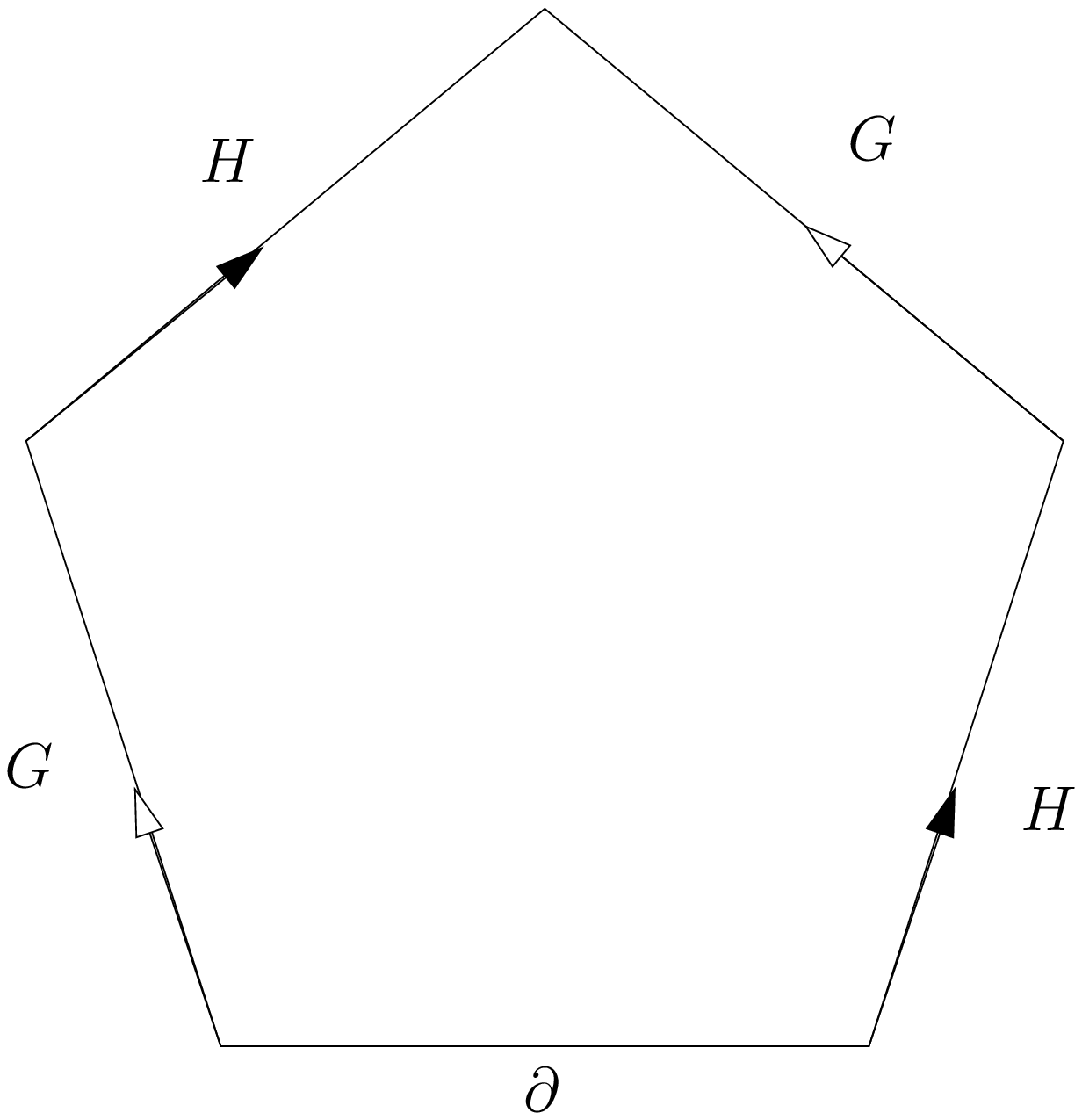}
\end{array}$
\caption{Cutting along $H$; the pentagon $\Pent$ bounds an immersed open disc}
\end{center}
\end{figure}

Thus we obtain a geodesic $H$ on $S$ which intersects $G$ only at $q$. Cutting along $G, H$ reduces $S$ to a topological disc. Since $G,H$ are geodesic arcs, the developing map of a lift of this topological disc shows that the obtained surface is isometric with an immersed open disc in $\hyp^2$ bounded by a geodesic pentagon.
\qed
\end{proof}

Consider the pentagon $\Pent$ obtained by this procedure. Two pairs of sides are identified, which correspond respectively to the curves $G$ and $H$. The
sum of the interior angles of the pentagon is equal to the corner angle $\theta$ at $q$. Furthermore, $\{G,H\}$ forms a free basis for $\pi_1(S,q)$. In $\pi_1(S,q)$, the boundary of $S$ is the commutator $[G,H]$. Note that $\Pent$ need not be a simple pentagon, if $\theta$ is large: see figure \ref{fig:7}.

\begin{figure}[tbh]
\centering
\includegraphics[scale=0.5]{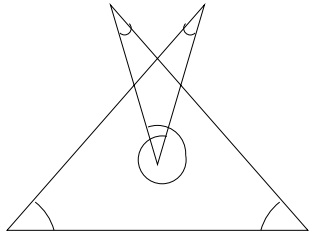}
\caption[$\Pent$ need not be a simple pentagon.]{$\Pent$ need not be
a simple pentagon, but $\Pent$ bounds an immersed
open disc.} \label{fig:7}
\end{figure}

The universal cover $\tilde{S}$ can be considered as a tessellation by copies of this pentagon according to the edge pairings. The developing map of the cone-manifold structure on $S$ is a (generally overlapping) tessellation by isometric copies of the pentagonal fundamental domain $\Pent$ in $\hyp^2$.

Let $\rho: \pi_1(S, q) \To PSL_2\R$ be the holonomy map. Choose a basepoint $\tilde{q}$ lifting $q$ in $\tilde{S}$, and its developing image $\bar{q} \in \hyp^2$. Let $\rho(G) = g$, $\rho(H) = h$. Then with $\bar{q}$ and $\Pent$ as shown in figure \ref{fig:8}, $g$ and $h$ identify pairs of sides in $\Pent$ as shown. Labelling one of the vertices $p = h^{-1} g^{-1} \bar{q}$ as shown, we can describe the other vertices of $\Pent$ as the images of $q$ under various combinations of $g$ and $h$. Thus a hyperbolic cone-manifold structure on $S$ with no interior cone points and at most one corner point gives rise to a basis $G,H$ of the fundamental group and a holonomy representation $\rho$ such that this pentagon is the boundary of an immersed open disc forming a fundamental domain.

\begin{figure}[tbh]
\centering
\includegraphics[scale=0.5]{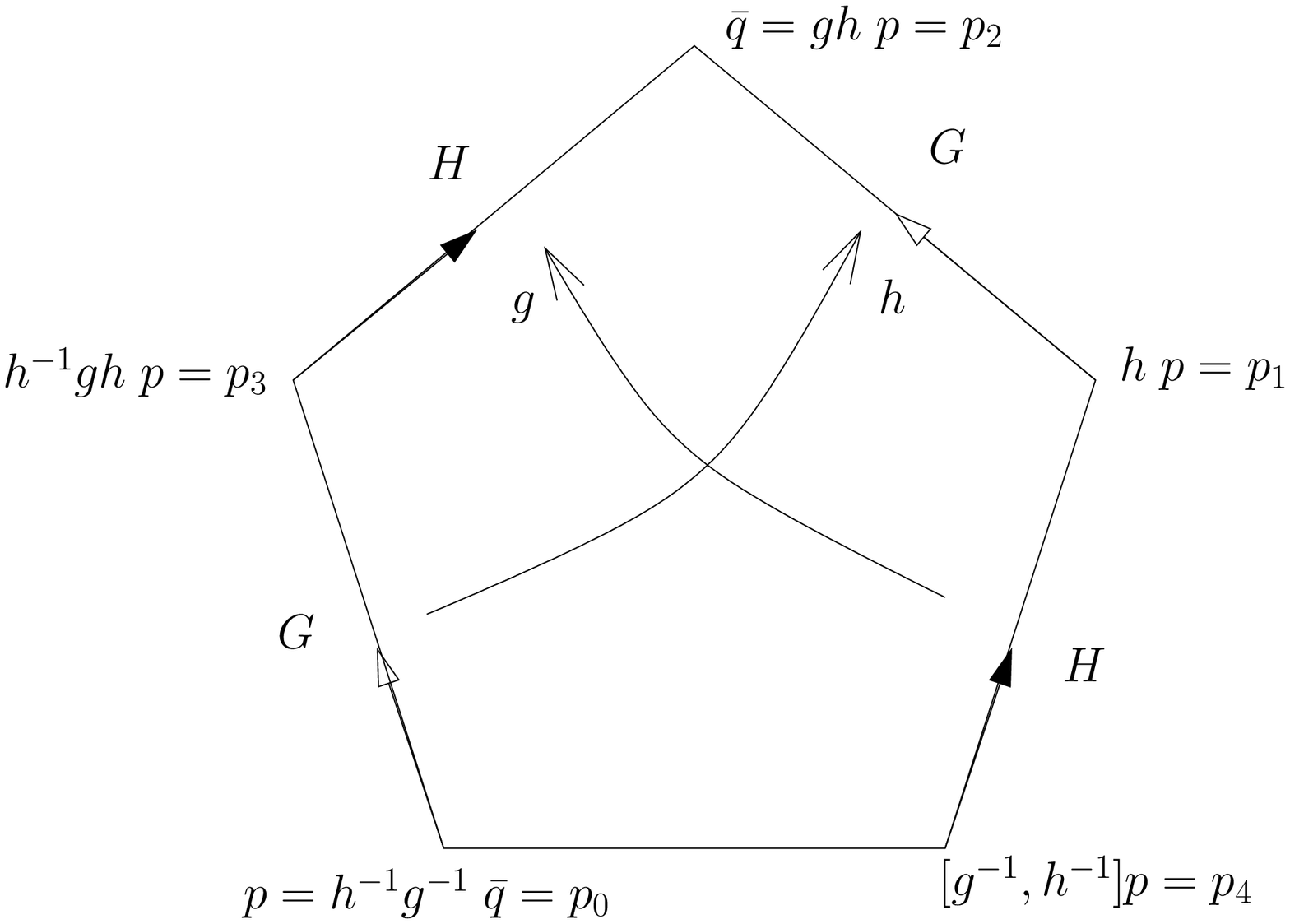}
\caption{Edge identifications in $\Pent$} \label{fig:8}
\end{figure}

Conversely, suppose we have a representation $\rho: \pi_1(S, q)
\To PSL_2\R$ and we can find a basis $G,H$ of $\pi_1(S, q)$ and a
point $p \in \hyp^2$ such that the pentagon described above is
non-degenerate and bounds an immersed open disc. Then it is clear
that this pentagon is a fundamental domain of a developing map for a
hyperbolic cone-manifold structure on $S$ with no interior cone
points and at most one corner point, with holonomy $\rho$. The rest
of the developing map is obtained by extending equivariantly. We record this fact.
\begin{defn}
\label{defn:pentagon}
    Let $g,h \in PSL_2\R$ and $p \in \hyp^2$. Then the geodesic
    pentagon in $\hyp^2$ obtained by joining the segments
\[
        p \To g^{-1} h^{-1} gh p \To hp \To g h p \To h^{-1} gh p \To p
\]
    is called \emph{the pentagon generated by $g,h$ at $p$} and is
    denoted $\Pent(g,h;p)$.
\end{defn}

\begin{lem}
\label{construction_lemma}
    Let $\rho: \pi_1(S, q) \To PSL_2\R$ be a representation. The
    representation $\rho$ is the holonomy of a hyperbolic cone
    manifold structure on $S$ with no interior cone points
    and at most one corner point if and only if there exist a
    free basis $G,H$ of $\pi_1(S, q)$ and a point $p \in \hyp^2$ such
    that $\Pent(g,h;p)$ is a non-degenerate pentagon bounding an immersed open disc in
    $\hyp^2$.
    \qed
\end{lem}

Despite its simplicity, lemma \ref{construction_lemma} will be
crucial in constructing geometric structures with prescribed
holonomy.

\subsection{Twisting and the corner angle}

We now note a relationship between the twisting involved in a holonomy representation for $S$, and the corner angle obtained. Denote the vertices of $\Pent(g,h;p)$ as 
\[
    p_0 = p, \quad p_1 = h p, \quad p_2 = g h p, \quad p_3 = h^{-1} g h p, \quad p_4 = g^{-1}h^{-1}gh p.
\]
(Note that the $p_i$ are \emph{not} labelled in cyclic order around the pentagon.) Let the corresponding angles of the pentagon be $\theta_0, \ldots, \theta_4$, so that their sum is equal to the corner angle $\theta$. Orient $\Pent(g,h;p)$ so that the induced boundary orientation is given by the sequence of vertices in definition \ref{defn:pentagon}. Denote by $\Delta[\Pent]$ its signed area. In \cite{Me10MScPaper0} we prove the following proposition.
\begin{prop}
\label{prop:twist_area}
If $\Pent(g,h;p)$ is nondegenerate and bounds an immersed disc, then
\[
\Twist \Big( \left[ g^{-1}, h^{-1} \right], p \Big) = \Delta[ \Pent(g,h;p) ].
\]
\qed
\end{prop}
The area of $\Pent(g,h;p)$ is just $3\pi - \theta$. Thus:
\begin{lem}
\label{pentagon_twist}
Suppose $\Pent(g,h;p)$ is nondegenerate and bounds an immersed disc.
\begin{enumerate}
\item 
If the segment $p \to [g^{-1},h^{-1}]p$ bounds $\Pent(g,h;p)$ on its left, then $\theta = 3 \pi - \Twist([g^{-1},h^{-1}],p)$. 
\item
If the segment $p \to [g^{-1},h^{-1}]p$ bounds $\Pent(g,h;p)$ on its right, then $\theta = 3 \pi + \Twist([g^{-1},h^{-1}],p)$.
\end{enumerate}
\qed
\end{lem}

\section{Representations and character varieties}
\label{sec:Euler_class_representation_space}
\label{sec:char_var_measure}

\subsection{Generalities}

Take a general surface $S$ and let $G = SL_2\R$ or $PSL_2\R$ (much of what we say applies to very general $G$, see \cite{Goldman84}). The \emph{representation variety} $R_G(S)$ describes all homomorphisms $\rho: \pi_1(S) \To G$. When $G=SL_2\R$, we may take a presentation for $\pi_1(S)$ with one relator; then a choice of homomorphism $\rho$ amounts to choosing for each generator a matrix in $SL_2\R$, such that the matrices satisfy the condition of the relator. The entries of the matrices can be considered as coordinate variables, so that $R_{SL_2\R}(S)$ is the solution set of some polynomial equations, a closed algebraic set.

Every $SL_2\R$-representation projects to a $PSL_2\R$-representation. When $S$ is a punctured torus, $\pi_1(S)$ is free so every $PSL_2\R$-representation lifts to an $SL_2\R$-representation; and $R_{PSL_2\R}$ is an obvious quotient of $R_{SL_2\R}$. This is not true for general surfaces; the sequel deals with such details. Henceforth in this paper, we write $R(S)$ to denote $SL_2\R$-representations.

The \emph{character} $\chi$ of an $SL_2\R$-representation $\rho$ is the function $\chi: \pi_1(S) \To \R$ given by $\chi(G) = \Tr (\rho(G))$. By using
trace relations, it can be shown that $\chi$ is determined by its values at only finitely many elements $\gamma_1, \ldots, \gamma_m \in \pi_1(S)$. We can then define a function $t: R(S) \To \R^m$ by $t(\rho) = \left( \Tr(\rho(\gamma_1)), \ldots, \Tr(\rho(\gamma_m)) \right)$; the \emph{character variety} is $X(S) = t(R(S)) = \text{Im}(t)$. It can be shown that $X(S)$ is a closed algebraic set. For details see \cite{Culler-Shalen}.

There is an action of $SL_2\R$ on $R(S)$ by conjugation, and we can consider the quotient space $R(S)/SL_2\R$. We can think of this quotient space as the moduli space of isomorphism classes of flat principal $SL_2\R$-bundles over $S$. In general it has singularities. The character variety can be considered as an ``algebraic" version of this quotient. Away from singularities, the character variety and this quotient can be identified.

\label{action of modular group}

There is an action of $\Aut \pi_1(S)$ on the representation and character varieties from the right, given by pre-composition: $\phi \in \Aut \pi_1(S)$ acts on a representation $\rho$ to give $\rho \circ \phi$, and descends to an action on the character variety. Such an action changes nothing in terms of the underlying geometry, but the representation and character change.

Since traces are invariant under conjugation, the action of $\Inn \pi_1(S)$ on $X(S)$ is trivial and we consider the action of $\Out \pi_1(S) = \Aut\pi_1(S)/\Inn \pi_1(S)$. Points in $X(S)$ which are related under this action ought to be considered as equivalent in terms of the underlying geometry.

In the present paper we are only concerned with punctured tori; in the sequel we consider higher-genus surfaces. For punctured tori, we can describe the character variety, and the action of $\Out \pi_1(S)$, explicitly.

\subsection{Characters of punctured torus representations}
\label{sec:algebra_tori}
\label{Characters_of_representations}

We now analyse representations $\pi_1(S) \To PSL_2\R$ and $SL_2\R$, where $S$ denotes a punctured torus. We do not consider geometric structures. We take a basepoint $q \in \partial S$.

Let $G,H$ be a basis, $\pi_1(S)= \langle G,H \rangle$. A representation $\rho: \pi_1(S) \To PSL_2\R$ or $SL_2\R$ is determined by $\rho(G)$ and $\rho(H)$. A representation into $PSL_2\R$ obviously lifts to $SL_2\R$, and we have two choices each for the lifts of $\rho(G)$ and $\rho(H)$. For now consider $\rho$ as
a representation into $SL_2\R$ and denote $\rho(G) = g$, $\rho(H) = h$.

We have stated that the character of $\rho$ is determined by the value of $\Tr \circ \rho$ at finitely many elements of $\pi_1(S)$. For the punctured torus with $\pi_1(S) = \langle G, H \rangle$, it is sufficient to consider only the three elements $G, H, GH$. For any word $W$ in $G, H$ and their inverses, we can write $\Tr(\rho(W))$ as a polynomial in $(x,y,z) = (\Tr g, \; \Tr h, \; \Tr gh)$ (see e.g. \cite{Magnus80,Fricke_Klein}).
For instance we have the important relation
\[
    \Tr[g,h] = \Tr^2 g + \Tr^2 h + \Tr^2 gh - \Tr g \Tr h \Tr gh -
    2
\]
and hence we define the polynomial
\[
    \kappa(x,y,z) = x^2 + y^2 + z^2 - xyz - 2.
\]
following notation of \cite{Goldman88,Goldman03}. For details see also \cite{Magnus80,Culler-Shalen,Fricke_Klein} or \cite[3.4]{Maclachlan_Reid}.

It is a classical result that if $\rho$ is \emph{irreducible} and defines the same triple $(\Tr g, \Tr h, \Tr gh)$ as another representation $\rho'$, then $\rho$ and $\rho'$ are conjugate; so that the triple $(\Tr g, \Tr h, \Tr gh)$ defines the pair $g,h \in SL_2\R$ uniquely up to conjgacy: see \cite{Goldman88,Fricke,Fricke_Klein}. (We shall characterise the triples arising from reducible representations below.) Recall a representation into $SL_2\R$ is \emph{reducible} if its image is a set of matrices which, acting via linear transformations on $\C^2$, leaves invariant
a line in $\C^2$. Thus in principle, for irreducible $\rho$ it is possible to deduce all
the geometry of $g$ and $h$, considered as isometries of the
hyperbolic plane, from the triple $(\Tr g, \Tr h, \Tr gh)$. This
motivates results such as the lemmata of section \ref{compositions of isometries}.

The set of all $(x,y,z) = (\Tr g, \Tr h, \Tr gh) \in \R^3$ is the
character variety $X(S)$ of $S$. It is not all of $\R^3$, and the
following theorem describes $X(S)$ exactly. We refer to \cite[thm.
4.3]{Goldman88} for a proof.
\begin{thm}[Goldman, \cite{Goldman88}]
\label{character variety}
    Given $(x,y,z) \in \R^3$, there exist $g,h \in SL_2\R$
    such that \[(\Tr g, \Tr h, \Tr gh) = (x,y,z)\] if and only if
    \[
    \Tr[g,h] = x^2 + y^2 + z^2 - xyz - 2 \geq 2 \quad \text{or at least one of } |x|, |y|, |z| \text{ is } \geq 2.
    \]
    \qed
\end{thm}
Thus the set of $(x,y,z) \in \R^3$ without corresponding representations (i.e $\R^3 \backslash X(S)$) are those with $\kappa(x,y,z) < 2$ and $-2 < x, y, z < 2$: see figure \ref{fig:9}. (Actually if $g,h$ exist but $\Tr[g,h] < 2$ then $\Tr[g,gh] < 2$ also, so by lemma \ref{axes crossing} all of $g,h,gh$ give hyperbolic isometries
of $\hyp^2$ and hence \emph{all} $|x|,|y|,|z| > 2$.) 

\begin{figure}[tbh]
\centering
\includegraphics[scale=0.3, angle=-90]{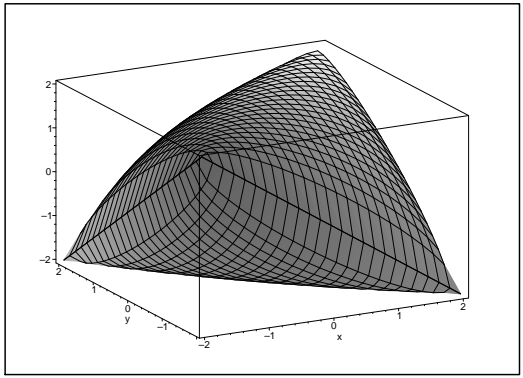}
\caption[$\R^3 \backslash X(S)$, i.e. the region of $\R^3$ without
representations]{$\R^3 \backslash X(S)$, i.e. the region of $\R^3$
without representations: (strictly) inside this curved
tetrahedron-like surface.} \label{fig:9}
\end{figure}

For representations $\pi_1(S) \To PSL_2\R$, the character variety
can be described simply also. There are four ways to  lift $\rho(G),
\rho(H)$ into $SL_2\R$, which are related by sign changes. Thus we
simply take the character variety $X(S)$ of representations into
$SL_2\R$ modulo the equivalence relation
\[
    (x,y,z) \sim (-x,-y,z) \sim (-x,y,-z) \sim (x,-y,-z)
\]
induced by these four possible lifts. The notion of reducibility
still makes sense: elements of $PSL_2\R$ act via linear
transformations on $\C^2$ up to a reflection in the origin, hence on
$\CP^1$, so the idea of an invariant line still makes sense. And for
representations into $PSL_2\R$ the value of $\kappa(x,y,z) =
\Tr[g,h]$ is well-defined, even if the signs of $x,y,z$ are
ambiguous.

The reducible representations have a simple characterisation: see
\cite{Culler-Shalen,Goldman03} for a proof.
\begin{prop}
\label{reducible representation}
    The representation $\rho: \pi_1(S) \To PSL_2\R$ or $SL_2\R$ is
    reducible if and only if $\Tr [g,h] = 2$, i.e. iff the character
    $(x,y,z)$ of $\rho$ satisfies $\kappa(x,y,z)=2$.
    \qed
\end{prop}

Note this implies that an abelian representation in $SL_2\R$ is reducible. For representations into $PSL_2\R$ we see $[g,h] = \pm I$; but by the classification in corollary \ref{commutator_trace_info} this implies $\Tr[g,h] = 2$.

We have now defined the character variety $X(S) \subset \R^3$.
Points with $\kappa(x,y,z) = 2$ describe reducible representations,
which include abelian representations. Points with $\kappa(x,y,z)
\neq 2$ describe irreducible representations, hence describe a
conjugacy class of representations precisely. For $t \neq 2$, we denote the
space of all representations (up to conjugacy) with $\Tr[g,h] = t$
by $X_t(S) = \kappa^{-1}(t) \cap X(S)$:
this is a \emph{relative} character variety of $S$.

\subsection{Nielsen's Theorem}
\label{sec:Nielsen}

When $S$ is a punctured torus, $\Out \pi_1(S)$ has a particularly explicit geometric interpretation. Every homeomorphism of $S$ which preserves a basepoint determines an automorphism of $\pi_1(S)$. A general homeomorphism of $S$ determines an automorphism of $\pi_1(S)$, up to conjugacy, i.e. an outer automorphism. On the other hand, we have the following theorem: see \cite{Goldman03}, and for further details \cite{Stillwell,Nielsen27}.
\begin{thm}
    Any automorphism of $\pi_1(S) = \langle G, H \rangle$ is induced from a homeomorphism of $S$.
    \qed
\end{thm}
There is then an isomorphism
\[
    \mcg (S) = \frac{\text{Homeo}\, (S)}{\text{Isotopy}} \cong \frac{\Aut \pi_1(S)}{\Inn \pi_1(S)} = \Out \pi_1(S).
\]
Here $\mcg(S)$ is the mapping class group, i.e. the group of homeomorphisms of $S$ up to isotopy. These homeomorphisms need not be fixed on the boundary (of course $\partial S$ must be sent to itself, as a set). So, for instance, a Dehn twist about the boundary is equivalent to the identity. A similar result is true for closed surfaces, but for no other surfaces with boundary: this is the Dehn--Nielsen theorem (see e.g. \cite{Stillwell,Nielsen27}).

We also have the following algebraic theorem of Nielsen: see \cite{Nielsen18}, \cite[thm. 3.9]{Magnus_Karrass_Solitar} or \cite[prop. 5.1]{Lyndon_Schupp}. The proof relies upon Nielsen's description of the automorphism group of the free group on two generators.
\begin{thm}[Nielsen]
\label{Nielsen_thm}
    An automorphism $\phi$ of $\langle G,H \rangle$ takes $[G,H]$ to a conjugate
    of itself or its inverse $[H,G]$.
    \qed
\end{thm}
Thus there is an algebraic notion of orientation of bases; we call the automorphism $\phi$ either \emph{orientation-preserving} or \emph{orientation-reversing} as $[G,H]$ is taken respectively to a conjugate of itself or of $[H,G]$.

\subsection{The action on the character variety}
\label{action of modular group 2}

Now consider the effect of changing basis $(G,H) \mapsto (G',H')$ on a representation $\rho: \pi_1(S) \To SL_2\R$ by pre-composition, as discussed in section \ref{action of modular group}. The underlying geometry does not change but the character changes $(x,y,z) = (\Tr g, \Tr h, \Tr gh) \mapsto (\Tr g', \Tr h', \Tr g'h') = (x',y',z')$. Since trace is invariant under conjugation, this action descends to an action of $\Out \pi_1(S) \cong \mcg(S)$. Points in $X(S)$ which are related under this action ought to be considered as equivalent; we will now describe this equivalence relation.

By Nielsen's theorem \ref{Nielsen_thm}, $[G,H]$ is conjugate to $[G',H']^{\pm 1}$, so $\Tr[g,h] = \Tr[g',h']$ and $\kappa(x,y,z) = \kappa(x',y',z')$. That is, $(x,y,z)$ and $(x',y',z')$ lie on the same level set of the polynomial $\kappa$. The level set $\kappa(x,y,z)=t$ amounts to fixing the trace on the boundary; this is the relative character variety $X_t(S)$.

The group $\mcg (S) \cong \Out \pi_1(S)$ is well known to be isomorphic to $GL_2\Z$, viewing the punctured torus $S$ as a quotient of the Euclidean plane by two linearly independent translations, with a lattice removed. There are natural identifications between: bases of $H_1(S) \cong \Z \oplus \Z$; pairs of free isotopy classes of simple closed curves on $S$ spanning $H_1(S)$; and conjugacy classes of bases of $\pi_1(S)$. The group $\mcg(S) \cong \Out \pi_1(S) \cong GL_2\Z$ acts simply and transitively on these objects, by the usual matrix multiplication on $\Z \oplus \Z$. Take a basis $(G,H)$ for $\pi_1(S)$, and identify the conjugacy class $[(G,H)]$ of this basis with the basis $((1,0), (0,1))$ of $H_1(S) \cong \Z \oplus \Z$.

It is well known that $GL_2\Z$ has a small set of generators, for instance
\[
    \begin{bmatrix} 1 & 0 \\ 0 & -1 \end{bmatrix}, \quad
    \begin{bmatrix} -1 & 0 \\ 0 & -1 \end{bmatrix}, \quad
    \begin{bmatrix} 0 & -1 \\ 1 & -1 \end{bmatrix}, \quad
    \begin{bmatrix} 1 & 1 \\ 0 & 1 \end{bmatrix}.
\]
It follows from the above that any two conjugacy classes of bases of $\pi_1 (S)$ are related by some combination of the matrices above, considered as elements of $\Out \pi_1(S) \cong \mcg (S)$. We will consider the actions of these matrices on $X(S)$ separately.

\begin{enumerate}

\item {\bf The matrix $M = \begin{bmatrix} 1 & 0 \\ 0 & -1
\end{bmatrix}$.}

As an element of the mapping class group, $M$ is orientation-reversing and an involution. If $G,H$ are as in figure \ref{fig:10}, then $M$ acts topologically (not metrically) as a reflection in a plane intersecting $S$ in an arc and circle, sending $[(G,H)] \mapsto [(G,H^{-1})]$. Letting $G',H'$ be the image of $G,H$ under the automorphism, and letting $(x,y,z) = (\Tr g, \Tr h, \Tr gh)$, $(x',y',z') = (\Tr g', \Tr h', \Tr g'h')$ denote the respective characters obtained, we have $(G',H',G'H') = (G, H^{-1}, GH^{-1})$ so it follows that
\[
    (x',y',z') = (\Tr g', \Tr h', \Tr g'h') = (\Tr g, \Tr h^{-1}, \Tr gh^{-1}) = (x,y,xy-z).
\]
Here we use a standard trace relation $\Tr gh^{-1} = \Tr g \Tr h - \Tr gh$. Since $M =
M^{-1}$, the actions of $M$ and $M^{-1}$ on $X(S)$ are both given by $(x,y,z) \mapsto (x,y,xy-z)$.

\begin{figure}[tbh]
\centering
\includegraphics[scale=0.5]{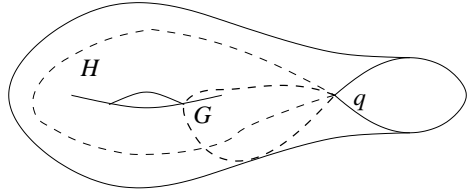}
\caption{A standard set of basis curves on $S$} \label{fig:10}
\end{figure}

\item {\bf The matrix $-I$.}

As an element of $\mcg(S)$, $-I$ gives a homeomorphism which is isotopic to an involution, corresponding topologically (not metrically) to a rotation of $\pi$ about an axis intersecting $S$ in $3$ points. After adjusting to find a representative fixing $\partial S$, we see $M$ is represented by the automorphism $(G,H) \mapsto (G^{-1},H^{-1})$. The induced action of $-I$ on the character variety $X(S)$ is trivial.

\item {\bf The matrix $M = \begin{bmatrix} 0 & -1 \\ 1 & -1
\end{bmatrix}$.}

    This matrix is of order $3$, represented by the automorphism $(G,H) \mapsto (H, H^{-1} G^{-1})$. We have
    \[
        (x',y',z') = (\Tr g', \Tr h', \Tr g'h') = (\Tr h, \Tr h^{-1} g^{-1}, \Tr g^{-1}) = (y,z,x).
    \]
    So the actions of $M,M^{-1}$ on $X(S)$ are given by $(x,y,z) \mapsto (y,z,x), (z,x,y)$.

\item {\bf The matrix $M = \begin{bmatrix} 1 & 1 \\ 0 & 1
\end{bmatrix}$.}

    The automorphism $(G,H) \mapsto (G, GH)$ represents $M$ and corresponds to a Dehn twist about $G$. We have $(x',y',z') = (\Tr g', \Tr h', \Tr g'h') = (\Tr g, \Tr gh, \Tr g^2 h)$ which can easily be computed in terms of $x,y,z$. Similarly we have $M^{-1}$ represented by the automorphism $(G, G^{-1} H)$ and can
    again compute $(x',y',z') = (\Tr g, \Tr g^{-1} h', \Tr h)$. The actions of $M, M^{-1}$ on $X(S)$ are given by
    \[
        (x,y,z) \mapsto (x,z,xz-y), (x,xy-z,y).
    \]

\end{enumerate}

Putting these together now gives the following result.
\begin{prop}
Let $\rho_1, \rho_2: \pi_1 (S) \To SL_2\R$ be irreducible
representations and let $(G_1, H_1)$, $(G_2, H_2)$ be free bases of
$\pi_1 (S)$ such that $\rho_i$ has character $(x_i,y_i, z_i)$ with
respect to the basis $(G_i, H_i)$.The following are equivalent:
\begin{enumerate}
    \item
        There exists $\phi \in \Aut \pi_1(S)$ such that
        $\rho_1 \circ \phi$ and $\rho_2$ are conjugate representations
        into
        $SL_2\R$;
    \item
        $(x,y,z) \sim (x', y', z')$ under the
        equivalence relation generated by permutation of coordinates and the
        relation $(x,y,z) \sim (x,y,xy-z)$.
\end{enumerate}
\qed
\end{prop}

For representations into $PSL_2\R$ we must also add sign-change
relations.
\begin{prop}
\label{Markoff_moves} Let $\rho_1, \rho_2: \pi_1(S) \To PSL_2\R$ be
irreducible representations and let $(G_1, H_1)$, $(G_2, H_2)$ be
free bases of $\pi_1 (S)$. Choosing lifts of $g_i, h_i$ into
$SL_2\R$ arbitrarily, let $\rho_i$ have character $(x_i, y_i, z_i)$
with respect to the basis $(G_i, H_i)$. The following are
equivalent:
\begin{enumerate}
    \item
        There exists $\phi \in \Aut \pi_1(S)$ such that
        $\rho_1 \circ \phi$ and $\rho_2$ are conjugate representations in
        $PSL_2\R$;
    \item
        $(x,y,z) \sim (x', y', z')$ under the
        equivalence relation generated by permutation of coordinates and the
        relations $(x,y,z) \sim (x,y,xy-z)$ and $(x,y,z) \sim (-x,-y,z)$.
\end{enumerate}
\qed
\end{prop}
We call triples of numbers with this relation \emph{Markoff triples}. (Classically, this term denotes solutions to $\kappa(x,y,z)=-2$, but we use it more broadly: see \cite{Bowditch_McShane,Bowditch96}) The equivalence relation can be considered as the action of a semidirect product $PGL_2\Z \ltimes \left( \frac{\Z}{2} \oplus \frac{\Z}{2} \right)$. See \cite{Goldman03}. 

If we restrict our attention to orientation-preserving changes of
basis, then we may not consider all of the moves above. In
particular, we cannot transpose coordinates freely. But we can
certainly apply the relations given by the action of matrices
(ii)--(iv) above; and if we are considering representations into
$PSL_2\R$, then we may apply the sign-change relations also. While
at times we will need to consider the orientation of a basis, we
will always consider the above machinery without
orientation-preserving restrictions.

\subsection{Abelian and virtually abelian representations}
\label{virtually abelian reps}

Consider abelian representations: we have seen above that all abelian representations are reducible (over $\C$). Conversely, a character of a reducible representation is also the character of an abelian representation: a reducible representation $\rho$ can be taken to map $G,H$ to upper triangular matrices; ignoring the top right entry gives an abelian representation $\rho'$ taking $G,H$ to diagonal matrices, with the same character.

It is easy to see that the image of an abelian representation $\rho:
\pi_1(S) \To PSL_2\R$ consists of one of the following:
\begin{enumerate}
    \item
        elliptics which all rotate about the same point, and the identity;
    \item
        parabolics with the same fixed point at infinity, and the identity;
    \item
        hyperbolics with the same axis, and the identity;
    \item
        the identity alone.
\end{enumerate}

Now consider virtually abelian representations, i.e. those whose image contains an abelian subgroup of finite index. Define the set $V \subset \R^3$ as
    \[
        V = \left\{ 0 \times 0 \times \R \backslash [-2,2]
        \right\} \cup \left\{ 0 \times \R \backslash [-2,2] \times
        0 \right\} \cup \left\{ \R \backslash [-2,2] \times 0
        \times 0 \right\}.
    \]
We can easily verify, checking the conditions of theorem
\ref{character variety}, that $V \subset X(S)$. Further, using
\ref{character variety}, the set of
points in $X(S)$ with two coordinates equal to $0$ is precisely $V$,
taken together with the six points $(0,0, \pm 2)$, $(0, \pm 2, 0)$,
$(\pm 2, 0, 0)$. (If $(0,0,z) \in X(S)$ then $\kappa(x,y,z) = z^2 -
2 \geq 2$ is equivalent to $|z| \geq 2$.) We can also see from above
that no abelian representations have characters in $V$.

A geometric description of representations with characters
in $V$ can easily be given.
\begin{lem}
\label{va1}
    Let $g,h \in PSL_2\R$. The following are equivalent:
    \begin{enumerate}
    \item
        We may lift $g,h$ to $SL_2\R$ so that
        $(\Tr(g), \Tr(h), \Tr(gh)) \in V$.
    \item
        Two of $\{g,h,gh\}$ are half-turns about points $q_1 \neq
        q_2 \in \hyp^2$ and the third is a nonzero translation
        along the axis $q_1 q_2$.
    \end{enumerate}
\qed
\end{lem}

Note that in this situation, the subgroup $\langle g,h \rangle$ of
$PSL_2\R$ is an infinite dihedral group consisting of translations
along $q_1 q_2$ and half-turns about points on $q_1 q_2$. It
therefore contains an index 2 subgroup of translations along $q_1
q_2$, which is abelian. So $\rho$ in this case is indeed virtually
abelian. It is easy to see that $V$ is preserved by Markoff moves, sign changes, and permutations of coordinates, hence:

\begin{lem}
\label{va2}
    Let $\rho$ be a representation with
    $(\Tr(g), \Tr(h), \Tr(gh))
    \in V$. Let $G',H'$ be another basis of $\pi_1(S)$.
    Then $(\Tr(g'), \Tr(h'), \Tr(g'h')) \in
    V$ also.
\qed
\end{lem}

In fact, this is a complete characterisation of virtually
abelian representations.

\begin{lem}
\label{virtually abelian but not abelian}
    Let $\rho: \pi_1(S) \To PSL_2\R$ be a representation. The character
            \[
                (x,y,z) = (\Tr g, \Tr h, \Tr gh) \in V
            \]
    if and only if $\rho$ is virtually abelian but not abelian.
\end{lem}

\begin{proof}
    We have already established that $(x,y,z) \in V$ is the
    character of a virtually abelian but not abelian representation. So let
    $\Lambda = \rho(\pi_1(S)) \subset PSL_2\R$ be virtually abelian but not abelian. So there is a
    finite index subgroup $F$ of $\Lambda$ which is abelian. Let
    $F$ have index $n>1$ in $\Lambda$. Note if $\alpha, \beta \in
    \Lambda$ lie in the same left coset of $F$ then $\alpha F
    = \beta F$ and $F \alpha^{-1} = F \beta^{-1}$, so $\alpha F
    \alpha^{-1} = \beta F \beta^{-1}$. Hence there are only finitely
    many conjugate subgroups of $F$ in $\Lambda$; by taking their
    intersection we obtain
    a \emph{normal} finite index abelian subgroup of $\Lambda$.
    Passing to this subgroup, we may assume $F$ is normal.

    Let $\Fix(F)$ denote the set of points in $\overline{\hyp^2}$ fixed
    by every element of $F$. I claim $\Fix(F)$ is invariant under
    the action of $\Lambda$. Take $h \in \Lambda$ and $p \in
    \Fix(F)$; we must show $h(p) \in \Fix(F)$. So take $f \in F$;
    then $h^{-1}fh \in F$ by normality; so $h^{-1}fh (p) = p$, hence
    $f(h(p)) = h(p)$. So $h(p) \in \Fix(F)$ as desired.
    We now split into cases according
    to the possibilities for $F$.

    \noindent \textbf{Case (i).} Suppose $F = \{1\}$, so $\Lambda$ is
    finite, so every element has finite order, hence is elliptic or the identity.
    Take an arbitrary point $q \in \hyp^2$ and let $p$ be the centre of
    mass of the (finite) orbit of $q$ under $\Lambda$ (see \cite[2.5.19]{Thurston_book} for
    more details). Then $p$ is fixed by every element of $\Lambda$.
    So every element of $\Lambda$ is the identity or an elliptic
    fixing $p$. Hence $\Lambda$ is abelian, a contradiction.

    \noindent \textbf{Case (ii).} Assume $F$ consists of
    the identity and elliptics fixing a point $q$, so $\Fix(F) = q$.
    So every element of $\Lambda$ fixes $q$, and $\Lambda$ consists
    of the identity and elliptics fixing $q$. Thus $\Lambda$ is
    abelian, a contradiction.

    \noindent \textbf{Case (iii).} Assume $F$ consists of the identity and parabolics
    with fixed point $q$, so $\Fix(F) = q$. So every element of
    $\Lambda$ fixes $q$. There cannot exist a hyperbolic $h \in \Lambda$,
    for then $h^n \in F$ would be hyperbolic. So $\Lambda$ consists of the identity and
    parabolics fixing $q$, a contradiction.

    \noindent \textbf{Case (iv).} Now assume $F$ consists of the
    identity and
    hyperbolic isometries with axis $l$, so $\Fix(F)$ consists of the endpoints of $l$ at infinity. So every
    element of $\Lambda$ is either the
    identity, or hyperbolic with axis $l$, or elliptic of order $2$
    with fixed point on $l$. If there are no elliptics then
    $\Lambda$ is abelian and we have a contradiction. Otherwise
    the translations (and the identity) form an index-2 subgroup
    of $\Lambda$. The pair $g,h$ (where $G,H \in \pi_1 (S)$ is a basis) must
    contain at least one half turn; hence the triple
    $g,h,gh$ contains exactly two half turns about distinct points on $l$, and one hyperbolic element
    translating along $l$.
    By lemma \ref{va1} the character of $\rho$ with
    respect to this basis lies in $V$.
\qed
\end{proof}

\subsection{Reducible representations}

We have seen (proposition \ref{reducible representation}) that the reducible representations are precisely those with $\Tr[g,h] = 2$. We can classify these more explicitly. These include abelian representations. From the previous section, all the representations which are virtually abelian, but not abelian, have character in $V$, hence have $\Tr[g,h] > 2$. So we have immediately:
\begin{lem}
\label{reducible_va_abelian}
    A reducible virtually abelian representation $\rho: \pi_1(S) \To
    PSL_2\R$ is abelian.
    \qed
\end{lem}

We will now describe the non-abelian reducible representations
rather explicitly.
\begin{lem}
\label{reducible_description}
    Let $\rho$ be a non-abelian reducible representation $\pi_1(S)
    \To PSL_2\R$ and let $G,H$ be a basis of $\pi_1(S)$. Then one of
    the following occurs:
    \begin{enumerate}
        \item
            one of $g,h$ is hyperbolic and the other is parabolic,
            and $g,h$ share a fixed point at infinity;
        \item
            $g,h$ are both hyperbolic, sharing exactly one fixed
            point at infinity.
    \end{enumerate}
\end{lem}

\begin{proof}
    If $g$ or $h$ is the identity then $\rho$ is trivially abelian.
    Suppose $g$ is elliptic. Then we may conjugate in $PSL_2\R$ so
    that the fixed point of $g$ lies at $i$ in the upper half plane
    model. Then we may write
    \[
        g = \pm \begin{bmatrix} \cos \theta & - \sin \theta \\ \sin
        \theta & \cos \theta \end{bmatrix}, \quad h =
        \pm \begin{bmatrix} a & b \\ c & d \end{bmatrix}.
    \]
    where $\sin \theta \neq 0$. We obtain
    \[
        \Tr[g,h] = 2 + (a^2 + b^2 + c^2 + d^2 - 2) \sin^2 \theta
    \]
    hence as $\Tr[g,h] = 2$ and $\sin \theta \neq 0$,
    \[
        a^2 + b^2 + c^2 + d^2 = 2 = 2(ad - bc).
    \]
    Thus $(a-d)^2 + (b+c)^2 = 0$, so $a=d$ and $b=-c$. With determinant $1$, then $h$ is a rotation about $i$, so $g,h$ commute and $\rho$ is abelian.

    If $h$ is elliptic, then we apply the same argument noting
    $\Tr[h,g] = \Tr[g,h]$.

    Hence each of $g,h$ is hyperbolic or parabolic.
    Suppose first that one of $g,h$ is
    parabolic, without loss of generality $g$. Then we may conjugate
    in $PSL_2\R$, and replacing $G$ with $G^{-1}$ if necessary we have
    \[
        g^{\pm 1} = \pm \begin{bmatrix} 1 & 1 \\ 0 & 1 \end{bmatrix}, \quad
        h = \pm \begin{bmatrix} a & b \\ c & d \end{bmatrix}.
    \]
    We can calculate $\Tr[g,h] = \Tr[g^{-1},h] = 2+c^2$ but
    $\rho$ is reducible,
    so $\Tr[g,h] = 2$. Thus $c=0$ and $h$ is upper triangular, hence $h$ fixes
    $\infty$ in common with $g$. If $h$ is parabolic then $\rho$ is
    abelian, since $g,h$ are parabolics with the same fixed point.
    So $h$ is hyperbolic.

    Suppose now that both $g,h$ are hyperbolic. We may conjugate so $g$ has fixed points at infinity $\{0, \infty\}$ in the upper half-plane model, and may then write
    \[
      g = \pm \begin{bmatrix} r & 0 \\ 0 & r^{-1} \end{bmatrix}, \quad h = \pm \begin{bmatrix} a & b \\ c & d \end{bmatrix}
    \]
    where $r \neq \pm 1$. We obtain
    \[
      2 = \Tr[g,h] = 2ad - \left( r^2 + r^{-2} \right) bc
    \]
    which, since $ad-bc = 1$, gives $bc(r-r^{-1})^2 = 0$. As $r \neq \pm 1$ then we have $b=0$ or $c=0$, but not both: if $b=c=0$ then $\rho$ is abelian. Thus $h$ shares exactly one fixed point at infinity with $g$.
\qed
\end{proof}

\section{The Construction of Punctured Tori}
\label{sec:construction_tori}

\subsection{Statement and preliminaries}

Throughout this section, as usual, let $S$ be a punctured torus, and let $G,H$ be a basis of $\pi_1(S)$, with a basepoint $q$ chosen on the boundary. Let $\rho: \pi_1 (S) \To PSL_2\R$ be a representation, and let $\rho(G)=g, \rho(H)=h$. We prove theorem \ref{punctured_torus_theorem}.  The strategy of the proof is as follows. We take some lift of $g,h,gh$ into $SL_2\R$ and let $(x,y,z) = (\Tr g, \Tr h, \Tr gh) \in X(S)$ be the character of $\rho$, which is well-defined up to the equivalence relations $(x,y,z) \sim (-x,-y,z) \sim (-x,y,-z) \sim (x,-y,-z)$. Then we have $\Tr[g,h] = \kappa(x,y,z) = x^2 + y^2 + z^2 - xyz - 2$, which is well-defined regardless of the choice of lift into $SL_2\R$; indeed $[g,h]$ gives a well-defined element of $\widetilde{PSL_2\R}$. The proof is split into cases according to the value of $\Tr[g,h]$.

In section \ref{sec:less_than_-2} we treat the case $\Tr[g,h] \in (-\infty, -2)$. By corollary \ref{commutator_trace_info}, we see $\Tr[g,h] < -2$ implies that $[g,h] \in \Hyp_1 \cup \Hyp_{-1}$. We will construct a hyperbolic cone-manifold structure with a preferred orientation, accordingly as $[g,h] \in \Hyp_1$ or $\Hyp_{-1}$.

In section \ref{sec:-2} we treat $\Tr[g,h] = -2$. In this case we have, similarly, from corollary \ref{commutator_trace_info}, $[g,h] \in \Par_1^-$ or $\Par_{-1}^+$. Again we will construct a hyperbolic cone-manifold structure with a preferred orientation accordingly as $[g,h] \in \Par_1^-$ or $\Par_{-1}^+$.

In section \ref{sec:-2_to_2} we consider $\Tr[g,h] \in (-2,2)$. In this case from corollary \ref{commutator_trace_info}, $[g,h] \in \Ell_{-1}$ or $\Ell_1$. We find cone-manifold structures of one of the two possible orientations accordingly as $[g,h] \in \Ell_1$ or $\Ell_{-1}$.

In section \ref{sec:2} we treat $\Tr[g,h] = 2$. From corollary \ref{commutator_trace_info}, $[g,h] \in \{1\} \cup \Par_0$. By proposition \ref{reducible representation}, these are precisely the reducible representations. Some of these representations are virtually abelian (in fact abelian, using lemma \ref{virtually abelian but not abelian}); we will prove these are not holonomy representations. For the other reducible representations we will find a cone-manifold structure of a preferred orientation, accordingly as $[g,h] \in \Par_0^+$ or $\Par_0^-$.

In section \ref{sec:more_than_2} we consider the most difficult case, $\Tr[g,h] > 2$. Some of these representations are virtually abelian, and we will eliminate these. For the other representations, by corollary \ref{commutator_trace_info} we have $[g,h] \in \Hyp_0$. There is no preferred orientation; we do not specify it in advance.

\subsection{The case $\Tr[g,h] < -2$: complete and discrete}
\label{sec:less_than_-2}

From lemma \ref{axes crossing}, if $\Tr[g,h] < 2$ then $g,h$ are both hyperbolic and their axes cross. If $\Tr[g,h] < -2$ then this commutator is hyperbolic. By four applications of lemma \ref{hyperbolic commutator}, the arrangement of axes of various commutators is as shown in figure \ref{fig:11}.

Taking an arbitrary point $p \in \Axis [g,h]$ we investigate the arrangement of $\Pent(g,h;p)$. In general we have $\alpha (\Axis \beta) = \Axis (\alpha \beta \alpha^{-1})$. So $h p$ lies on $\Axis [h,g^{-1}] = \Axis [g^{-1}, h]$. Similarly, $g h p \in \Axis [g, h]$ and $h^{-1}gh \in \Axis [g,h^{-1}]$. Obviously $[g^{-1},h^{-1}] p \in \Axis [g^{-1},h^{-1}]$. Given the arrangement of axes, it is clear that $\Pent(g,h;p)$ bounds an embedded disc.

\begin{figure}[tbh]
\begin{center}
\includegraphics[scale=0.35]{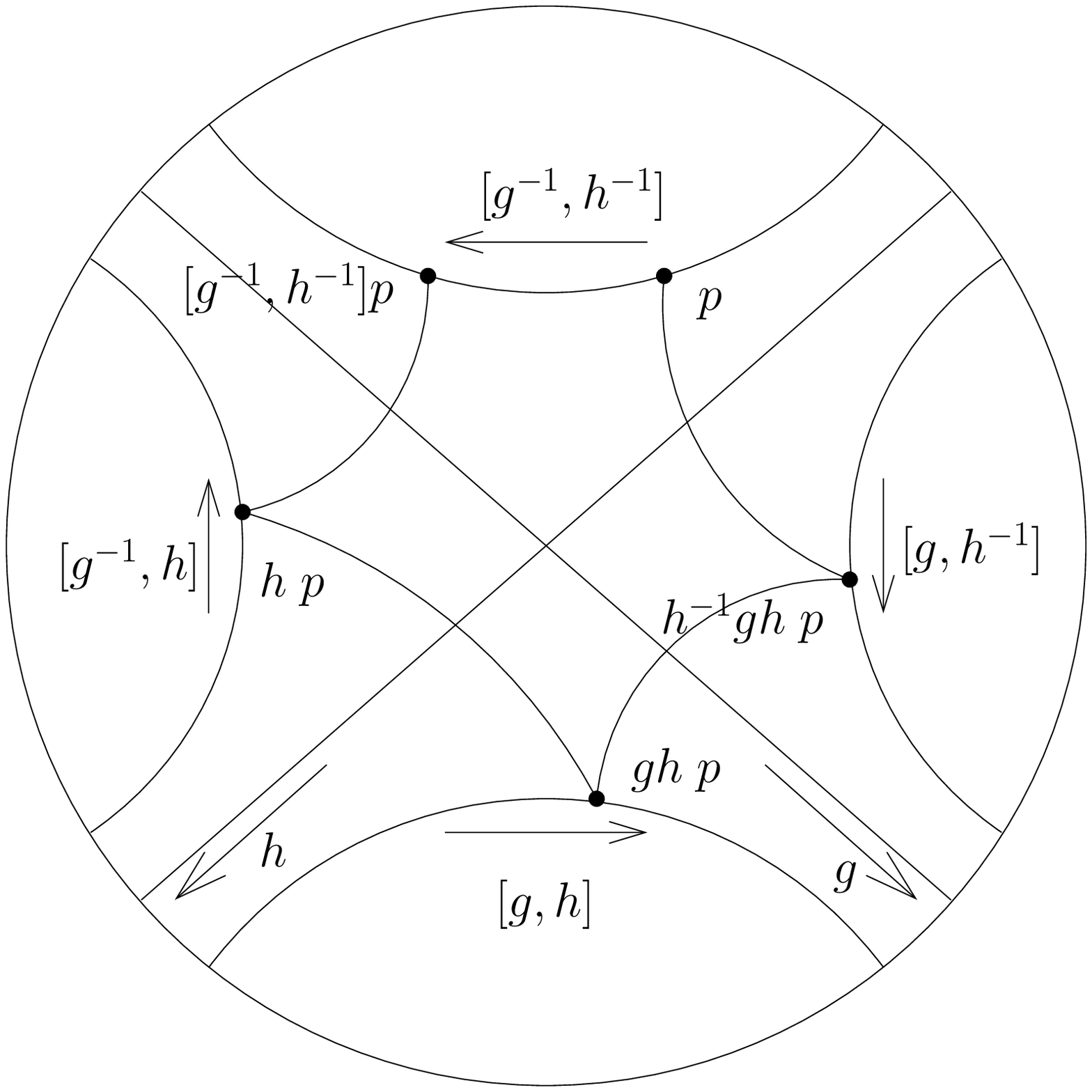}
\caption{The arrangement of axes of commutators if $\Tr[g,h] < -
2$; $\Pent(g,h;p)$ then bounds an embedded disc.}
\label{fig:11}
\end{center}
\end{figure}

By lemma \ref{construction_lemma}, this gives a hyperbolic cone-manifold structure on $S$ with no interior cone points and one corner point. Now $[g^{-1},h^{-1}]$, as a hyperbolic isometry, simply translates along $\Axis[g^{-1},h^{-1}]$. So $\Twist([g^{-1},h^{-1}],p)$ is a multiple of $2\pi$; since $[g,h] \in \Hyp_{\pm 1}$ is conjugate to $[g^{-1},h^{-1}]$, we have $\Twist([g^{-1},h^{-1}],p) = \pm 2\pi$. By lemma \ref{pentagon_twist}, the corner angle $\theta = 3\pi \pm 2\pi$. By lemma \ref{Gauss-Bonnet}, $\theta \in (0, 3\pi)$. So $\theta = \pi$. That is, the corner point at $q$ is actually no corner at all, and we have obtained a hyperbolic structure on $S$ with totally geodesic boundary.

The above construction works for any basis of $\pi_1(S)$, and any point on $\Axis [g^{-1},h^{-1}]$. It is clear why: $\rho$ is the holonomy of a complete hyperbolic structure on $S$ and is discrete. By choosing $p$ inside or outside the convex core, we may extend or truncate the surface with geodesic boundary, as described in section \ref{rigidity section}.

Either orientation of the torus is possible, depending on the arrangement of $g$ and $h$. If the axes of $g,h$ intersect in the manner of figure \ref{fig:11}, then $p \to [g^{-1},h^{-1}] p$ bounds $\Pent(g,h;p)$ to its left; hence $\partial S$ traversed in the direction of $[G,H]$ bounds $S$ on its left; and the twist of $[g^{-1},h^{-1}]$ at $p$ is positive, $\Twist([g^{-1},h^{-1}],p) = 2\pi$, so (from proposition \ref{twist_bounds}) $[g,h] \in \Hyp_{1}$. If we choose $p$ sufficiently close to $\Axis[g^{-1},h^{-1}]$, $\Pent(g,h;p)$ still bounds an embedded disc, and by lemma \ref{pentagon_twist}, $\theta = 3\pi - \Twist([g^{-1},h^{-1}],p)$. In the case where the axes of $g,h$ intersect in the opposite manner, we obtain oppositely oriented results.
\begin{prop}
\label{orientation_1}
    Let $\rho$ be a representation and $G,H$ a basis of $\pi_1(S)$
    with $\Tr[g,h] < -2$. Suppose $[g,h] \in \Hyp_1$ (resp.
    $\Hyp_{-1}$). Then $\rho$ is the holonomy of a complete
            hyperbolic structure in which $\partial S$, traversed in the
            direction homotopic to $[G,H]$, bounds $S$ on its left (resp.
            right).
            The axes of $g,h$ intersect in the manner shown in figure
            \ref{fig:11} (resp. the opposite manner).
            For $p \in \Axis
            [g^{-1},h^{-1}]$ we have $\Twist([g^{-1},h^{-1}],p) = 2\pi$ (resp. $-2\pi$).
            For $p$ sufficiently close to $\Axis[g^{-1},h^{-1}]$, we obtain a
            hyperbolic cone-manifold structure on $S$ with one
            corner point. The
            corner angle is given by $\theta = 3\pi - \Twist([g^{-1},h^{-1}],p)$
            (resp. $3\pi + \Twist([g^{-1},h^{-1}],p)$).
\qed
\end{prop}

\subsection{The case $\Tr[g,h] = -2$: parabolics and cusps}
\label{sec:-2}

This case proceeds similarly to the previous case. By corollary \ref{commutator_trace_info}, $[g,h]$ lies in $\Par_1^-$ or $\Par_{-1}^+$. The isometries $g,h$ are still hyperbolic and their axes cross. Using lemma \ref{parabolic_commutator} four times, we have the situation of figure \ref{fig:12}.

\begin{figure}[tbh]
\centering
\includegraphics[scale=0.35]{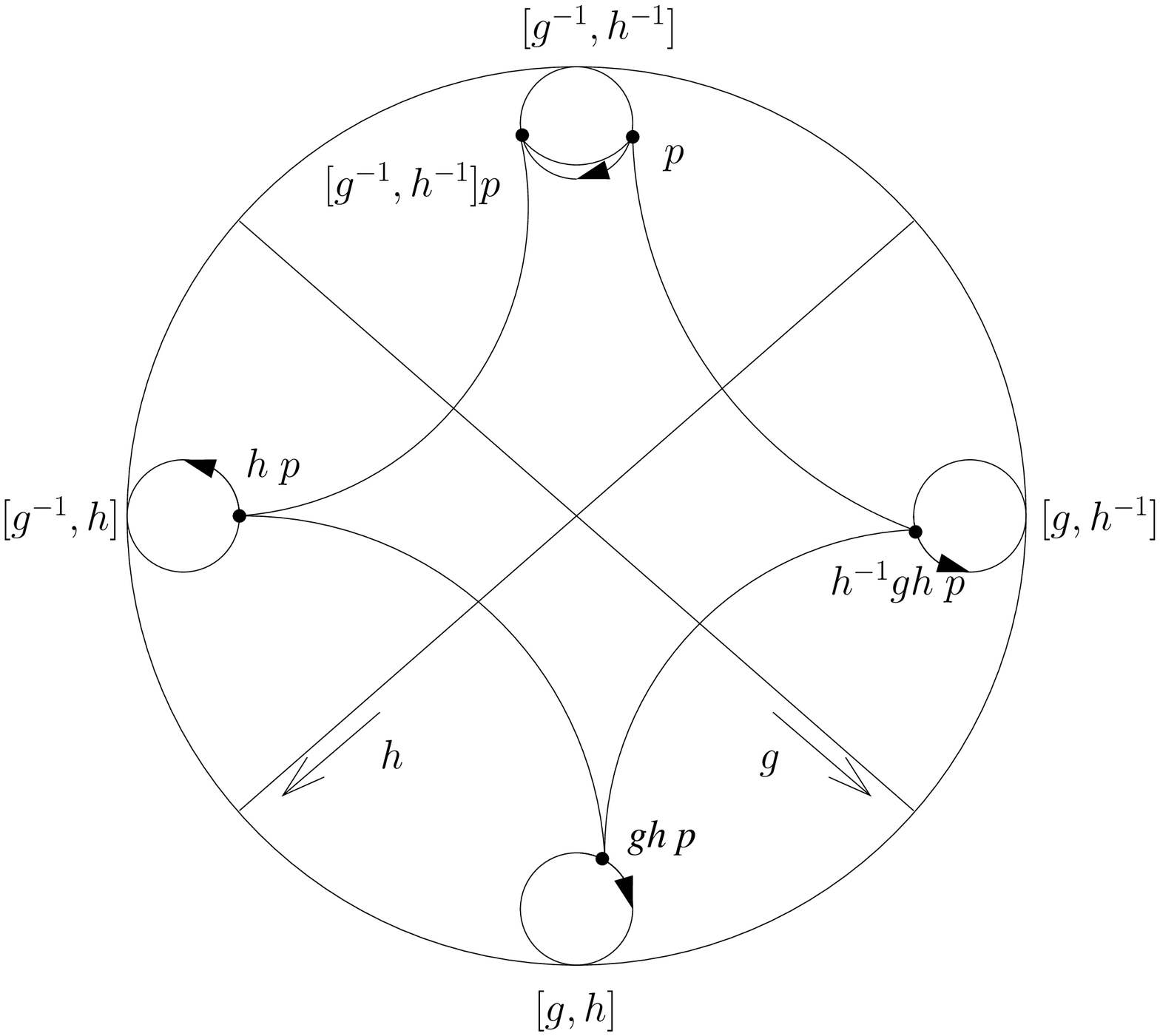}
\caption{The situation when $\Tr[g,h] = -2$.} \label{fig:12}
\end{figure}

Let $r = \Fix [g^{-1},h^{-1}]$. Then $h r = \Fix [g^{-1}, h]$, $g h r = \Fix [g, h]$, and $h^{-1}g h r = \Fix [g, h^{-1}]$. Obviously $[g^{-1},h^{-1}] r = r$. So taking $p=r$ gives $\Pent(g,h;p)$ an ideal quadrilateral, with degenerate ``boundary" edge, and bounds an embedded disc.

This gives a complete hyperbolic structure on $S$, where the boundary has become a cusp; $\rho$ is a discrete representation. Taking $p \in \hyp^2$
truncates this underlying surface and gives a cone-manifold structure on $S$ with no interior cone points and one corner point. As in the previous case, any basis $G,H$ will suffice.

Consider horocycles along which $[g^{-1},h^{-1}]$ translates. As $p$ approaches $\Fix [g^{-1},h^{-1}]$, since $[g^{-1},h^{-1}] \in \Par_1^- \cup \Par_{-1}^+$, $\Twist([g^{-1},h^{-1}],p)$ approaches $\pm 2\pi$. Thus by lemma \ref{pentagon_twist}, the corner angle $\theta$ is close to $\pi$. That is, the further out to infinity we choose $p$, the ``flatter" the corner angle obtained. 

The same argument regarding orientations as in the previous case gives the following proposition.
\begin{prop}
\label{orientation_2}
    Let $\rho$ be a representation and $G,H$ a basis of $\pi_1(S)$
    with $\Tr[g,h] = -2$. Suppose $[g,h] \in \Par_1^-$ (resp.
    $\Par_{-1}^+$). Then $\rho$ is the holonomy of a hyperbolic cone-manifold structure on $S$ with no interior cone points and one
            corner point and $\partial S$, traversed in the direction homotopic to
            $[G,H]$, bounds $S$ on its left (resp. right).
            The axes of $g,h$ intersect in the manner shown in figure
            \ref{fig:12} (resp. the opposite manner).
            As $p$ approaches the fixed point at infinity of $[g^{-1},h^{-1}]$, $\Twist([g^{-1},h^{-1}],p)$ approaches
            $2\pi$ from below (resp. $-2\pi$ from above).
            The
            corner angle is given by $\theta = 3\pi - \Twist([g,h],p)$
            (resp. $3\pi + \Twist([g,h],p)$).
\qed
\end{prop}

\subsection{The case $\Tr[g,h] \in (-2,2)$}
\label{sec:-2_to_2}

By lemma \ref{elliptic_commutator} the fixed points of $[g,h]$, $[g^{-1},h]$, $[g,h^{-1}]$ and $[g^{-1}, h^{-1}]$
are as shown in figure \ref{fig:13}. Let $r = \Fix [g^{-1},h^{-1}]$.

\begin{figure}[tbh]
\centering
\includegraphics[scale=0.4]{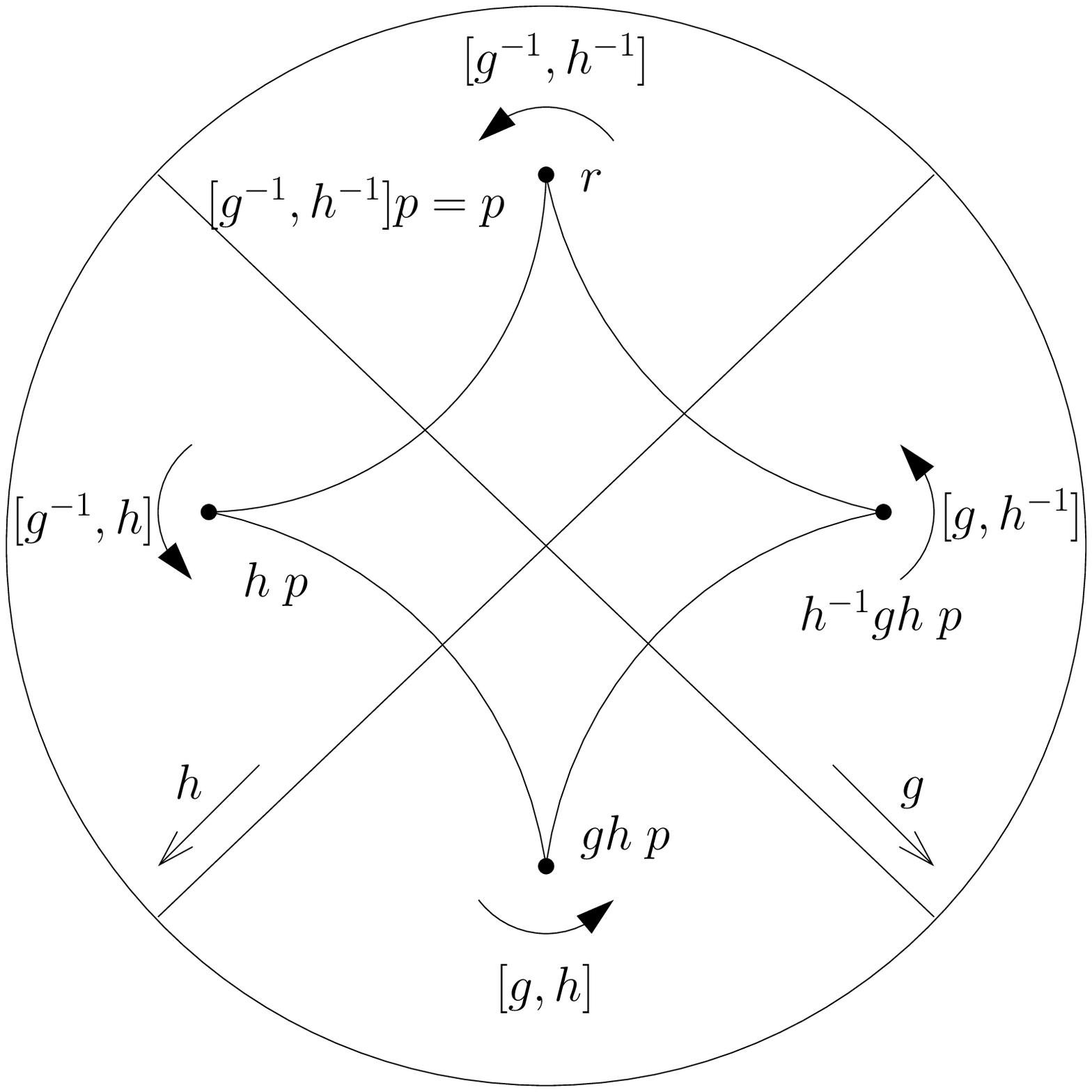}
\caption{The situation when $\Tr[g,h] \in (-2,2)$ and $p=r$.}
\label{fig:13}
\end{figure}

Letting $p = r$, $\Pent(g,h;p)$ degenerates to a quadrilateral, and $\rho$ is the holonomy of a hyperbolic cone-manifold structure on a (non-punctured) torus, with a single cone point. We perturb $p$ away from the fixed point $r$, in a direction so that $\Pent(g,h;p)$ bounds an embedded or immersed disc. For sufficiently small $\epsilon$, consider a small circle $C_\epsilon(r)$ of radius $\epsilon$ about $r$, and ask: for which $p \in C_\epsilon(r)$ does $\Pent(g,h;p)$ bound an embedded or immersed disc?

First a remark about orientation. From corollary \ref{commutator_trace_info} we know $[g,h] \in \Ell_{-1} \cup \Ell_1$. In the situation of figure \ref{fig:13}, a unit vector chase shows $\Twist([g^{-1},h^{-1}],r) > 0$, so by proposition \ref{twist_bounds} $[g,h], [g^{-1},h^{-1}] \in \Ell_1$. If the axes of $g,h$ intersect in the opposite manner then $\Twist([g^{-1},h^{-1}],r) < 0$ and $[g^{-1}, h^{-1}] \in \Ell_{-1}$. We will treat the case shown, i.e.
$[g,h] \in \Ell_1$; the other case is mirror reversed.

The gist of the idea is that, with $r, hr, h^{-1}ghr$ drawn as in figure \ref{fig:14}, with $hr$ to the left and $h^{-1}ghr$ to the right, we choose $p_0$ so that it is ``more right'' and $[g^{-1}, h^{-1}]p_0$ is ``more left'' (these ``directions'' are only to be taken in a vague sense). From proposition \ref{twist_bounds} we have $\Twist([g^{-1},h^{-1}],r) \in (0, 2\pi)$. The details work out somewhat differently if $\Twist([g^{-1},h^{-1}],r) \in (0, \pi]$ or $\Twist([g^{-1},h^{-1}],r) \in [\pi, 2\pi)$, and so we treat these two cases separately.

First assume $\Twist([g^{-1},h^{-1}],r) \in [\pi, 2\pi)$. Let $\varphi = 2\pi - \Twist([g^{-1},h^{-1}],r)$, so that $\varphi \in (0, \pi]$ and for $p$ on $C_\epsilon(r)$, the angle $\angle p r ([g^{-1},h^{-1}]p) = \varphi$, as shown in figure \ref{fig:14}. Let $\alpha$ denote the angle $\angle (h r)r(h^{-1}gh r)$. As $p$ moves around $C_\epsilon(r)$, its images under $h$, $gh$, $h^{-1}gh$, $[g^{-1}, h^{-1}]$ move around $C_\epsilon(h r)$, $C_\epsilon(g h r)$, $C_\epsilon(h^{-1}gh r)$, $C_\epsilon (r)$ respectively with the same angular velocity.

\begin{figure}[tbh]
\centering
\includegraphics[scale=0.4]{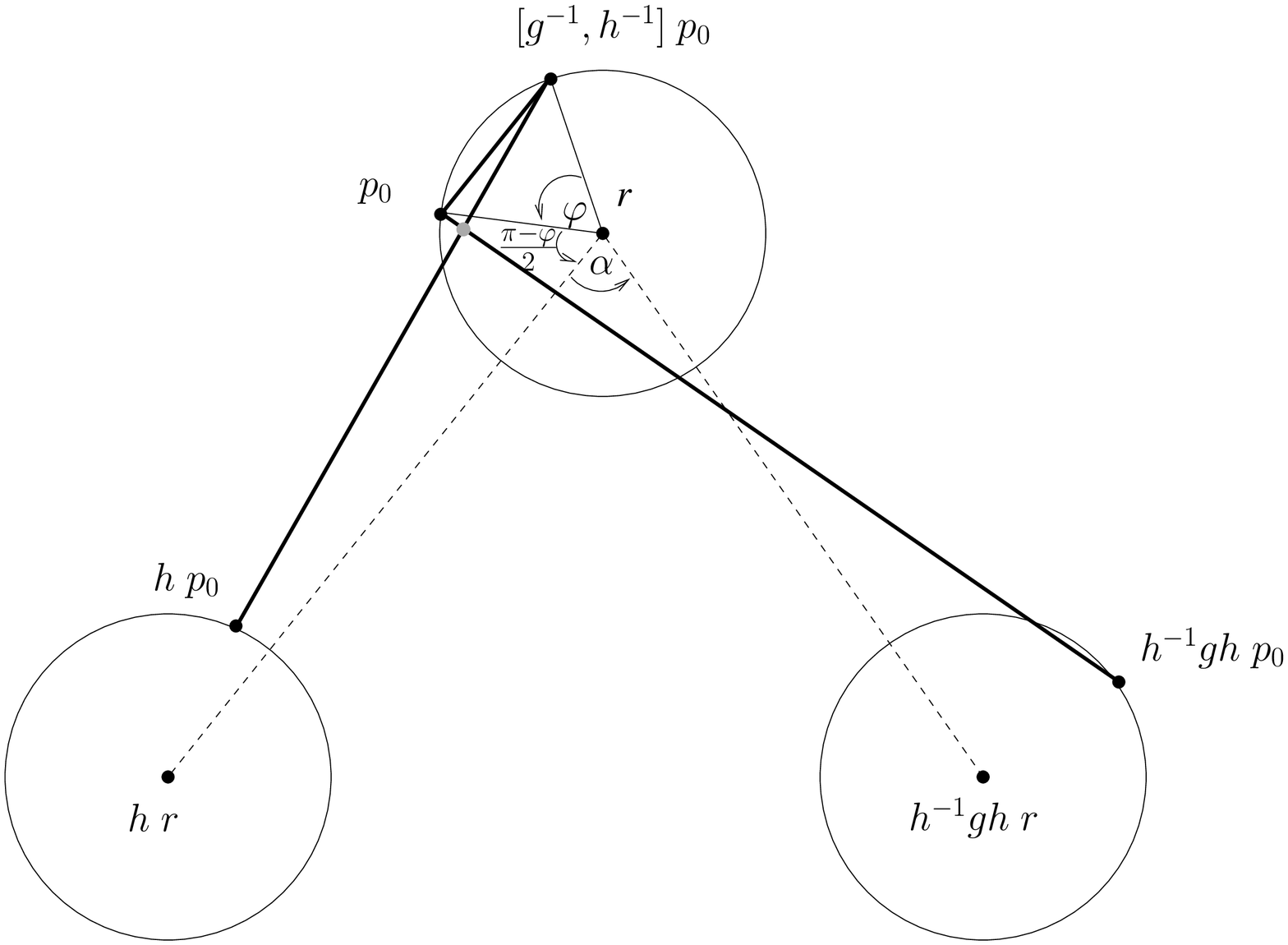}
\caption[The situation in $\Pent(g,h;p)$ for $p \in
C_\epsilon(r)$]{The situation in $\Pent(g,h;p)$ for $p \in
C_\epsilon(r)$. (These lines are all hyperbolic geodesics.)}
\label{fig:14}
\end{figure}

We rotate $p$ around $C_\epsilon(r)$ to the point $p_0$ lying $\frac{\pi-\varphi}{2} \in [0, \pi/2)$ clockwise of the point where $C_\epsilon(r)$ intersects the geodesic segment $r \To h r$, as in figure \ref{fig:14}. It follows that $p_0$ and $[g^{-1},h^{-1}]p_0$ both lie the same perpendicular (hyperbolic) distance from the line through $r$ and $hr$.

We claim that, while for this $p_0$ the pentagon $\Pent(g,h;p_0)$ is self-intersecting (as shown), for any $p$ lying anticlockwise of and close to $p_0$, we may take $\epsilon$ sufficiently small so that $\Pent(g,h;p)$ is simple (i.e. non-self-intersecting).

It is clear that most sides of $\Pent(g,h;p)$ pose no problem for simplicity; to show $\Pent(g,h;p)$ is simple it is sufficient to show that the segment
$[g^{-1},h^{-1}]p \To hp$ does not intersect $h^{-1}ghp \To p$. Consider the heights of various points with respect to the line $r \To hr$. It is sufficient to show that, in the arrangement of figure \ref{fig:15}, the segment $hp \To [g^{-1},h^{-1}]p$ lies above the segment $p \To [g^{-1},h^{-1}]p$. (The point $h^{-1}ghp$ lies far below $r \To hr$.) But for $p$ anticlockwise of $p_0$, by definition $p$ is lower than $[g^{-1},h^{-1}]p$ with respect to $r \To h r$. By taking $\epsilon$ sufficiently small, the segment $[g^{-1},h^{-1}]p \To h p$ can be made arbitrarily flat, rising by a height at most $2\epsilon$ over some fixed distance; and then it will lie above the segment $p \To [g^{-1},h^{-1}]p$ as required.

\begin{figure}[tbh]
\centering
\includegraphics[scale=0.4]{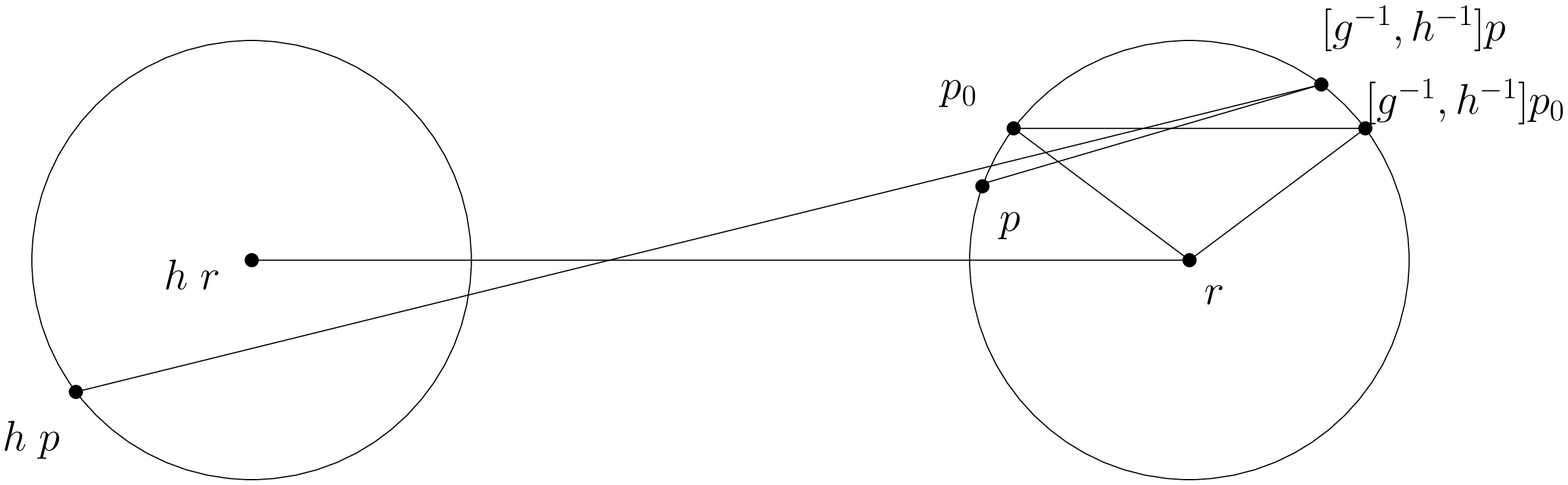}
\caption{Points in $\Pent(g,h;p)$ relative to the line $h r \To r$}
\label{fig:15}
\end{figure}

By a similar argument, we may rotate $p$ anticlockwise until $[g^{-1},h^{-1}]p$ lies $\frac{\pi-\varphi}{2}$ anticlockwise past the intersection of $C_\epsilon(r)$ with the segment $r \To h^{-1} gh r$. While $\Pent(g,h;p)$ is not simple for this $p=p_0$, for any $p$ up to this point, we may take $\epsilon$ sufficiently small so that $\Pent(g,h;p)$ is simple.

Thus we have found an open arc of angle $\pi+\alpha$ of directions from $r$, and for each direction there exists $\epsilon$ such that perturbing $p$ in this direction by a distance less than $\epsilon$ gives $\Pent(g,h;p)$ non-degenerate and simple: see figure \ref{fig:16}.

\begin{figure}[tbh]
\centering
\includegraphics[scale=0.4]{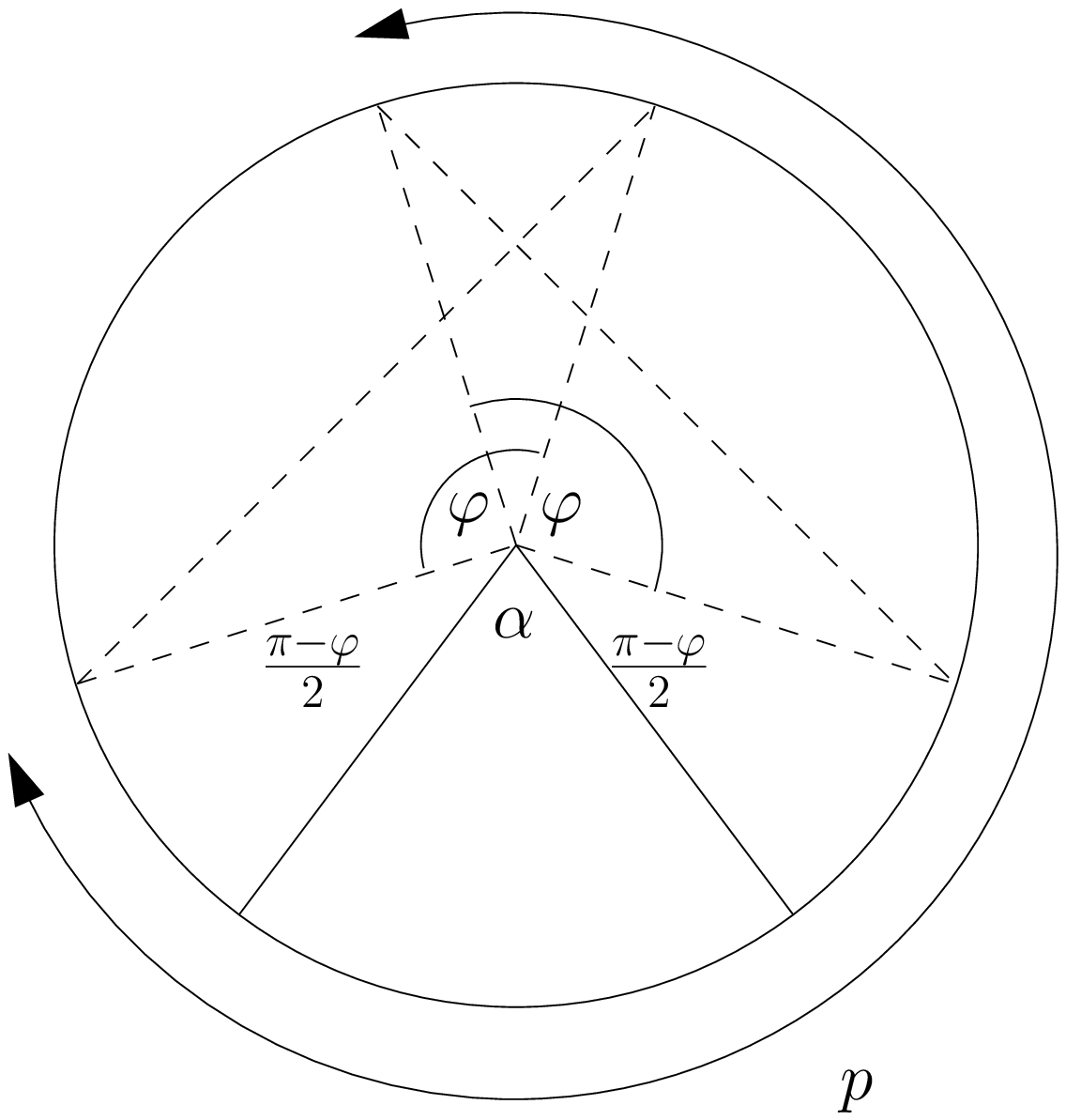}
\caption{Directions $p$ may be perturbed when
$\Twist([g^{-1},h^{-1}],r) \in [\pi, 2\pi)$.} \label{fig:16}
\end{figure}

In particular, there is a \emph{closed} arc of angle $\pi$ in which $p$ may be chosen such that $\Pent(g,h;p)$ is simple; and then by compactness we may choose an $\epsilon$ uniformly.

Next we consider the case $\Twist([g^{-1},h^{-1}],r) \in (0, \pi]$; the argument is similar. Let $\varphi = \Twist([g^{-1},h^{-1}],r)$, so $\varphi \in (0, \pi]$. Again we rotate $p$ around $C_\epsilon(r)$ and want $\Pent(g,h;p)$ simple. Rotate $p$ to the point $p_0$ where $[g^{-1},h^{-1}]p_0$ lies $\frac{\pi-\varphi}{2}$ clockwise of the intersection of $C_\epsilon(r)$ with the segment $r \To h^{-1}gh r$: see figure \ref{fig:17}. For $p$ clockwise of $p_0$, considering heights of points with respect to the line $r \To h^{-1}ghr$, $\epsilon$ may be taken sufficiently small so that $\Pent(g,h;p)$ is simple. And similarly for all $p$ clockwise up to $p_1$, lying $\frac{\pi-\varphi}{2}$ anticlockwise of the intersection point of $C_\epsilon(r)$ and the segment $r \To hr$. Again we obtain an open arc of angle $\pi+\alpha$ in which $\Pent(g,h;p)$ can be made simple: see figure \ref{fig:18}. And again there is a closed arc of angle $\pi$, and a uniform $\epsilon$, giving good choices for $p$.

\begin{figure}[tbh]
\centering
\includegraphics[scale=0.4]{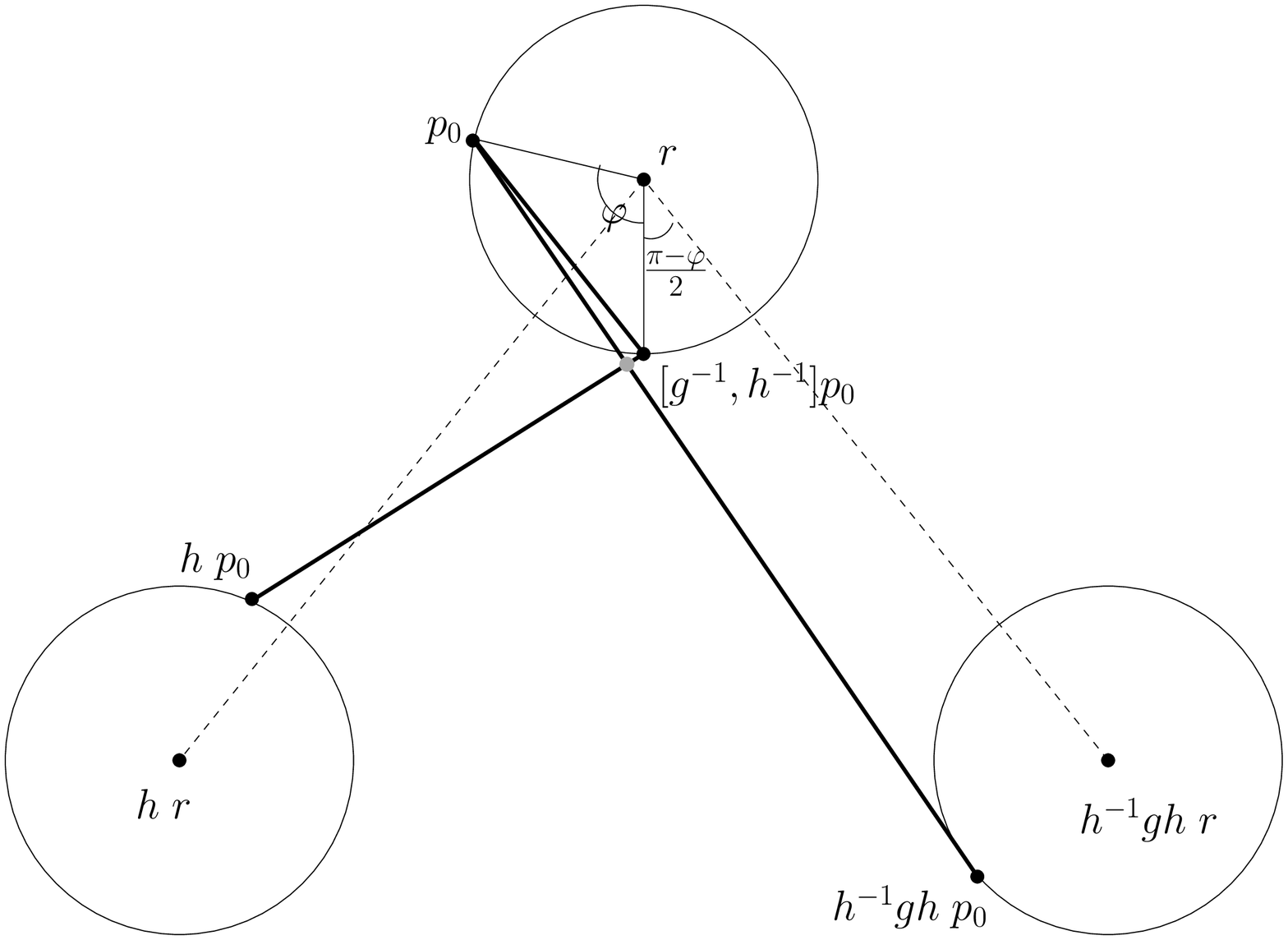}
\caption{The situation of $\Pent(g,h;p)$ in the case
$\Twist([g^{-1},h^{-1}],r) \in (0, \pi]$. Note $\Pent(g,h;p)$ may or may not contain $r$ in its interior, depending on the position of $p$.} \label{fig:17}
\end{figure}

\begin{figure}[tbh]
\centering
\includegraphics[scale=0.4]{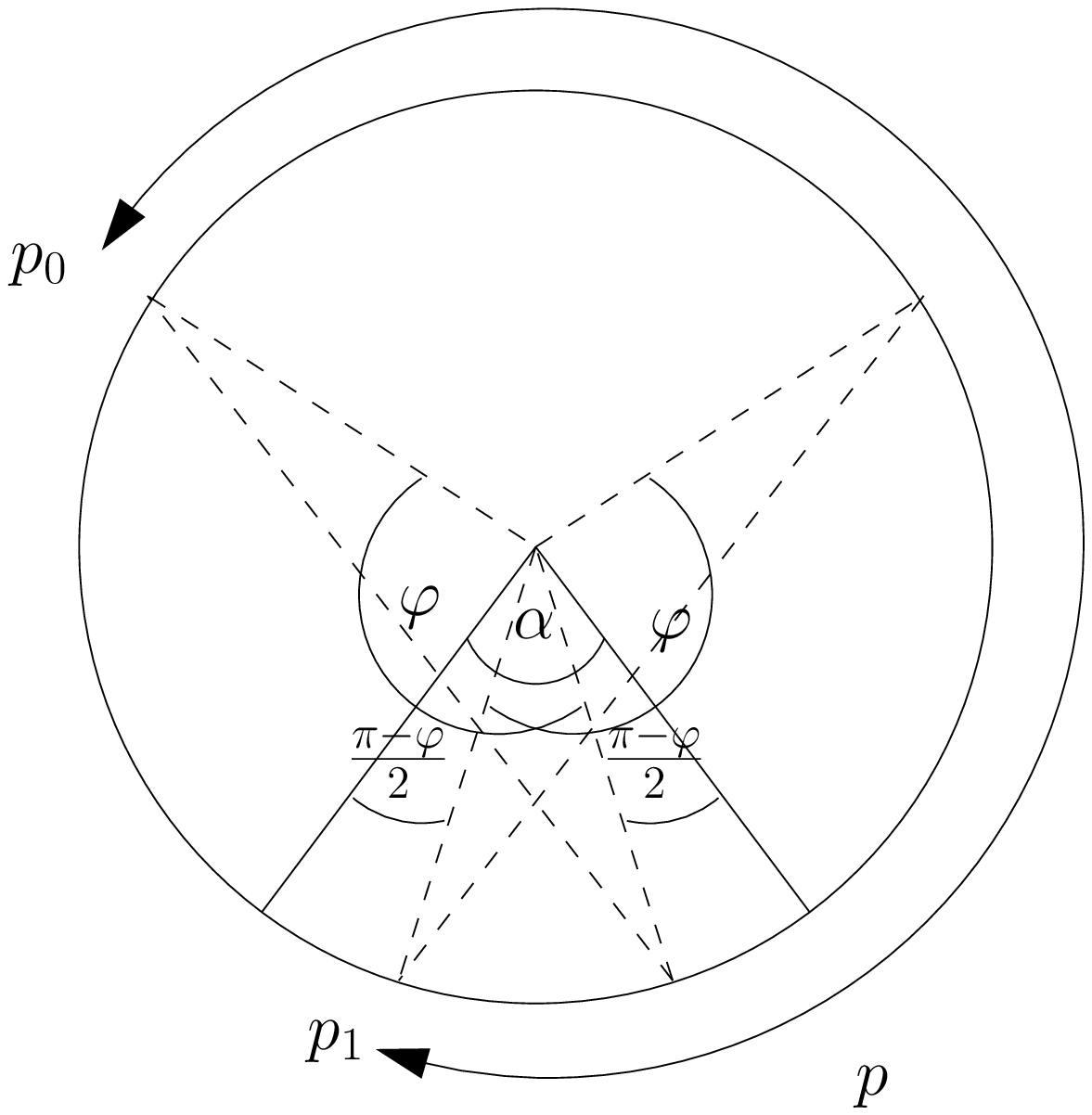}
\caption{Directions $p$ may be perturbed when
$\Twist([g^{-1},h^{-1}],r) \in (0, \pi]$.} \label{fig:18}
\end{figure}

Note the flexibility in choice of $p$: there is a closed semicircular disc of radius
$\epsilon$ with centre $r$ in which $p$ may be chosen arbitrarily
(except that $p \neq r$!). Note also that the above works for any basis $G,H$.

We can calculate the corner angles obtained, as previously; by lemma \ref{pentagon_twist} it is $\theta = 3\pi - \Twist([g^{-1},h^{-1}],p)$ in either of the above two cases. Thus $\theta$ is either less or more than $2\pi$ respectively in these cases; whether the corner angle is large or small is inherent in the rotation angle of $[g,h]$. We record our conclusions.

\begin{prop}
\label{elliptic_construction}
    Let $\rho: \pi_1(S) \To PSL_2\R$ be a representation with
    $\Tr[g,h] \in (-2,2)$. Suppose $[g,h] \in \Ell_1$ (resp.
    $\Ell_{-1}$). Let $r$ denote the fixed point of $[g,h]$. Then
    there exists a closed semicircular disc $D_\epsilon(r)$
    with centre $r$ such
    that if $p$ is chosen anywhere in this disc, except $r$, then
    $\Pent(g,h;p)$ is simple and non-degenerate, giving a hyperbolic
    cone-manifold structure on $S$ with no cone points and one
    corner point of angle $\theta$. The boundary $\partial S$,
    traversed in the direction $[G,H]$ bounds $S$ on its left (resp. right).
    The corner angle $\theta = 3\pi - \Twist([g^{-1},h^{-1}],p)$ (resp. $3\pi + \Twist([g^{-1},h^{-1}],p)$) lies in $(\pi, 3\pi)$; it lies in $(\pi, 2\pi]$ or $[2\pi, 3\pi)$ accordingly as the rotation angle of $[g^{-1},h^{-1}]$ lies in $[\pi, 2\pi)$ or $(0, \pi]$ (resp. $(-2\pi,-\pi]$ or $[-\pi,0)$).
    \qed
\end{prop}

\subsection{The case $\Tr[g,h] = 2$: reducible representations}
\label{sec:2}

By proposition \ref{reducible representation}, $\rho$ is reducible precisely when $\Tr[g,h] = 2$. Thus abelian representations are reducible. By lemma
\ref{reducible_va_abelian}, reducible virtually abelian representations are abelian. We will show that abelian representations do not give cone-manifold structures of the desired type; and we will show that the reducible non-abelian (hence not virtually abelian) representations do give cone-manifold structures of the desired type.

\begin{lem}
    An abelian representation is not the holonomy of any hyperbolic cone
    manifold structure on $S$ with no interior cone points and at most
    one corner point.
\end{lem}

\begin{proof}
    Let $\rho$ be abelian. So for any basis $G,H$ of $\pi_1(S)$ (with
    basepoint on $\partial S$), $g,h$ commute. Hence for any $p \in \hyp^2$, $p = [g^{-1},h^{-1}]p$, so
    $\Pent(g,h;p)$ has a degenerate boundary edge. By lemma
    \ref{construction_lemma}, $\rho$ is not the holonomy of any such
    cone-manifold structure.
\qed
\end{proof}

Now consider $\rho$ non-abelian and reducible; we construct a hyperbolic cone-manifold structure. Lemma \ref{reducible_description} describes the situation: there is a parabolic/hyperbolic case, and a hyperbolic/hyperbolic case. 

First consider the parabolic/hyperbolic case. After possibly reordering $G,H$ and replacing them with their inverses, we may conjugate and assume that in the upper half-plane model $g(z) = z+1$ and $h(z) = ez$, where $e>1$. Let $p = x + iy = (x,y) \in \hyp^2$. The vertices of $\Pent(g,h;p)$ are: $p = (x,y)$; $hp = \left( ex, ey \right)$; $ghp = \left( ex + 1, ey \right)$; $h^{-1}ghp = \left( x + \frac{1}{e}, y \right)$; and $[g^{-1}, h^{-1}] p = \left( x + \frac{1}{e} - 1, y \right)$. We obtain the situation of figure \ref{fig:19}. For \emph{any} choice of $p \in \hyp^2$, the pentagon $\Pent(g,h;p)$ is non-degenerate and bounds an embedded disc. So by lemma \ref{construction_lemma} we have a desired hyperbolic cone-manifold structure.

\begin{figure}[tbh]
\begin{center}
$\begin{array}{c}
\includegraphics[scale=0.4]{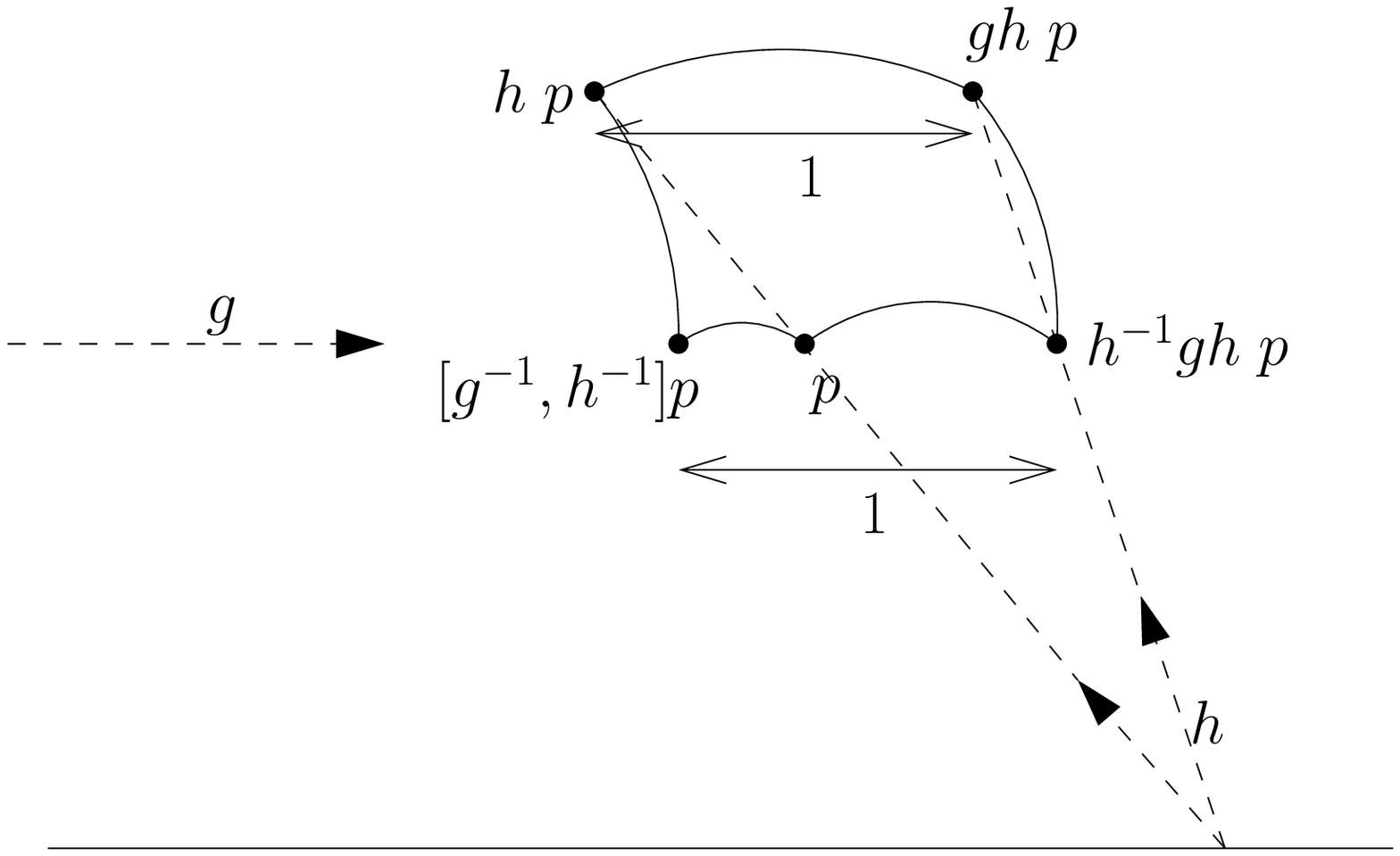}
\end{array}$
\caption{$\Pent(g,h;p)$ when $\Tr[g,h] = 2$, non-abelian, $g$ parabolic.} \label{fig:19}
\end{center}
\end{figure}

Now assume both $g,h$ are hyperbolic, with precisely one common fixed point. Again after possibly reordering and replacing $G,H$ with their inverses, we may conjugate and assume $g(z) = az$ and $h(z) = e(z+1)$, where $a,e > 1$. The fixed points at infinity of $h$ are then $\{\frac{e}{1-e}, \infty\}$.
Let $p = (x,y)$; we compute $h p = \left( e(x+1), ey \right)$, $gh p = \left( ae(x+1), aey \right)$, $h^{-1} gh p = \left( a(x+1)-1, ay \right)$, $[g^{-1},h^{-1}] p = \left( x + 1 - \frac{1}{a}, y \right)$.

Under our assumptions, $\frac{e}{1-e} < 0$, so we have a situation as in figure \ref{fig:20}. Choosing $p$ to lie above the fixed point $\frac{e}{1-e}$ of $h$, i.e. $x = \frac{e}{1-e}$, we see that $h p$ lies directly above $p$, along the (Euclidean and hyperbolic) line $\frac{e}{1-e} \to p$. Then $gh p$ lies above $h p$, along the Euclidean line $0 \to h p$; and $h^{-1} gh p$ lies below $gh p$, in the Euclidean segment $\frac{e}{1-e} \to gh p$. In particular, the line through $\frac{e}{1-e}, p, h p$ splits the plane with $gh p, h^{-1} gh p$ on its left, with the four hyperbolic segments $p \to h p \to gh p \to h^{-1} gh p \to p$ forming a non-degenerate simple quadrilateral. To show that $\Pent(g,h;p)$ is simple it is sufficient that $[g^{-1},h^{-1}]p$ lies right of the line $\frac{e}{1-e} \to p \to h p$. But $p$ and $[g^{-1},h^{-1}]p$ lie at the same height, so it is sufficient that $[g^{-1},h^{-1}]p$ lies to the right of $p$, i.e. $1 - \frac{1}{a} > 0$, which is true since $a>1$. Hence $\Pent(g,h;p)$ bounds an embedded disc and we have our cone-manifold structure.

\begin{figure}[tbh]
\centering
\includegraphics[scale=0.4]{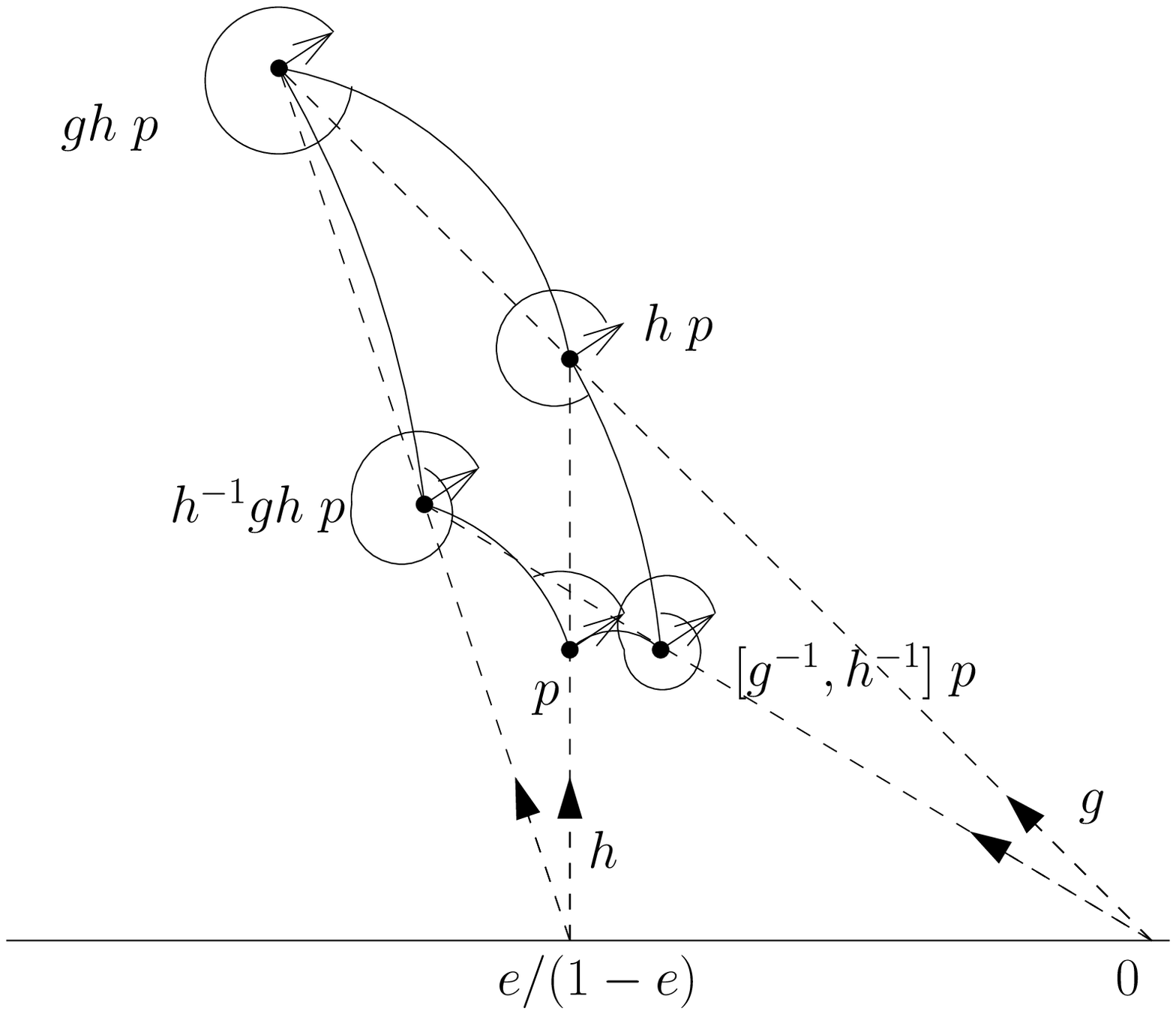}
\caption[$\Pent(g,h;p)$ when $\Tr{[g,h]} = 2$, non-abelian, $g,h$
hyperbolic]{$\Pent(g,h;p)$ when $\Tr[g,h] = 2$,
non-abelian, $g,h$ hyperbolic. A unit vector chase shows
$\theta > 2\pi$.} \label{fig:20}
\end{figure}

Note that almost any basis is good enough; at most we reordered the basis or replaced them with inverses. And there is freedom in the choice of
$p$ also: completely arbitrary in the parabolic/hyperbolic case; in the hyperbolic/hyperbolic case, $p$ can be placed arbitrarily along a certain line in
$\hyp^2$. Certainly $p$ can be chosen arbitrarily close to $\Fix [g^{-1},h^{-1}]$.

In both cases, $p \to [g^{-1},h^{-1}]p$ bounds $\Pent(g,h;p)$ on its right, iff $\partial S$ traversed in the
direction of $[G,H]$ bounds $S$ on its right, iff $[g^{-1},h^{-1}]$ is parabolic, fixing $\infty$, translating to the
left, i.e. $[g,h] \in \Par_0^-$.

As for the corner angle $\theta$, we may perform a unit vector chase and obtain $\theta \in (2\pi, 3\pi)$ for the above constructions: see e.g. figure \ref{fig:20}. Applying lemma \ref{pentagon_twist}, we obtain $\theta = 3\pi \pm \Twist([g^{-1},h^{-1}],p)$ depending on the orientation of $S$.

\begin{prop}
\label{reducible_construction}
    Let $\rho: \pi_1(S) \To PSL_2\R$ be a representation with
    $\Tr[g,h] =2$ for some basis $G,H$ of $\pi_1(S)$. Then $\rho$
    is the holonomy of a hyperbolic cone
    manifold structure on $S$ with no cone points and at most one
    corner point if and only if $\rho$ is not virtually abelian, i.e. $[g,h]
    \in \Par_0$. A fundamental domain for the developing map is
    given by $\Pent(g',h';p)$ where $(G',H')$ is obtained from
    $(G,H)$ at most by reordering and replacing with inverses.
    Suppose $[g',h'] \in \Par_0^+$ (resp. $\Par_ 0^-$). The
    point $p$ may be chosen arbitrarily close to the fixed point at
    infinity of $[g'^{-1},h'^{-1}]$.
    Then the boundary $\partial S$, traversed in the direction of
    $[G,H]$, bounds $S$ on its left (resp. right).  The corner angle
    $\theta = 3\pi - \Twist([g^{-1},h^{-1}],p)$ (resp. $3\pi +
    \Twist([g^{-1},h^{-1}],p)$) lies in $(2\pi, 3\pi)$.
    \qed
\end{prop}

\subsection{The case $\Tr[g,h] > 2$}
\label{sec:more_than_2}

We now come to the most difficult case. This case includes virtually abelian representations. The abelian representations all belong to the case $\Tr[g,h] = 2$; by lemma \ref{virtually abelian but not abelian}, the representations which are virtually abelian but not abelian are precisely those with $(\Tr g, \Tr h, \Tr gh) \in V$, in the notation of section \ref{virtually abelian reps}, and hence $\Tr[g,h] > 2$. Our proof is in the following three subsections, which respectively prove the following three results.

\begin{prop}
\label{virtually_abelian_not_holonomy}
    Let $\rho: \pi_1(S) \To PSL_2\R$ be a representation which
    is virtually abelian but not abelian.
    Then $\rho$ is not the holonomy of any hyperbolic cone
    manifold structure on $S$ with no interior cone points
    and at most one corner point.
\end{prop}

\begin{prop}
\label{complicate_basis}
    Let $G,H$ be a basis of $\pi_1(S)$ and let $\rho: \pi_1(S) \To PSL_2\R$ be a representation
    with $\Tr[g,h] > 2$ which is not virtually abelian. Then there
    exists a basis $G',H'$ of $\pi_1(S)$ such that
    \[
        (x,y,z) = (\Tr g', \Tr h', \Tr g'h') \in
        (2, \infty)^3.
    \]
\end{prop}

\begin{prop}
\label{explicit_construction}
    Let $\rho: \pi_1(S) \To PSL_2\R$ be a representation
    which is not virtually abelian, and suppose
    there exists a basis $G,H$ of $\pi_1(S)$ such that $\Tr
    [g,h] > 2$ and $(x,y,z) = (\Tr g, \Tr h, \Tr gh) \in (2, \infty)^3$.
    Then $\rho$ is the
    holonomy of a hyperbolic cone-manifold structure on
    $S$ with no interior cone points and at most one corner point.
\end{prop}

\subsubsection{Virtually abelian degeneration}

We prove proposition \ref{virtually_abelian_not_holonomy}. Let $\rho$ be a representation which is virtually abelian but not abelian, hence with character in $V$; in fact, by lemma \ref{va2}, with character in $V$ for any basis $G,H$ of $\pi_1(S)$. By lemma \ref{construction_lemma} then it suffices to prove the following.

\begin{lem}
    Let $g,h \in PSL_2\R$ such that $(\Tr g, \Tr h, \Tr gh) \in V$. Then for any $p \in \hyp^2$, the pentagon $\Pent(g,h;p)$ does not bound an immersed open disc in $\hyp^2$.
\end{lem}

\begin{proof}
    By lemma \ref{va1}, two of $\{g,h,gh\}$ are half-turns about distinct points $q_1, q_2 \in \hyp^2$, and the third is hyperbolic with axis $q_1 q_2$. There are three possible cases:
    \begin{enumerate}
        \item
            $g,h$ are half-turns, $gh$ is hyperbolic.
        \item
            $h,gh$ are half-turns, $g$ is hyperbolic.
        \item
            $g,gh$ are half-turns, $h$ is hyperbolic.
    \end{enumerate}
    In each case, all of $g,h,gh$ preserve the line $q_1 q_2$. As $[g,h]$ also preserves this line and $\Tr[g,h]>2$, $[g,h]$ is hyperbolic with axis $q_1 q_2$.

    \textbf{ Case (i).} If $p \in q_1 q_2$ then all vertices of $\Pent(g,h;p)$ lie on $q_1 q_2$ and the pentagon clearly cannot bound an immersed disc. Consider Fermi coordinates on $\hyp^2$ with axis $q_1 q_2$, and let $p = (\alpha, d)$. With these coordinates, $g(y,z) = (-y+a,-z)$ and $h(y,z) = (-y,-z)$, for some nonzero $a \in \R$. Then we compute $h p = (-\alpha, -d)$, $gh p = (\alpha + a,d)$, $h^{-1}gh p = (-\alpha - a, -d)$ and $[g^{-1},h^{-1}]p = (\alpha + 2a, d)$.

    Regardless of the signs of $\alpha$ and $a$, the point $gh p$ lies between $p$ and $[g^{-1},h^{-1}]p$ on the curve at height $d$ from $q_1 q_2$. Since $a \neq 0$ these three points are distinct. But $gh p$ lies on the opposite side of the geodesic segment $p \To [g^{-1},h^{-1}]p$ from the points $hp$ and $h^{-1}ghp$. It follows that $\Pent(g,h;p)$ does not bound an immersed disc. See figure \ref{fig:21}.

\begin{figure}[tbh]
\begin{center}
$\begin{array}{c}
\includegraphics[scale=0.35]{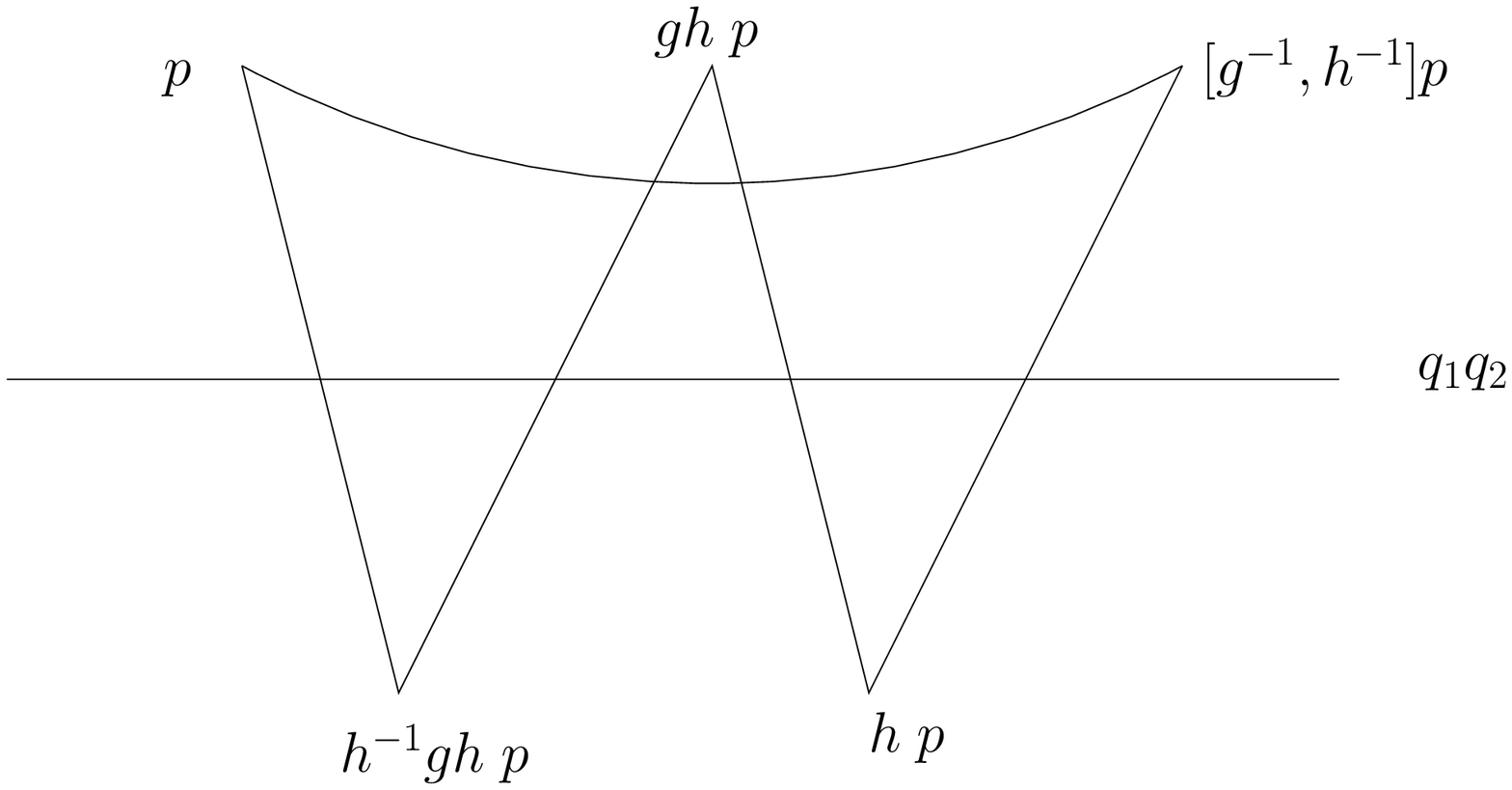}
\end{array}$
\hspace{1 cm}
$\begin{array}{c}
\includegraphics[scale=0.35]{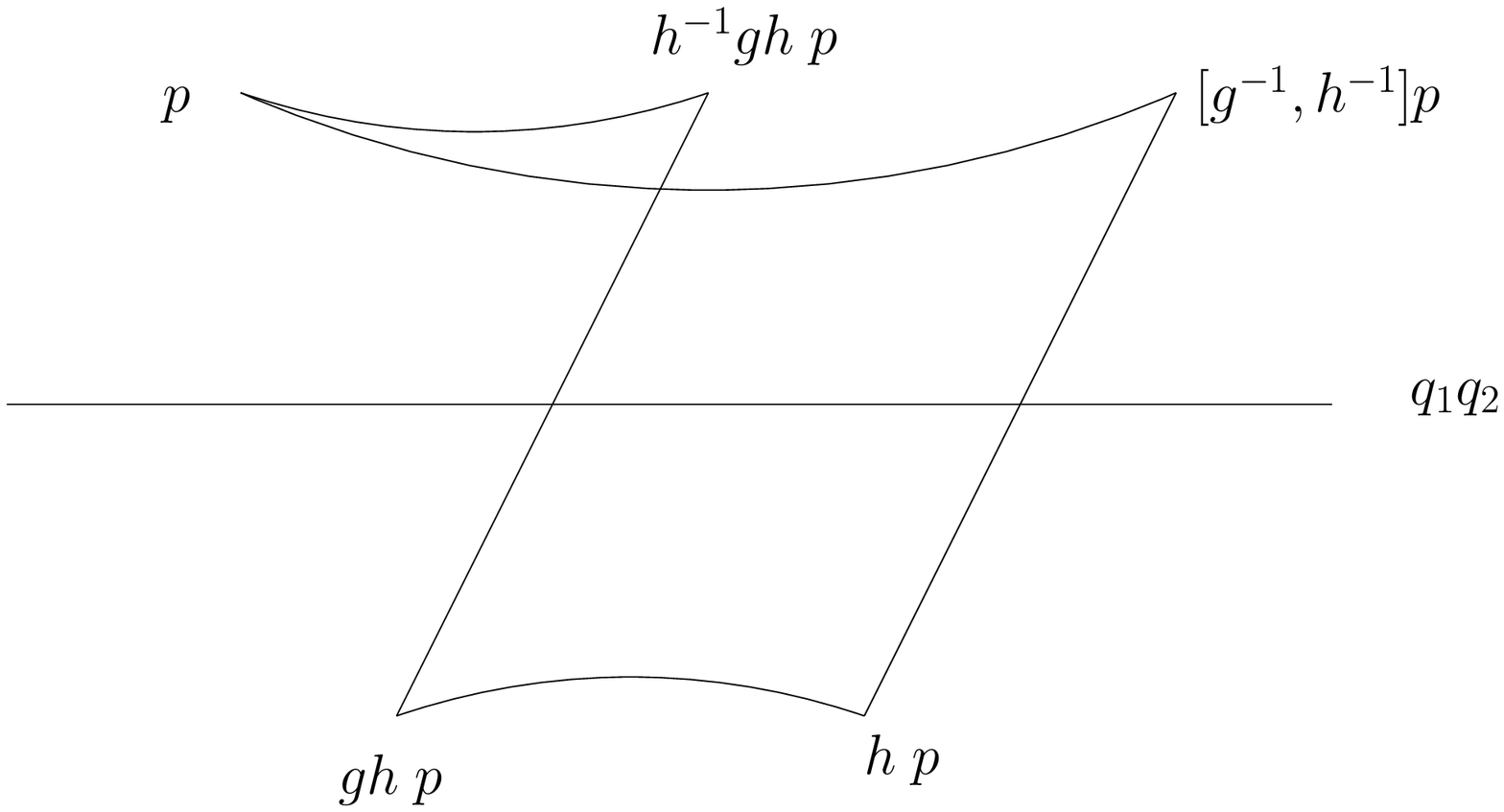}
\end{array}$
\caption{$\Pent(g,h;p)$ does not bound an immersed disc. Left: case (i). Right: case (ii).} \label{fig:21}
\end{center}
\end{figure}

    \textbf{ Case (ii).} Again take Fermi coordinates with axis $q_1 q_2$. We may assume that $g(y,z) = (y+c,z)$ and $h(y,z) = (-y,-z)$.
    Let $p = (\alpha, d)$; if $d=0$ $\Pent(g,h;p)$ lies on $q_1 q_2$ and cannot bound an immersed disc. We compute $hp = (-\alpha, -d)$, $ghp = (-\alpha + c,-d)$, $h^{-1}ghp = (\alpha - c,d)$, $[g^{-1}, h^{-1}] p = (\alpha - 2c,d)$. Now $h^{-1}ghp$ lies between $p$ and $[g^{-1},h^{-1}]p$ at height $d$. But $h^{-1}ghp$ lies on the opposite side of the geodesic segment $p \To [g^{-1},h^{-1}]p$ from $hp$ and $ghp$. Again $\Pent(g,h;p)$ cannot bound an immersed disc: see figure \ref{fig:21}.

    \textbf{ Case (iii).} This is similar to case (ii).
\qed
\end{proof}

\subsubsection{An algorithm to increase traces}
\label{sec:algorithm_to_increase_traces}

We now prove proposition \ref{complicate_basis}; so let $G,H$ be a basis and $\rho$ a non-virtually-abelian representation with $\Tr[g,h]>2$. Applying lemma \ref{virtually abelian but not abelian}, the character $(x,y,z)$ of $\rho$ has $(x,y,z) \in \kappa^{-1}(2, \infty) \backslash V$. We wish to change basis until $(x,y,z) \in (2, \infty)^3$.

We have fully investigated the effect of changes of basis on characters $(x,y,z) = (\Tr g, \Tr h, \Tr gh)$ in section \ref{action of modular group 2}; by proposition \ref{Markoff_moves}, proposition \ref{complicate_basis} is reduced to the following purely algebraic claim.

\begin{lem}
\label{complication_algorithm}
    Let $(x,y,z) \in \R^3$ satisfy $x^2 + y^2 + z^2 - xyz > 4$ and
    $(x,y,z) \notin V$. Then under the equivalence relation
    generated by permutations of coordinates and
    \[
        (x,y,z) \sim (x,y,xy-z), \quad (x,y,z) \sim (-x,-y,z),
    \]
    we have $(x,y,z) \sim (x',y',z')$ for some
    $(x',y',z') \in (2,\infty)^3$.
\end{lem}

We will give an algorithm to obtain such an $(x',y',z')$. This algorithm is essentially the opposite of the algorithm used by Goldman in \cite{Goldman03}; it is a greedy algorithm. We define the following subsets of $\R^3$, each to be treated separately.
\begin{align*}
    R_1 &= (2, \infty) \times (2, \infty) \times (2, \infty) \\
    R_2 &= (-\infty, -2) \times (2, \infty) \times (2, \infty) \\
    R_3 &= [-2,2] \times (2, \infty) \times (2, \infty) \\
    R_4 &= [0,2] \times [0,2] \times (2, \infty) \\
    R_5 &= [-2,0] \times [0,2] \times (2, \infty) \\
    R_6 &= [0,2] \times [0,2] \times [0,2]\\
    R_7 &= [-2,0] \times [0,2] \times [0,2]
\end{align*}
Since sign changes on two coordinates and permutations of coordinates are valid moves, we may reorder $(x,y,z)$ so that $|x| \leq |y| \leq |z|$; and then change signs until $y,z \geq 0$. This point lies in some $R_i$. Thus every point in $\R^3$ is equivalent to a point in $\cup R_i$. We will show that every point in $R_i$ for $2 \leq i \leq 7$ (and lying in $\kappa^{-1}(2,\infty) \backslash V$), is equivalent to a point in some $R_j$, for $j<i$. It follows that every point in $\kappa^{-1}(2,\infty) \backslash V$ is equivalent to a point in $R_1$, proving lemma \ref{complication_algorithm}.

We will always proceed by a greedy algorithm: permute coordinates so that $x \leq y \leq z$ and then apply the Markoff move $(x,y,z) \mapsto (yz-x,y,z)$. So it is worth examining this algebra first. Recall that
\[
    \kappa(x,y,z) = \Tr[g,h] = x^2 + y^2 + z^2 - xyz-2 > 2
\]
where $\kappa$ is invariant under any automorphism of the free
group; in particular under a change of basis $(g,h) \mapsto
(g^{-1},h)$. Letting $x'=yz-x$ be the number replacing $x$ after
the Markoff move is applied, we see that $x,x'$ are the roots of
the quadratic in $t$
\[
    t^2 - yzt + y^2 + z^2 - \kappa - 2 = 0
\]
where $\kappa > 2$ is a constant. Here we think of $y,z$ as
constants. The quadratic has discriminant $\Delta$ and roots $x,x'$ given by
\[
    \Delta = (y^2 - 4)(z^2 - 4) + \underbrace{4 \kappa - 8}_{>0}, \quad 
    x, x' = \frac{ yz \pm \sqrt{ \Delta } }{2},
\]
and turning point at $t=yz/2$. We now turn to each of the regions
$R_2$ through to $R_7$ in turn.

\begin{itemize}
\item 
{\bf The region $R_7$.} After possibly reordering coordinates we may assume $-2 \leq x \leq 0 \leq y \leq z \leq 2$. We now simply take $(x',y',z') = (yz-x,y,z)$, in which all coordinates are non-negative, so that $(x',y',z')$ (after reordering coordinates) lies in $R_4$ or $R_6$.

\item {\bf The region $R_6$.} After possibly reordering coordinates we may assume $0 \leq x \leq y \leq z \leq 2$. We need a technical lemma, which could be an undergraduate exercise.\footnote{The relevant undergraduate exercise is: ``minimise the function $f(x,y,z) = yz-x-y$ subject to the constraints $0 \leq x \leq y \leq z \leq 2$ and $x^2 + y^2 + z^2 - xyz - 4 \geq 0$''. The constraints define a connected compact region in $\R^3$ bounded by the surfaces $x=0$, $x=y$, $y=z$, $z=2$ and $x^2 + y^2 + z^2 - xyz - 4 = 0$, and the exercise is straightforward.

The referee gives the following more elegant argument. Suppose $0 \leq x \leq y \leq z \leq 2$ and $x^2 + y^2 + z^2 - xyz > 4$. Find $0 \leq p \leq q \leq \frac{\pi}{2}$ such that $z=2 \cos p$ and $y = 2 \cos q$. The quadratic condition on $x$ says that $x$ lies outside $[2 \cos (p+q), 2 \cos (p-q)]$ (by computing the sum and product of the bounds of this interval). In other words, $x = 2 \cos r$ for some $r \in (p+q, \frac{\pi}{2})$. To prove $yz-x>y$ it is therefore enough to check $2 \cos q \cdot 2 \cos p - 2 \cos (p+q) \geq 2 \cos q$, i.e. $2 \cos (q-p) \geq 2 \cos q$, which is clear since $0 \leq q-p \leq q \leq \frac{\pi}{2}$.}

\begin{lem}
    Suppose $0 \leq x \leq y \leq z \leq 2$, and $x^2 + y^2 + z^2
    - xyz > 4$. Then $yz-x > y$.
\qed
\end{lem}

Define inductively and greedily the sequence $(x_n, y_n, z_n)$ by setting $(x_0,y_0,z_0) = (x,y,z)$ and letting $(x_{n+1},y_{n+1},z_{n+1})$ be
the triple obtained by taking $\left\{ y_n z_n - x_n, y_n, z_n \right\}$ and reordering so $x_{n+1} \leq y_{n+1} \leq z_{n+1}$.

The lemma tells us that $y_n z_n - x_n \geq y_n$, so that all coordinates remain non-negative. At most one of $(x,y,z)$ can be zero: if two are zero then $\rho$ is virtually abelian; if three are zero we have a contradiction to $\kappa(x,y,z) > 2$. The first Markoff move makes all coordinates positive,
after which they remain positive and non-decreasing. The sum $x_n + y_n + z_n$ is strictly increasing, and in fact $(x_{n+1} + y_{n+1} + z_{n+1}) - (x_n + y_n + z_n)$ is given by the difference between the roots $x_n - x'_n$ of the quadratic described above, which is
\[
    \sqrt{\Delta} = \sqrt{ (y_n^2 - 4)(z_n^2 - 4) + 4 \kappa - 8} \geq 2\sqrt{\kappa - 2}.
\]
The inequality follows since, if one of $x_n,y_n,z_n$
becomes larger than 2 then our point lies in a different region
(namely $R_4$) and we have completed the argument. Otherwise
$y_n^2-4, z_n^2-4 \leq 0$ and their product is non-negative.

Thus the sum $x_n + y_n + z_n$ increases each iteration by at
least $2 \sqrt{\kappa - 2} > 0$. It follows that after a finite
number of steps this sum becomes larger than 6, and hence one of
the coordinates becomes larger than 2, moving our point into
$R_4$.

\item {\bf The region $R_5$.} Here $-2 \leq x \leq 0 \leq y \leq 2 < z$. We take $(x',y',z') =
(yz-x,y,z)$. Now all coordinates are non-negative and $z>2$ so that
$(x',y',z')$, after permuting coordinates to put them in ascending
order, lies in $R_4$ or $R_3$.

\item {\bf The region $R_4$.} After possibly permuting coordinates we may assume that $0 \leq x \leq y \leq 2 < z$. The character is virtually abelian iff $y=0$, so assume $y > 0$.

Applying the move $(x,y,z) \mapsto (x',y,z)$, where $x, x'=yz-x$ are the roots of the quadratic above, we see that $yz-x > 2y - x \geq y$ so that all coordinates remain non-negative at each stage and at least one coordinate is greater than 2. If two coordinates become greater than $2$, i.e. $yz-x > 2$, then we are in the region $R_3$. Otherwise, the new triple is, in ascending order, $(y,yz-x,z)$ where $0 \leq y < yz-x \leq 2 < z$.

We show that a finite number of these moves suffices to make two coordinates greater than 2, applying a greedy algorithm. Let $(x_0,y_0,z_0) = (x,y,z)$ and for $n \geq 0$ inductively let $(x_{n+1},y_{n+1},z_{n+1})$ be the triple obtained by taking $\left( y_n z_n - x_n, y_n, z_n \right)$ and ordering coordinates. From above we either enter $R_3$ or $(x_{n+1}, y_{n+1}, z_{n+1}) = (y_n, x_n', z_n)$.  The difference $x'_n - x_n = \sqrt{\Delta}$, where $\Delta$ is the discriminant of the appropriate quadratic; $\Delta>0$ since $x_n < x'_n$ are real roots. From the above paragraph we see $x_n < y_n$ strictly for $n \geq 1$, and hence $x_{n+1} = y_n > x_n = y_{n-1}$, i.e. $x_n$ and $y_n$ are strictly increasing for $n \geq 1$, and increasing for $n \geq 0$.

Clearly $z_n$ is constant while we remain in $R_4$; the other two coordinates are strictly increasing, and 
\[
 \left( x_{n+1} + y_{n+1} \right) - \left( x_n + y_n \right) = x_n' - x_n = \sqrt{\Delta} = \sqrt{ (y_n^2-4)(z_n^2-4) + 4 \kappa - 8 }.
\]
Now $0 < y_n \leq 2$ and $z>2$, so that the product $(y_n^2-4)(z_n^2-4)$ is negative. The factor $(z_n^2-4)$ is a positive constant, and the other factor $y_n^2-4$ increases towards $0$ as $y_n$ increases. Thus the product $(y_n^2-4)(z_n^2-4)$ increases with $n$ and $x_n' - x_n \geq \sqrt{(y_0^2-4)(z_0^2-4) + 4 \kappa - 8}$, which is positive as it is the discriminant of the quadratic with $x_0 \neq x_0'$ as roots. Thus $x_n + y_n$ increases by at least this amount each time. After a finite number of moves then $x_n + y_n > 4$, so at least one of $x_n,y_n$ becomes larger than $2$.

\item {\bf The region $R_3$.} Here we may assume $-2 \leq x \leq 2 < y < z$ after reordering. Now simply take $(x',y',z') = (yz-x,y,z)$. Clearly
$y,z>2$ and $x'=yz-x > 2 \times 2 - 2 = 2$. So $(x',y',z') \in R_1$.

\item {\bf The region $R_2$.} Applying a sign change manoeuvre, we have $x,y,z < -2$. Now we apply
a Markoff move and a sign change $(x,y,z) \mapsto (yz-x,y,z) \mapsto (yz-x,-y,-z)$. Clearly $-y,-z>2$ and $yz>4$, $-x>2$ imply $yz-x>6>2$. Thus
$(x',y',z') \in R_1$.

\end{itemize}

This concludes the proof of lemma \ref{complication_algorithm} and
hence proposition \ref{complicate_basis}. Note that in fact the change of basis can be taken to be
orientation-preserving. If necessary, we simply make the change of
basis $(G,H) \mapsto (H,G)$ say, which on $X(S)$ maps $(x,y,z)
\mapsto (y,x,z)$.

\subsubsection{Explicit construction}

We now have a basis $G,H$ of $\pi_1(S)$ such that $\left( \Tr g, \Tr h, \Tr gh \right) \in (2, \infty)^3$, so that $g,h,gh$ are hyperbolic isometries of $\hyp^2$. We will first explain the significance of the fact that all traces are greater than $2$.

From lemma \ref{axes crossing} we see that the axes of $g$ and $h$
are disjoint. Since $\Tr[g,gh] = \Tr[h,gh] = \Tr[g,h]$, the axes of
$g,h,gh$ are all disjoint. These axes cannot share a fixed point at
infinity either, for then $\Tr[g,h] = \pm 2$.

We will rely on results of Gilman and Maskit in
\cite{Gilman_Maskit}. Let $C(g,h)$ denote the cross ratio
\[
    C(g,h) = \frac{ (r_g-a_h)(a_g-r_h) }{ (r_g - r_h)(a_g-a_h) }
\]
in the upper half plane model (recall $r_g$ and $a_g$ denote repulsive and attractive fixed points of $g$). This quantity just tells us the orientation of the axes
of $g,h$ with respect to each other. (Note this is the reciprocal of
the definition in \cite{Gilman_Maskit}; but the definition in that
paper conflicts with their theorem; and certainly with their figure
2. Rewriting their definition of cross-ratio seems better than
rewriting their theorem.) If we normalise so that $r_g = 0$, $a_g = \infty$, $r_h = 1$ then $a_h = C(g,h)$.

\begin{figure}[tbh]
\centering
\includegraphics[scale=0.35]{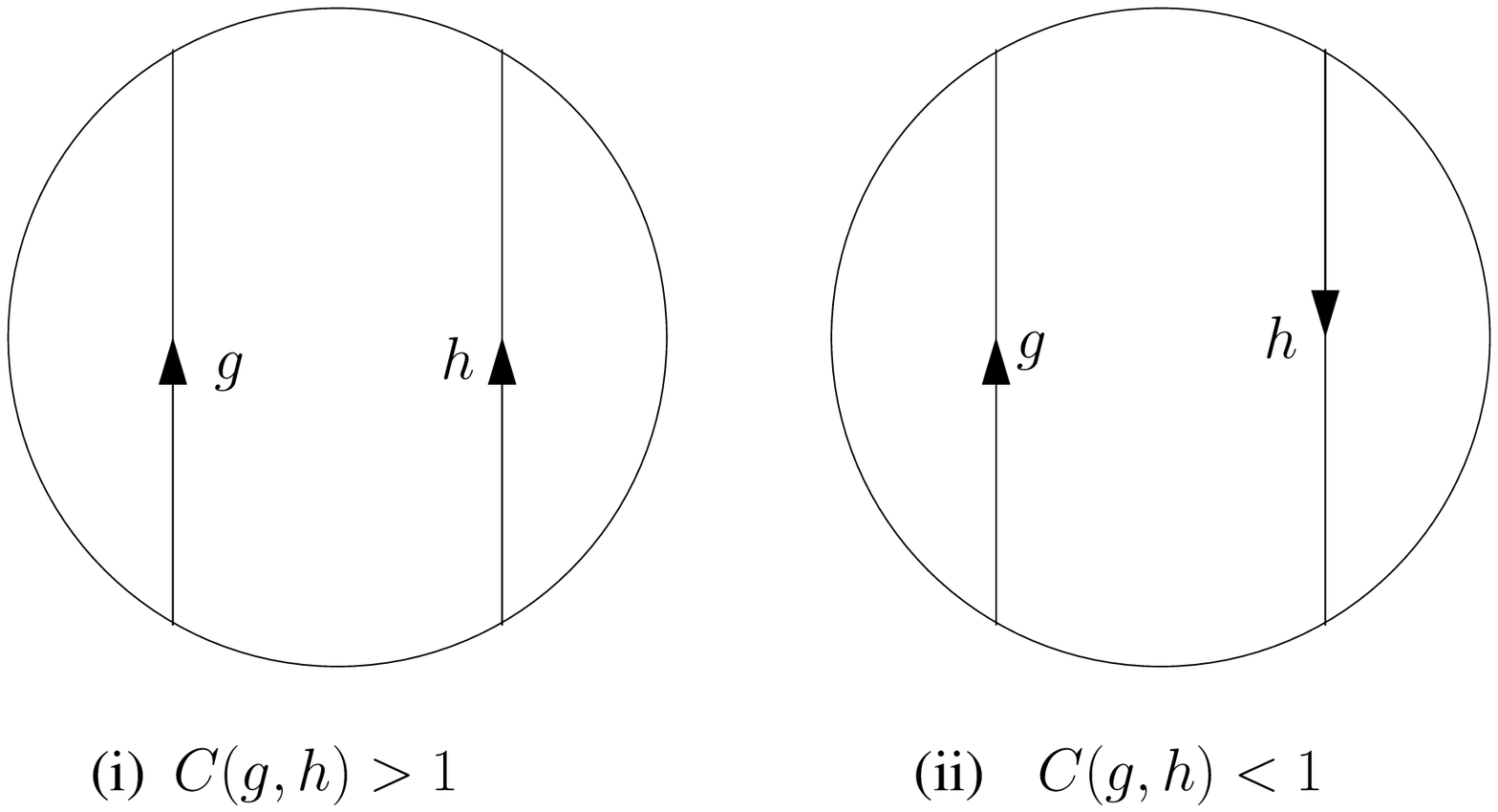}
\caption{Different respective orientations of axes of $g,h$.}
\label{fig:22}
\end{figure}

\begin{lem}
    Suppose $g,h \in PSL_2\R$ are hyperbolic and $\Tr[g,h]>2$. Then
    $C(g,h) \in (1, \infty)$ iff the axes of $g,h$ are oriented as
    in figure \ref{fig:22}(i), and $C(g,h) \in (0,1)$ iff the axes
    are oriented as in figure \ref{fig:22}(ii).
\end{lem}

\begin{proof}
    In the situation of figure \ref{fig:22}(i) we may project
    to the upper half plane as in figure \ref{fig:23}. With
    lengths along the real axis $\alpha, \beta, \gamma$ as labelled
    then we have
\begin{figure}[tbh]
\centering
\includegraphics[scale=0.5]{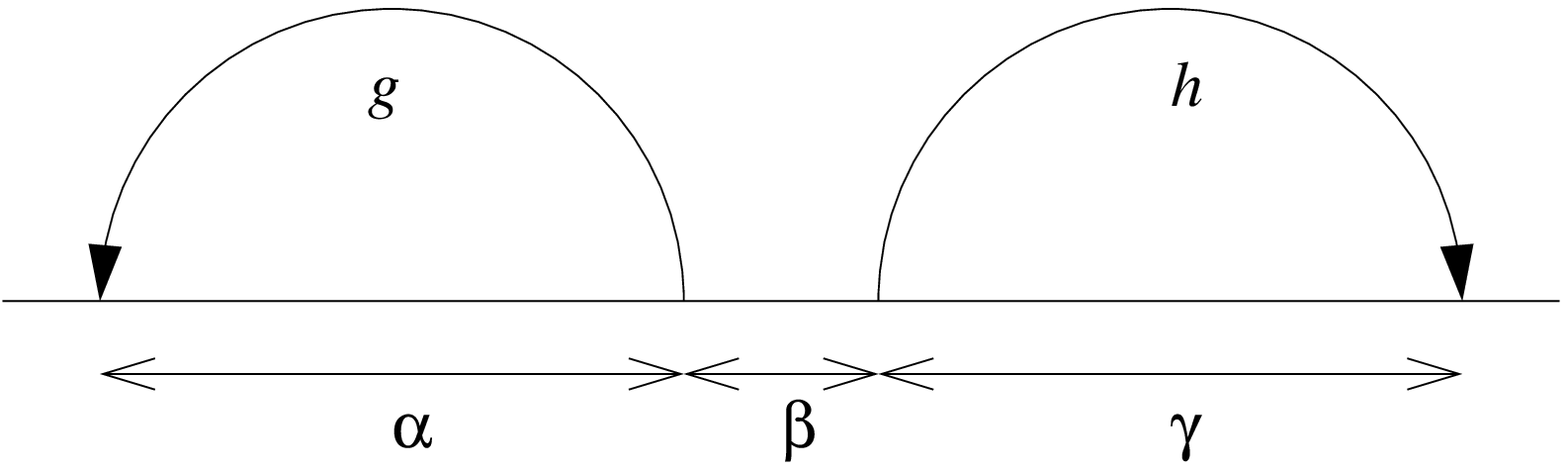}
\caption{$C(g,h)>1$, in the upper half plane.} \label{fig:23}
\end{figure}
    \[
        C(g,h) = \frac{ (r_g-a_h)(a_g-r_h) }{ (r_g - r_h)(a_g-a_h)
        } = \frac{(\beta+\gamma)(\alpha+\beta)}{\beta(\alpha + \beta
        + \gamma)} > 1.
    \]
    The inequality follows since $(\beta + \gamma)(\alpha + \beta) =
    \alpha \beta + \alpha \gamma + \beta^2 + \beta \gamma > \alpha
    \beta + \beta^2 + \beta \gamma = \beta (\alpha + \beta +
    \gamma)$. In the situation of figure \ref{fig:22}(ii) a
    similar computation gives $C(g,h) \in (0,1)$.
\qed
\end{proof}

\begin{lem}
\label{axes_configurations}
    Let $g,h \in PSL_2\R$ where $g,h,gh$ are hyperbolic and
    $\Tr[g,h]>2$. The possible arrangements of the axes of $g,h,gh$
    are shown in figure \ref{fig:24}, and have the
    following descriptions:
    \begin{enumerate}[(i)]
        \item
            $C(g,h) \in (1, \infty)$ and $\Tr(g) \Tr(h)
            \Tr(gh) > 8$;
        \item
            $C(g,h) \in (0,1)$ and $\Tr(g) \Tr(h)
            \Tr(gh) > 8$;
        \item
            $C(g,h) \in (0,1)$ and $\Tr(g) \Tr(h)
            \Tr(gh) > 8$;
        \item
            $C(g,h) \in (0,1)$ and $\Tr(g) \Tr(h)
            \Tr(gh) < -8$.
    \end{enumerate}
\begin{figure}[tbh]
\centering
\includegraphics[scale=0.5]{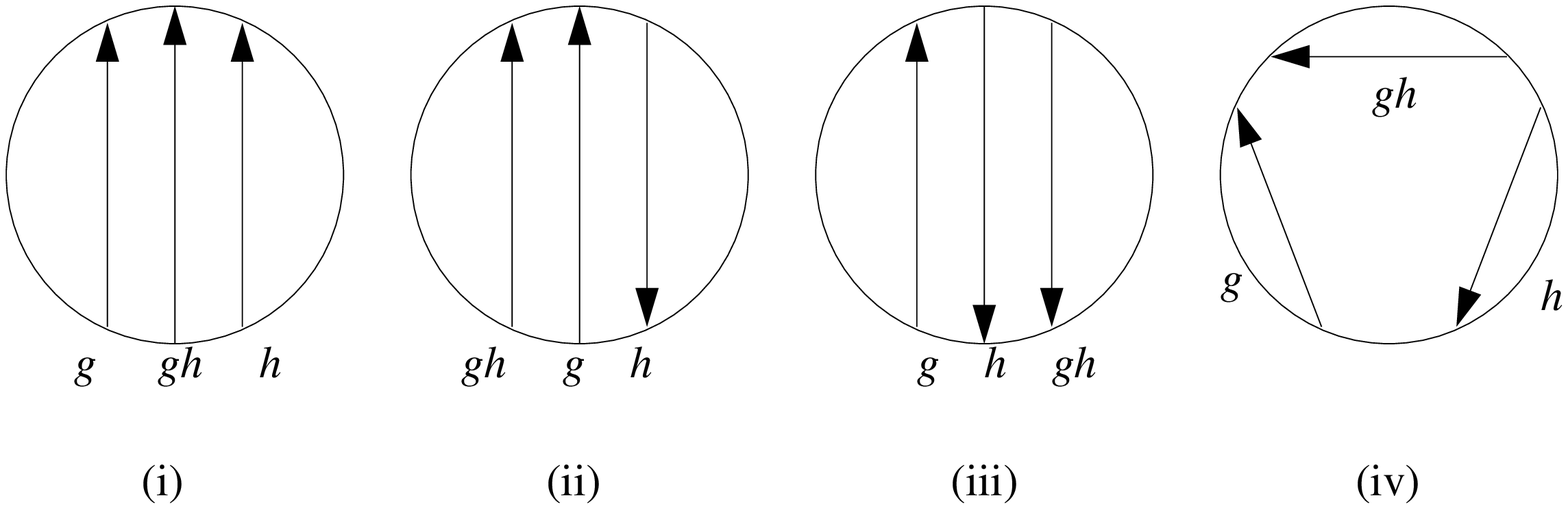}
\caption{The possible arrangements of axes of $g,h,gh$.}
\label{fig:24}
\end{figure}
\end{lem}

(We say nothing about what trace criteria might distinguish cases (ii) and (iii).) The proof will use the following theorem of Gilman and Maskit.
\begin{thm}[Gilman--Maskit \cite{Gilman_Maskit}]
\label{GM}
    Let $g,h$ be hyperbolic isometries such that $g,h$ have no fixed
    points in common, the axes of $g$ and $h$ do not intersect, and
    $gh$ is also hyperbolic.
    \begin{enumerate}
        \item
            If $C(g,h) \in (1, \infty)$ then $\Tr(g) \Tr(h) \Tr(gh) > 8$.
        \item
            If $C(g,h) \in (0,1)$ then $\Tr(g) \Tr(h) \Tr(gh) < -8$
            if and only if the axes of $g, h, gh$ bound a common
            region in $\hyp^2$.
    \end{enumerate}
    \qed
\end{thm}

\begin{proof}[of lemma \ref{axes_configurations}]
    From the above theorem and discussion, it is clear that figures \ref{fig:24}(i)--(iv) correspond to the cross ratios and products of traces as shown. But we must show that these figures are the only possible geometric arrangements of axes. We
    use the result that a hyperbolic isometry translating distance $d$ along an axis $l$ is the composition of two reflections, in
    lines perpendicular to $l$ spaced $d/2$ apart. Denote by $R_l$ the reflection in the line $l$. 

    Suppose $C(g,h)>1$. Let $\beta$ denote the common perpendicular of $\Axis g$ and $\Axis h$, and choose perpendiculars $\alpha, \gamma$ so that $g =
    R_\gamma R_\beta$ and $h = R_\beta R_\alpha$. Then $gh = R_\gamma R_\alpha$. Since $gh$ is hyperbolic, $\alpha, \gamma$ do not intersect, and their common perpendicular is the axis of $gh$. As $\Tr[g,gh] = \Tr[h,gh] = \Tr[g,h]>2$, lemma \ref{axes crossing} says that $\Axis gh$ is disjoint from $\Axis g$ and $\Axis h$. Thus, as shown in figure \ref{fig:25} (left), $\Axis gh$ must pass through the region
    bounded by $\Axis g$ and $\Axis h$, with the orientation shown. This is the situation of figure \ref{fig:24}(i).

    Now suppose $C(g,h) < 1$. Let
    $\alpha, \beta, \gamma$ be perpendiculars as before. We see by
    varying the possible positions of $\alpha$ and $\gamma$, and
    noting that $\Axis gh$ must be disjoint from $\Axis g$ and
    $\Axis h$, that there are precisely three possible locations
    for $\Axis gh$, namely those shown.
\qed
\end{proof}

Returning to the problem at hand, we have a basis with $\Tr g, \Tr
h, \Tr gh > 2$. So lemma \ref{axes_configurations} tells us that the
cases we must consider are precisely those in figure
\ref{fig:24}(i),(ii),(iii). We will explicitly show how to
choose $p$ so that $\Pent(g,h;p)$ is a non-degenerate simple
pentagon bounding an embedded disc.

\begin{itemize}

\item {\bf Case (i).} Assume $g,h,gh$ have axes as shown in figure \ref{fig:24}(i). Note that the axis of $hg$ is the image of the axis of $gh$ under either $h$ or $g^{-1}$. Thus $r_{hg}$ lies between $r_{gh}$ and $r_g$; and $a_{hg}$ lies between $a_{gh}$ and $a_h$. So $\Axis hg$ is arranged as shown in
figure \ref{fig:25} (right).

\begin{figure}[tbh]
\begin{center}
$\begin{array}{c}
\includegraphics[scale=0.4]{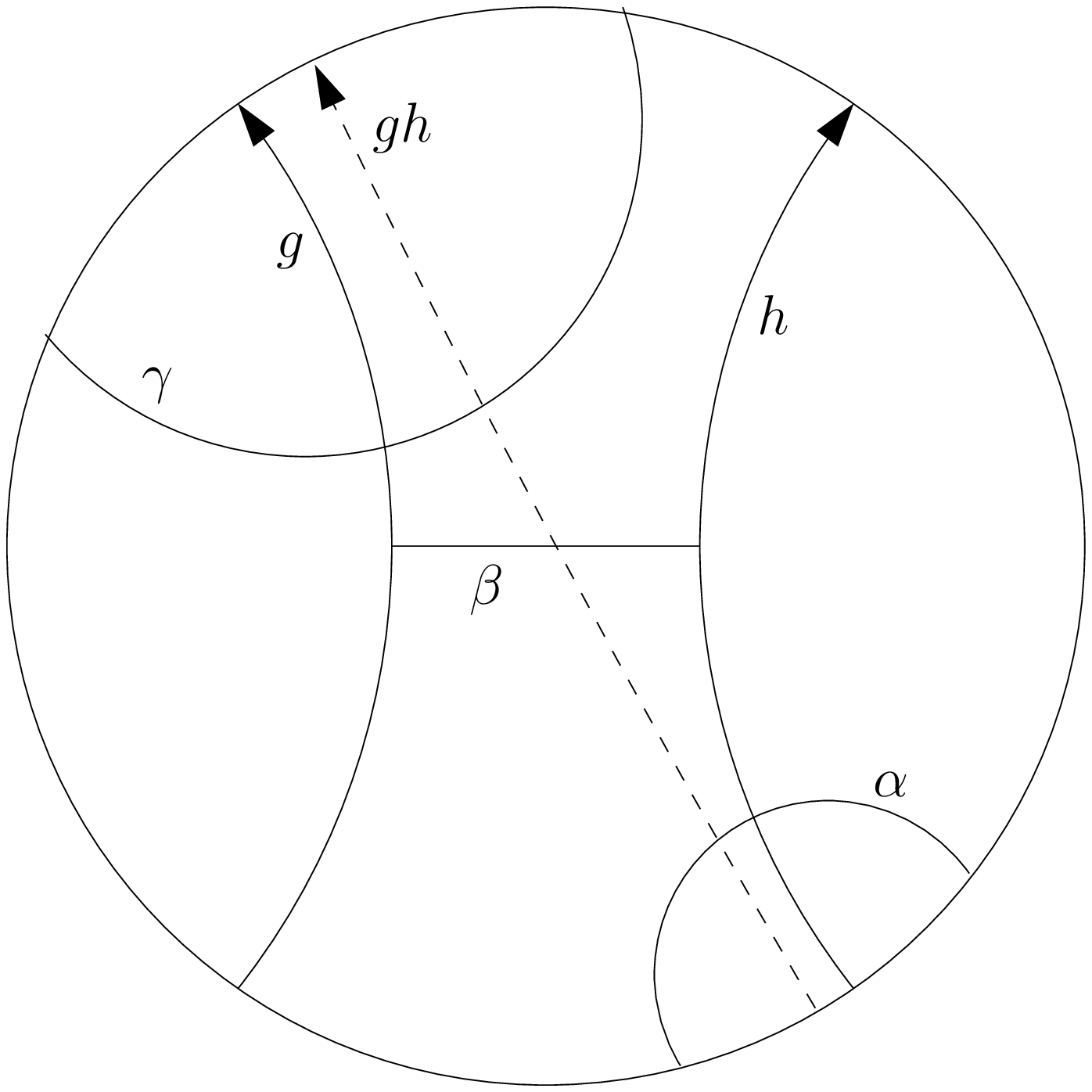}
\end{array}$
\hspace{1 cm}
$\begin{array}{c}
\includegraphics[scale=0.4]{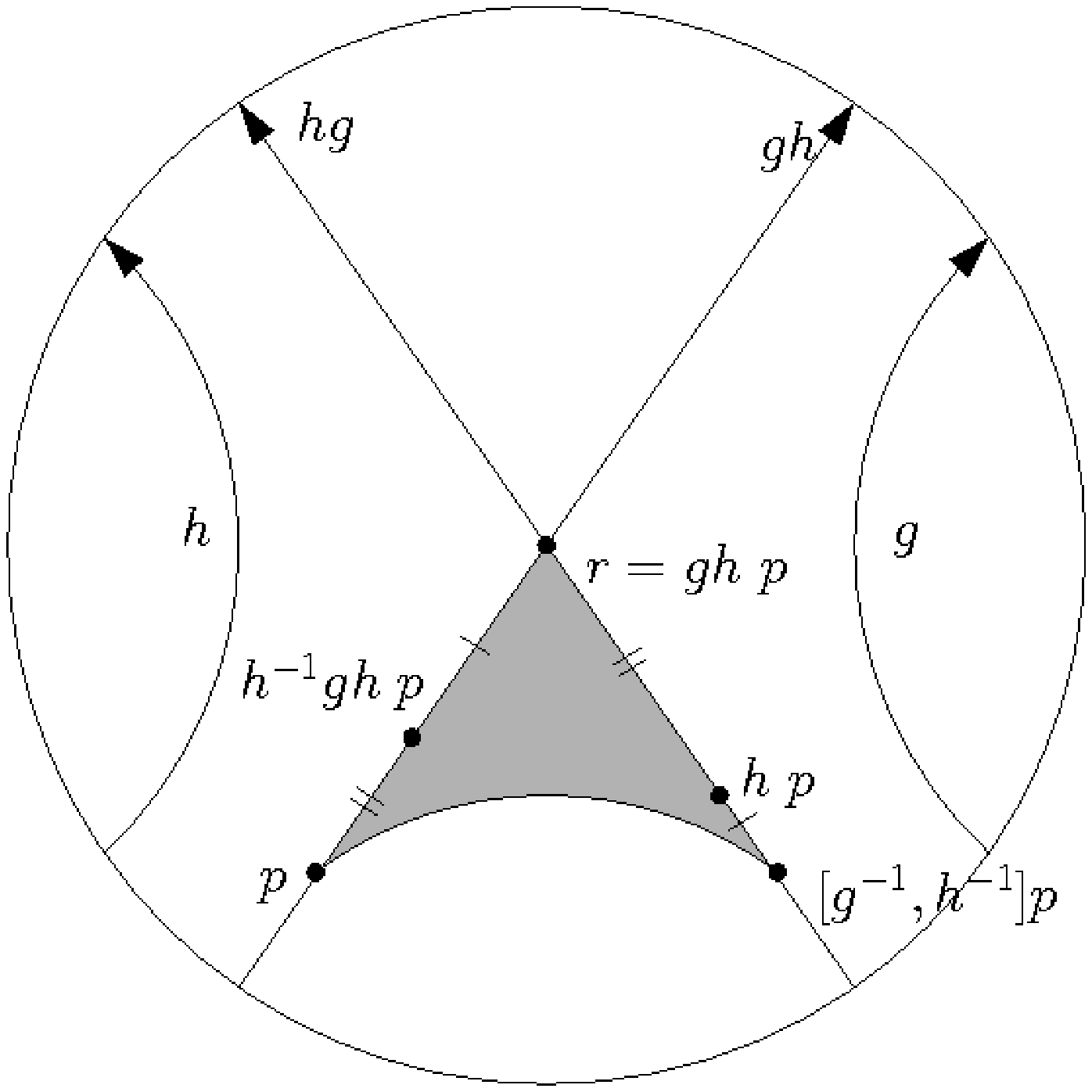}
\end{array}$
\caption{Left: the situation when $C(g,h)>1$. Right: construction in case (i).}
\label{fig:25}
\end{center}
\end{figure}

Let $r$ be the intersection of the axes of $gh$ and $hg$, and let $p = h^{-1}g^{-1}(r)$. Then we have immediately: $p \in \Axis gh$; $h p \in h(\Axis h^{-1} g^{-1}) = \Axis g^{-1} h^{-1} = \Axis hg$; $gh p = r = \Axis(hg) \cap \Axis(gh)$; $h^{-1} gh p \in h^{-1} (\Axis hg) = \Axis gh$; and $[g^{-1}, h^{-1}] p \in g^{-1} (\Axis gh) = \Axis hg$. Since $p = (gh)^{-1}(r)$, $p$ lies on $\Axis gh$ on the same side of $r$ as $r_{gh}$. Similarly $[g^{-1},h^{-1}]p = (hg)^{-1} (r)$ lies on $\Axis hg$ on the same side of $r$ as $r_{hg}$. Considering the action of $h$, we see that $h^{-1} gh p = h^{-1}(r)$ lies on $\Axis
gh$ on the same side of $r$ as $r_{gh}$; and similarly $h p$ lies on $\Axis hg$ on the same side of $r$ as $r_{hg}$. Further, since $h$ maps the directed segment $(h^{-1} ghp, p)$ to the directed segment $(gh p, h p)$, we see that $h^{-1} gh p$ lies on the same side of
$p$ as $r$. Similarly, $hp$ lies on the same side of $[g^{-1},h^{-1}]p$ as $r$. So $\Pent(g,h;p)$ appears as in figure \ref{fig:25} (right), and it is non-degenerate bounding an embedded disc.

Note $\Pent(g,h;p)$ contains two straight angles, so $\theta \in (2\pi, 3\pi)$. From corollary \ref{commutator_trace_info}, $\Tr[g,h]>2$ implies $[g,h]
\in \Hyp_0$, so we have by proposition \ref{twist_bounds} $\Twist([g^{-1},h^{-1}],p) \in (-\pi,\pi)$. From lemma \ref{pentagon_twist} we have $\theta = 3\pi \pm \Twist([g^{-1},h^{-1}],p)$; and $\Twist([g^{-1},h^{-1}],p)>0$ if and only if $p \to [g^{-1},h^{-1}]p$ bounds $\Pent(g,h;p)$ on its left, i.e. $\partial S$ traversed in the direction of $[G,H]$ bounds $S$ on its left.

\item {\bf Case (ii).} Assume $g,h,gh$ are arranged as in figure \ref{fig:24}(ii). Again $\Axis hg$, being the image of $\Axis gh$ under $h$ or $g^{-1}$, lies on the same side of $\Axis g$ as $\Axis gh$. The axes of $gh$ and $hg$ may or may not intersect; we do not care. See figure \ref{fig:26}.

\begin{figure}[tbh]
\begin{center}
$\begin{array}{c}
\includegraphics[scale=0.4]{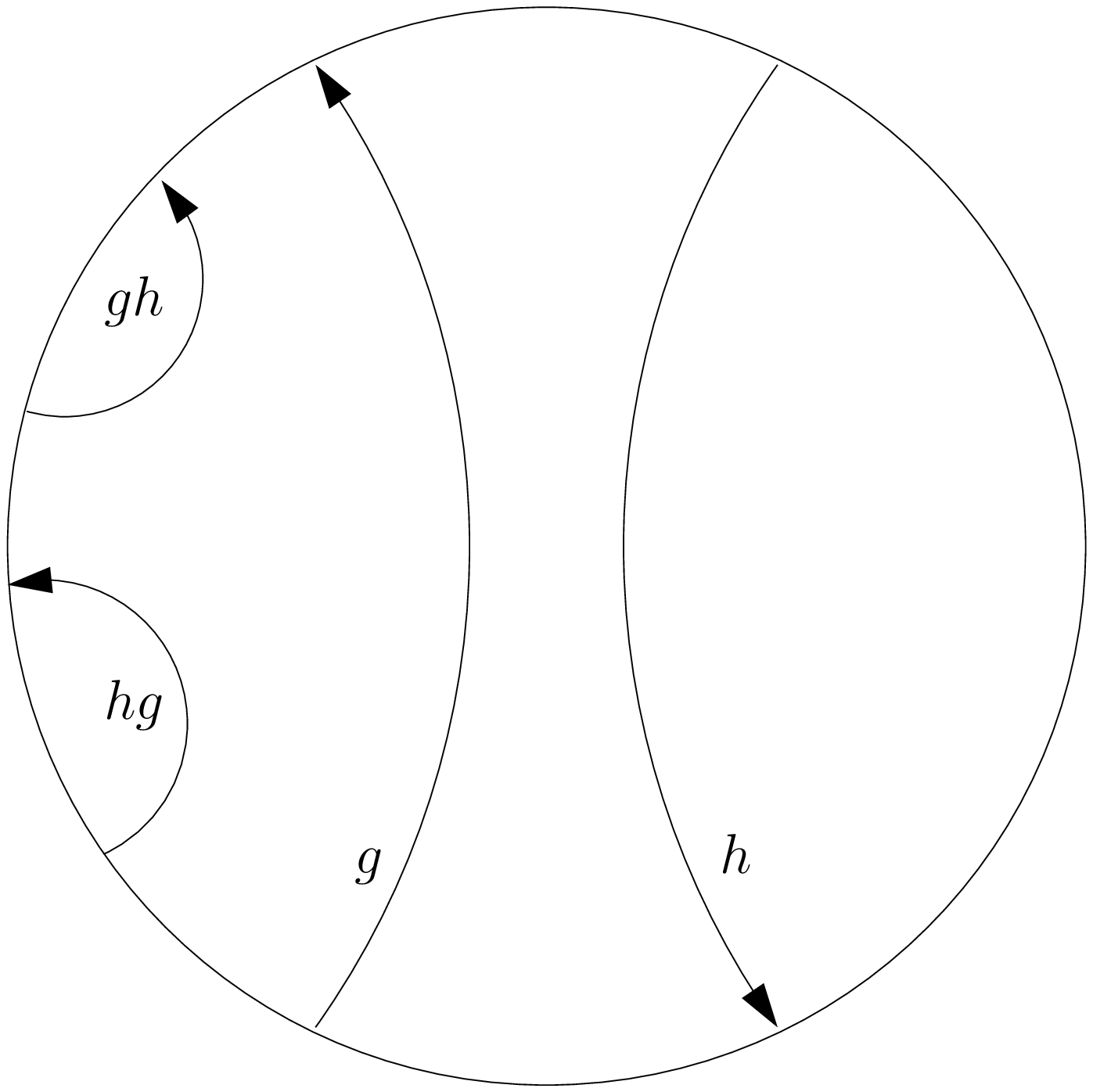}
\end{array}$
\hspace{1 cm}
$\begin{array}{c}
\includegraphics[scale=0.4]{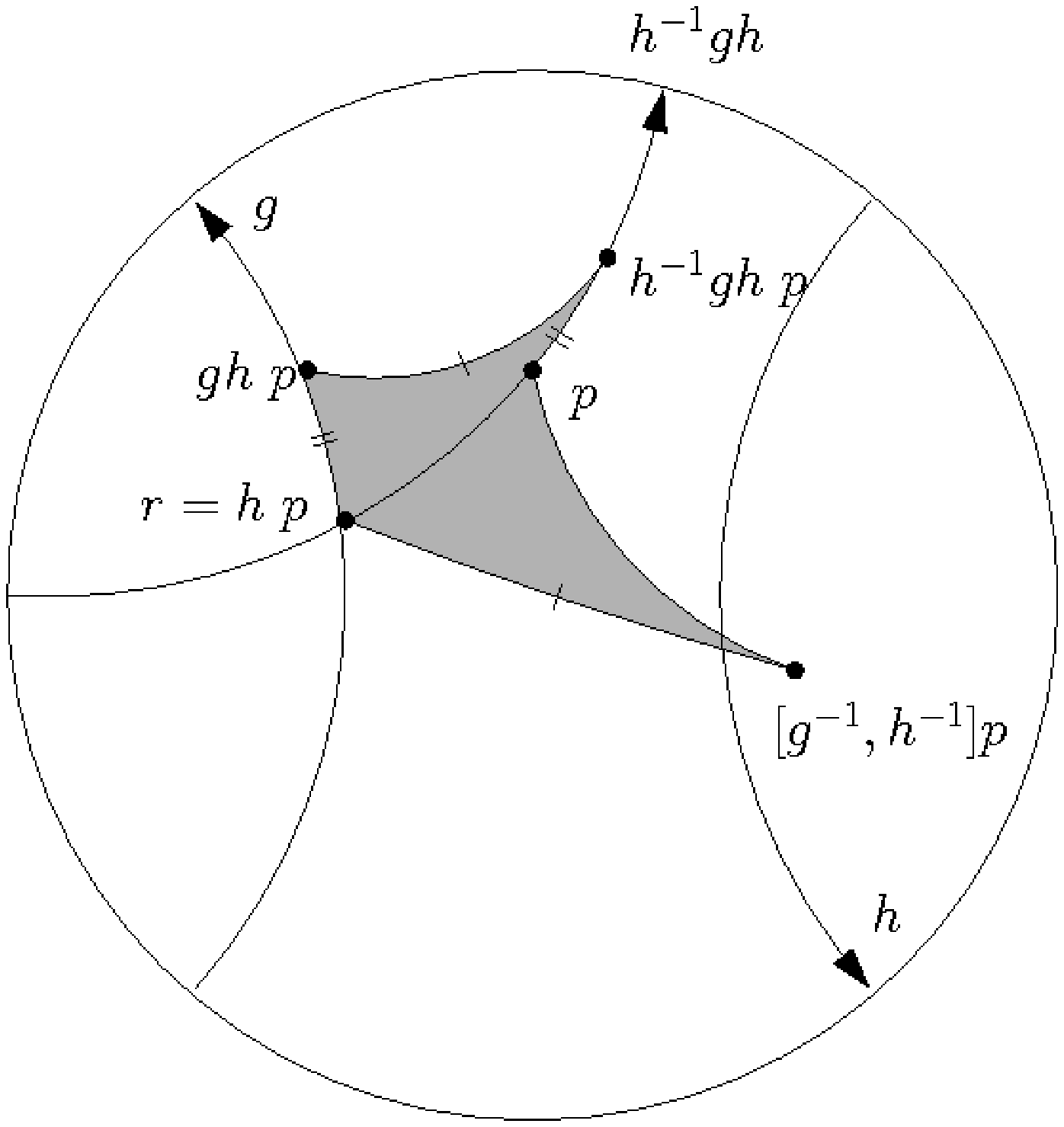}
\end{array}$
\caption{Axes of $g,h,gh,hg$ in case (ii); and construction.} \label{fig:26}
\end{center}
\end{figure}

Now $h^{-1}(a_{hg}) = a_{gh}$. Thus $h^{-1}(r_g)$ lies between $r_g$ and $a_{gh}$, in the same arc of the circle at infinity as $a_{hg}$. Also $h^{-1}(a_g)$ lies in the arc between $a_g$ and $r_h$. Since $h^{-1}(\Axis g) = \Axis h^{-1}gh$, we have the arrangement of axes as shown on the right of figure \ref{fig:26}

Let $r = \Axis g \cap \Axis h^{-1}gh$, and let $p=h^{-1} r$. Then we
have immediately: $p \in h^{-1}(\Axis g) = \Axis h^{-1}gh$; $h p = r = \Axis g \cap \Axis h^{-1} gh$; $ghp  \in \Axis g$; $h^{-1} ghp \in h^{-1}(\Axis g) = \Axis h^{-1}gh$. Considering the action of $h$, we see that $p=h^{-1}r$ lies to the same side of $r$ as $h^{-1}gh p$. Since $h$ maps the directed segment $(h^{-1}ghp , p)$ to $(ghp,hp)$, we see that $h^{-1}ghp$ lies on the opposite side of $p$ as $hp = r$. Now $h^{-1}ghp$ lies to the right of $\Axis g$ in the diagram shown, so $[g^{-1},h^{-1}]p$ lies to the right of $\Axis g$ also. And considering the action of $g^{-1}$, the image of $\Axis h^{-1}gh$ under $g^{-1}$ is disjoint from $\Axis h^{-1}gh$ and lies below it. So $[g^{-1},h^{-1}]p$ lies to the right of $\Axis g$ and also below $\Axis h^{-1} gh$. Hence $\Pent(g,h;p)$ is as shown in figure \ref{fig:26}, and is non-degenerate, bounding an embedded disc.

Examining figure \ref{fig:26}, we may chase unit vectors and see that $\Twist([g^{-1},h^{-1}],p) < 0$. By lemma \ref{pentagon_twist} we have $\theta = 3 \pi + \Twist([g^{-1},h^{-1}],p)$. If the opposite orientation occurs then $\Twist([g^{-1},h^{-1}],p) > 0$ and $\theta = 3\pi - \Twist([g^{-1},h^{-1}],p)$.

\item {\bf Case (iii).} This is similar to case (ii). By a similar argument as in case (ii), we deduce that $\Axis hg$ lies on the same side of $\Axis h$ as
$\Axis gh$. Thus, similarly to case (ii), we deduce that $\Axis g^{-1}hg$ lies as shown in figure \ref{fig:27}. Let $r$ be the intersection of $\Axis h$ and $\Axis g^{-1}hg$, and let $p=h^{-1}g^{-1}h r$. Then we have $h^{-1} gh p = r$, so $ghp \in \Axis h$ in the direction shown in figure \ref{fig:27}. The
segment $hp \to [g^{-1},h^{-1}]p$ is the image of $ghp \to h^{-1}ghp$ under $g^{-1}$, hence is a segment on $g^{-1} \Axis h = \Axis g^{-1} hg$ in the arrangement shown in figure \ref{fig:27}. Finally as $hp$ lies to the left of $\Axis h$, $p$ lies to the left of $\Axis h$, translated along the constant
distance curve from $\Axis h$ through $h p$. It follows that $p$ lies above $\Axis g^{-1}hg$. So $\Pent(g,h;p)$ lies as shown and is non-degenerate, bounding an embedded disc.

\begin{figure}[tbh]
\centering
\includegraphics[scale=0.4]{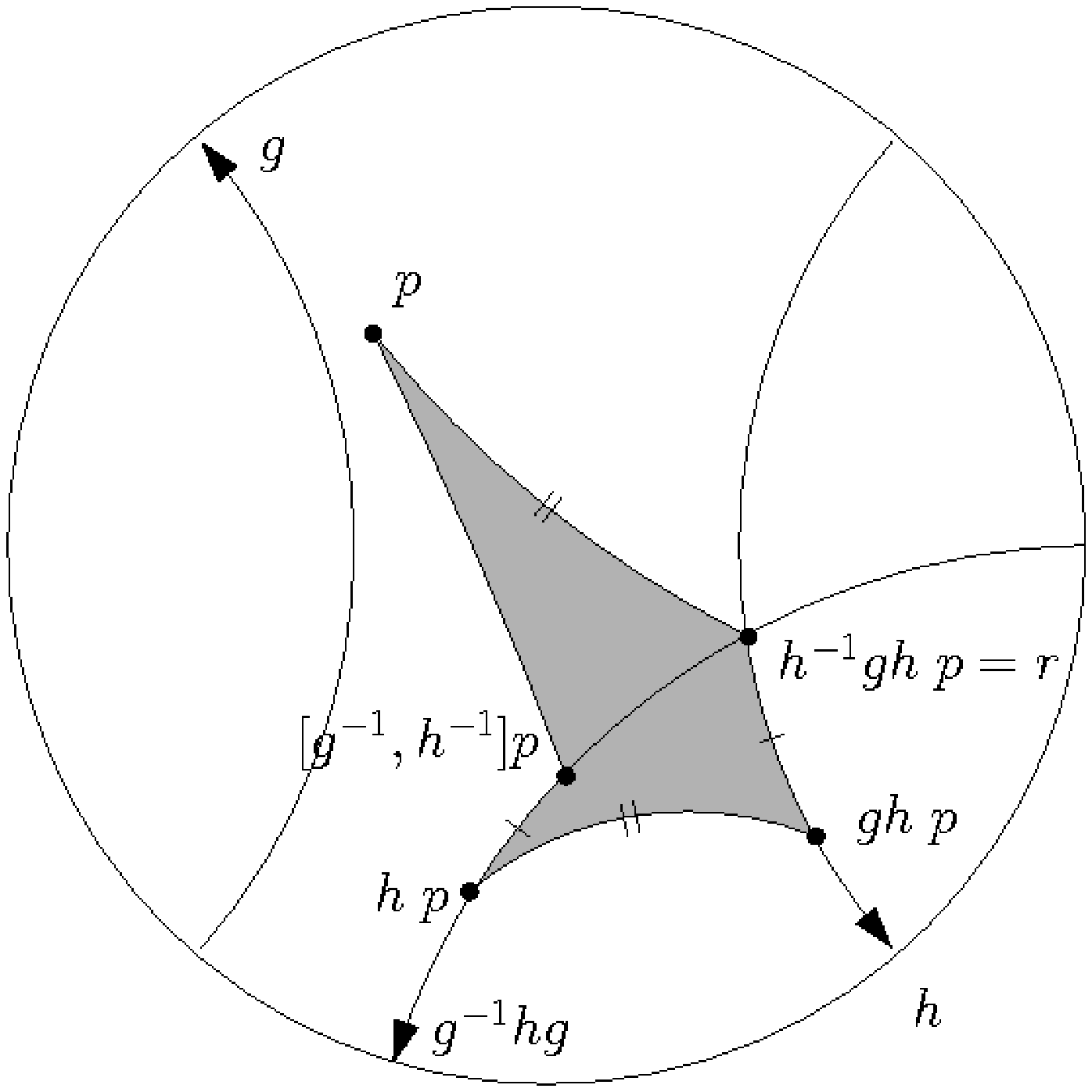}
\caption{Construction in case (iii).} \label{fig:27}
\end{figure}

By the same argument as in the previous case, $\theta \in (\pi, 3\pi)$ and $\theta = 3\pi \pm \Twist([g^{-1},h^{-1}],p)$, according to the orientation of $S$.
\end{itemize}

By lemma \ref{construction_lemma}, we conclude in each case that $\rho$ is the holonomy of a hyperbolic cone-manifold structure on $S$ with no interior cone points and at most one corner point. This completes the proof of proposition \ref{explicit_construction}, and indeed of theorem \ref{punctured_torus_theorem}.

Having completed the proof, we note that all pentagons we have constructed, which by lemma \ref{construction_lemma} were only required to be immersed, turned out to be embedded. Moreover, $p$ can be perturbed and the pentagon $\Pent(g,h;p)$ remains embedded. We now say more about this geometric flexibility.

\section{Non-uniqueness of geometric structures}
\label{rigidity section}

The geometric structures we have constructed are highly non-rigid.  For a given representation $\rho$, there may be many non-isometric structures on $S$, and even more non-isometric pentagons $\Pent(g,h;p)$.

The pentagon $\Pent(g,h;p)$ is the object containing the most information: not only does it encode a hyperbolic cone-manifold structure on the punctured torus, it also encodes a choice of basis curves $G,H$, and a particular location in $\hyp^2$. The hyperbolic cone-manifold structure on the torus $S$ encodes less information: it does not encode any choice of basis curves, but it does include particular locations in $\hyp^2$ via its developing map. (Here we take the view that a hyperbolic cone-manifold structure is a particular developing map, rather than an equivalence class of developing maps determined up to isometry.) The representation $\rho$ encodes less information again: it determines no basepoint from which to begin a developing map or a pentagon; and no choice of basis. A Markoff triple $(x,y,z)$ encodes even less information, since (for irreducible $\rho$) it encodes a conjugacy class of representatios.

The diagram below illustrates the situation schematically: solid arrows denote a complete determination of one object by another; broken arrows denote that some choice is involved. 
\[
\begin{array}{ccccccc}
    \left\{\begin{array}{c}
    \text{Pentagon} \\
    \Pent(g,h;p)
    \end{array}
    \right\}
    &
    \begin{array}{c}
    \rightarrow \\
    \dashleftarrow \\
    \text{choose} \\
    \text{basis}
    \end{array}
    &
    \left\{ \begin{array}{c}
    \text{Hyp. cone} \\
    \text{manifold} \\
    \text{structure on $S$} \\
    \text{(developing map)}
    \end{array} \right\}
    &
    \begin{array}{c}
    \rightarrow \\
    \dashleftarrow \\
    \text{choose} \\
    \text{basepoint}
    \end{array}
    &
    \left\{ \begin{array}{c}
    \text{Rep.} \\
    \rho \\
    \end{array} \right\}
    &
    \begin{array}{c}
    \rightarrow \\
    \dashleftarrow \\
    \text{choose} \\
    \text{conj. class}
    \end{array}
    &
    \left\{ \begin{array}{c}
    \text{Markoff triple} \\
    (x,y,z) \\
    \end{array} \right\}
\end{array}
\]
We consider the effect of choosing diffeent basepoints $p$; then the effect of choosing different bases for $\pi_1(S)$. 

Take $\rho$ as given, fix a basis $G,H$ of $\pi_1(S)$, and consider different choices of basepoint $p$. If we aleady have a pentagon $\Pent(g,h;p)$ bounding an immersed open disc, then with a small perturbation of $p$ to $p'$, the pentagon $\Pent(g,h;p')$ will still bound an immersed disc. The two pentagons will in general not be isometric. It is possible that different choices of $p$ can give non-isometric pentagons $\Pent(g,h;p)$, $\Pent(g,h;p')$ but isometric cone-manifold structures on $S$.

For instance, if $\rho$ is a discrete holonomy representation of a complete hyperbolic structure on $S$ with totally geodesic boundary, then the complete hyperbolic surface $S_0$ is the quotient of the \emph{convex core} of $\rho$, a convex subset of $\hyp^2$, by the image of $\rho$. Taking any $p$ on the axis of $[g^{-1},h^{-1}]$ gives a pentagon $\Pent(g,h;p)$ which is a fundamental domain for this complete hyperbolic structure on $S$. These pentagons are in general not isometric. Alternatively, if $p$ is chosen to lie slightly inside the convex core; then $\Pent(g,h;p)$ is a fundamental domain for a submanifold of $S_0$, which is obtained by truncating the hyperbolic punctured torus with totally geodesic boundary along a geodesic arc parallel to the boundary. It is a hyperbolic cone-manifold with corner angle greater than $\pi$. If instead $p$ lies outside the convex core, then we obtain a hyperbolic cone-manifold which contains $S_0$, with a cone angle less than $\pi$. Proposition \ref{orientation_1} shows that the cone angle depends only on the twist of $[g^{-1},h^{-1}]$ are $p$; hence only on the distance of $p$ from $\Axis[g^{-1}, h^{-1}]$: see figure \ref{fig:28}. Figure \ref{fig:29} shows partial developing maps for various choices of $p$.

\begin{figure}[tbh]
\begin{center}
$\begin{array}{c}
\includegraphics[width=4 cm, height=4 cm, angle=-90]{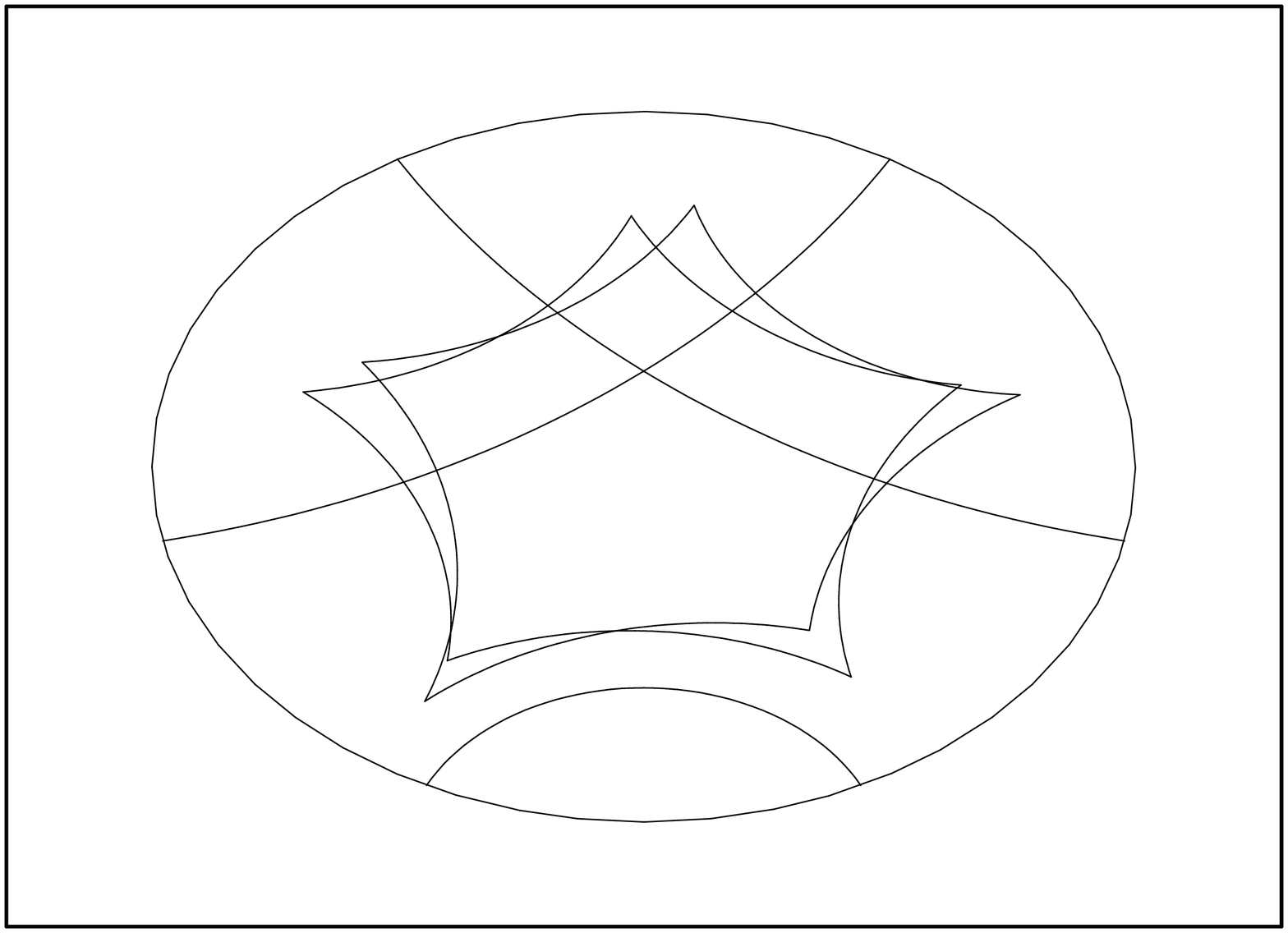}
\end{array}$
\hspace{1 cm}
$\begin{array}{c}
\includegraphics[scale=0.4]{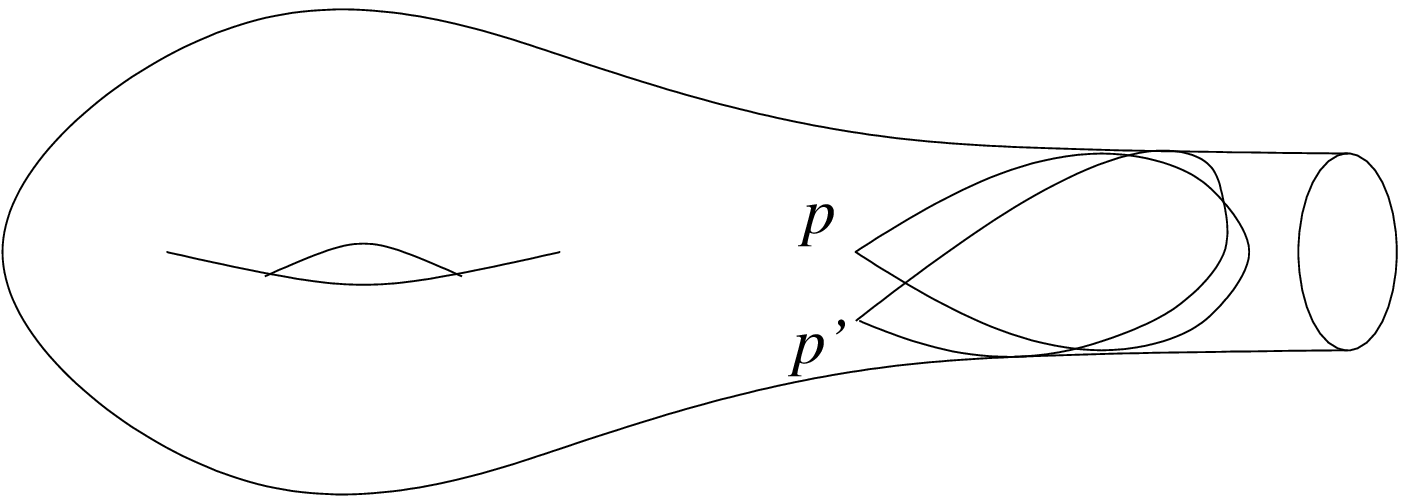}
\end{array}$
\end{center}
\caption[Distinct basepoints at the same distance from the
axis]{Distinct basepoints $p,p'$ at the same distance from the axis,
corresponding to a complete hyperbolic structure truncated at
different points the same distance from the boundary, creating corner points with the same angle.}
\label{fig:28}
\end{figure}

\begin{figure}[tbh]
\begin{center}
$\begin{array}{c}
\includegraphics[width=6.6 cm, height=6.6 cm, angle=-90]{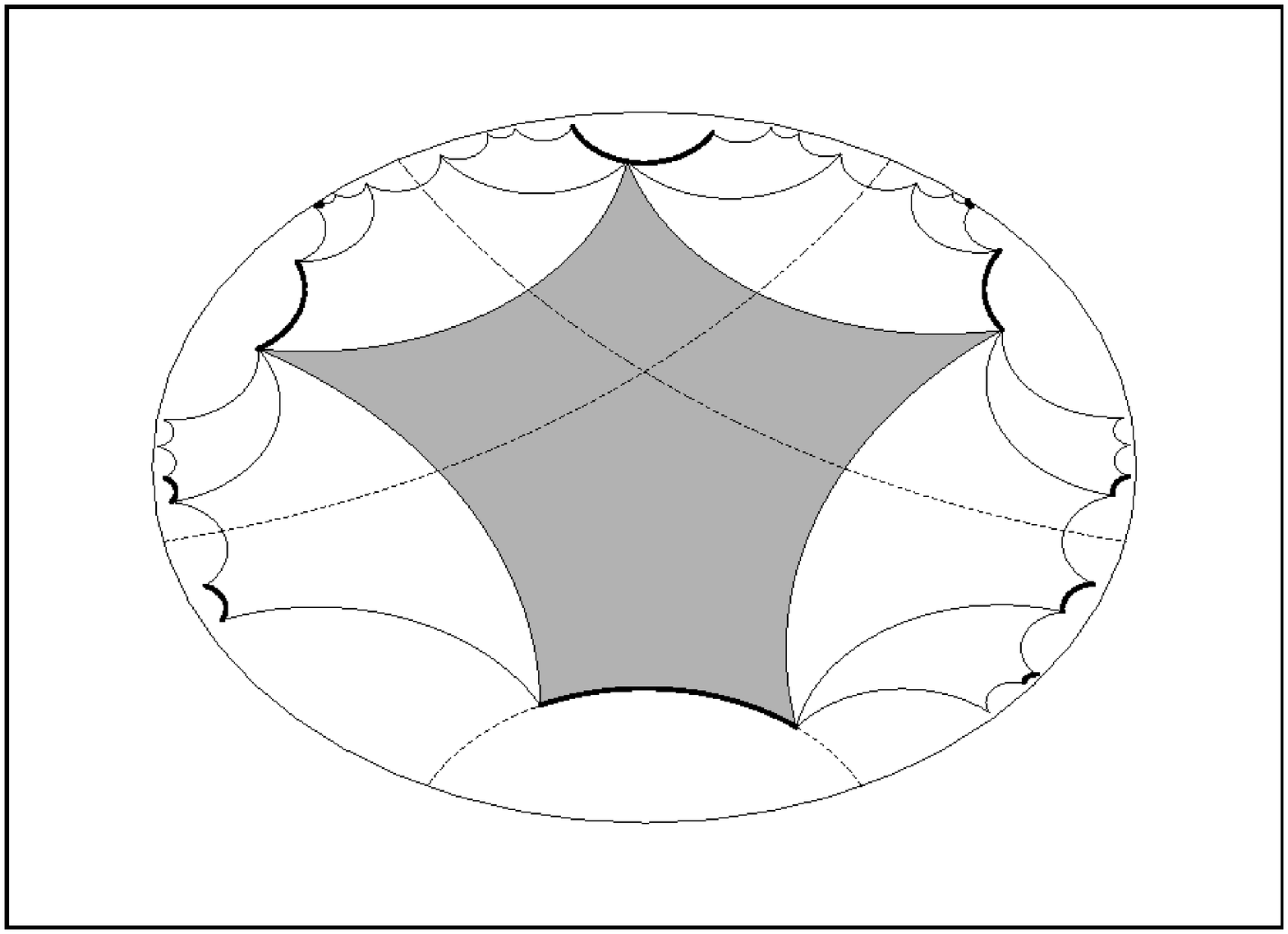}
\end{array}$
$\begin{array}{c}
\includegraphics[width=6.6 cm, height=6.6 cm, angle=-90]{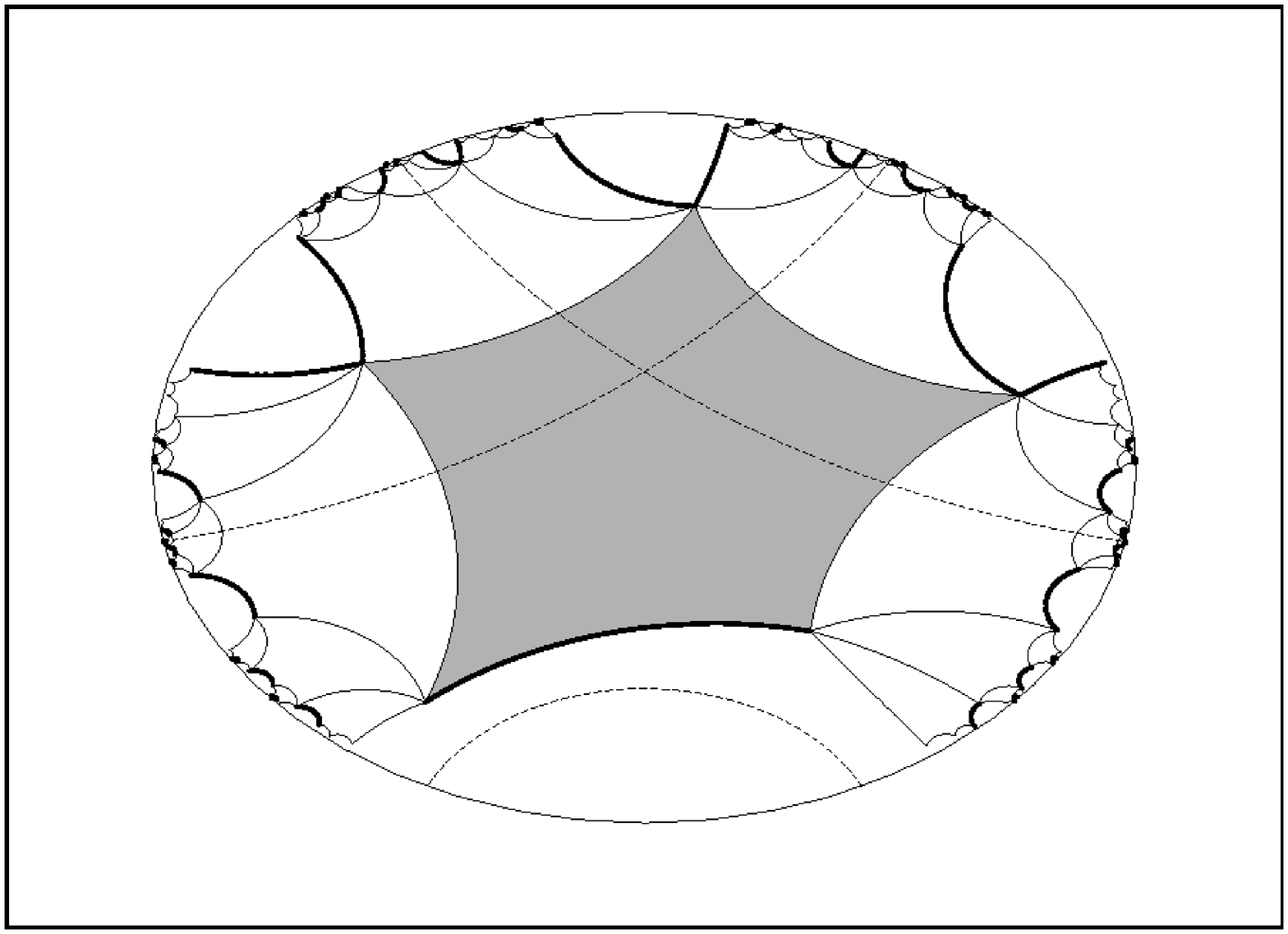}
\end{array}$
$\begin{array}{c}
\includegraphics[width=6.6 cm, height=6.6 cm, angle=-90]{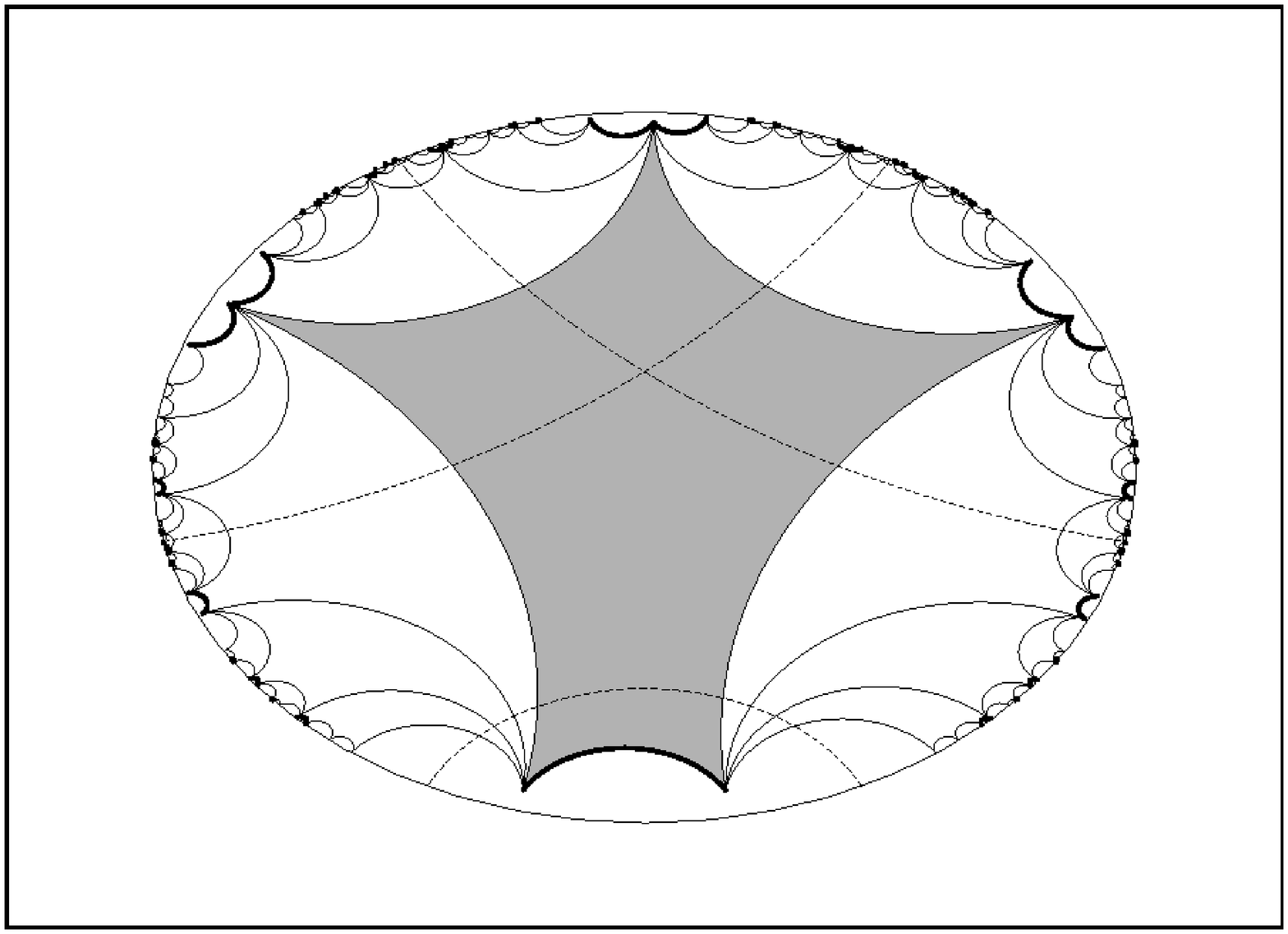}
\end{array}$
\caption[Part of the developing map]{Part of the developing map. Boundary edges are thickened. Axes of
$g,h,[g^{-1},h^{-1}]$ are dashed. Top left: geodesic boundary. Top right: truncated. Below: extended.} \label{fig:29}
\end{center}
\end{figure}

In this discrete case, the quotient of the entire hyperbolic plane by the image of $\rho$ is metrically a punctured torus which flares out past the geodesic boundary, and all of the above punctured tori are submanifolds thereof. Thus we may extend our hyperbolic cone-manifold arbitrarily far outwards from the complete structure $S_0$, and obtain corner angle arbitrarily close to $0$. In the other direction, it is possible to truncate $S_0$ with a geodesic loop based arbitrarily far from $\partial S_0$, but we must choose our basepoint judiciously. For instance it is possible to choose such a sequence of basepoints in $\hyp^2$ converging to the point at infinity which is an endpoint of $\Axis(h)$: see figure \ref{fig:30}. This gives a corner angle arbitrarily close to $2\pi$; but for a general choice of basepoint, the geodesic loop will not be simple, and the pentagonal fundamental domain will no longer bound an immersed disc.

\begin{figure}[tbh]
\begin{center}
$\begin{array}{c}
\includegraphics[width=6.6 cm, height=6.6 cm, angle=-90]{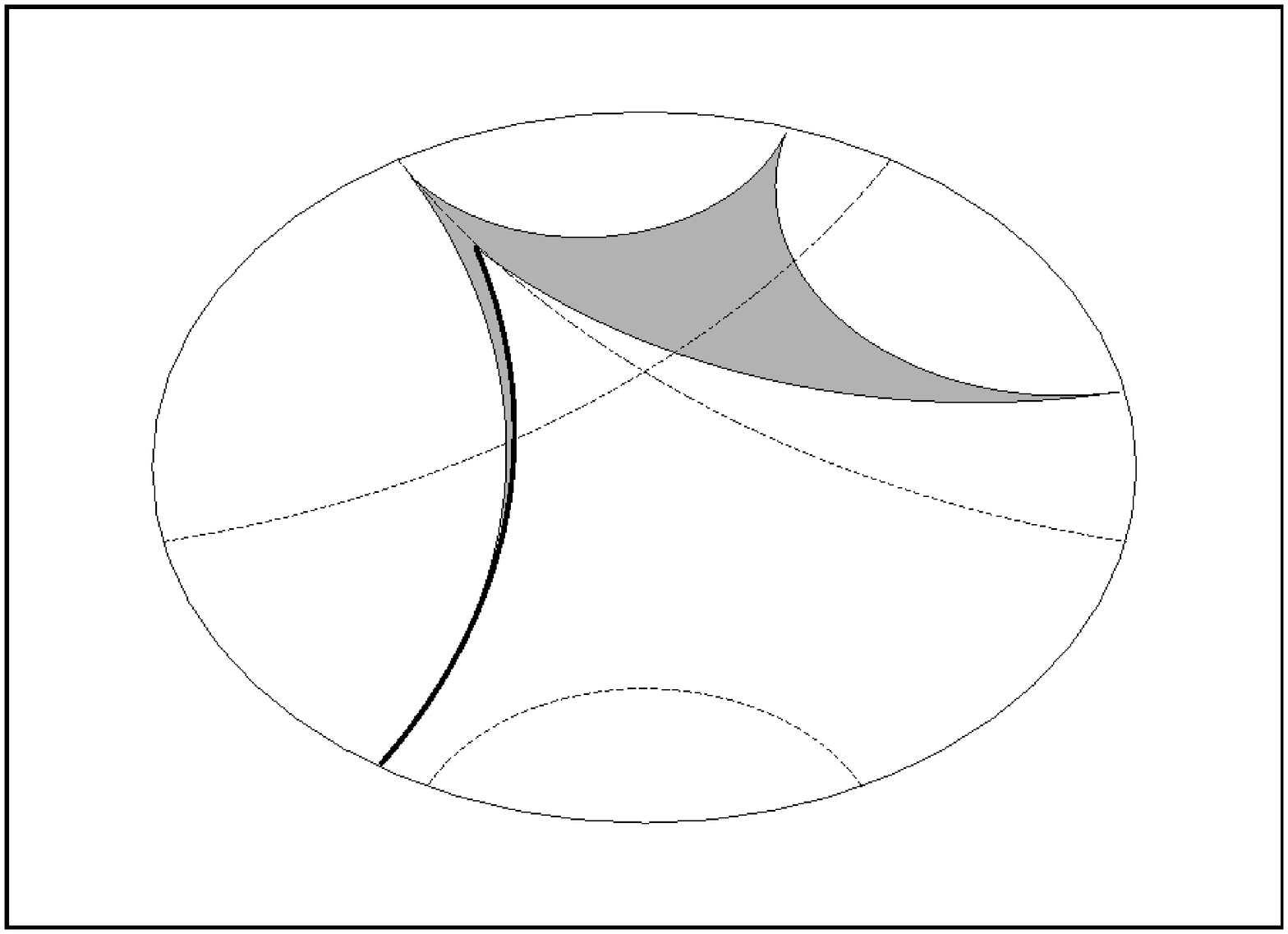}
\end{array}$
$\begin{array}{c}
\includegraphics[width=6.6 cm, height=6.6 cm, angle=-90]{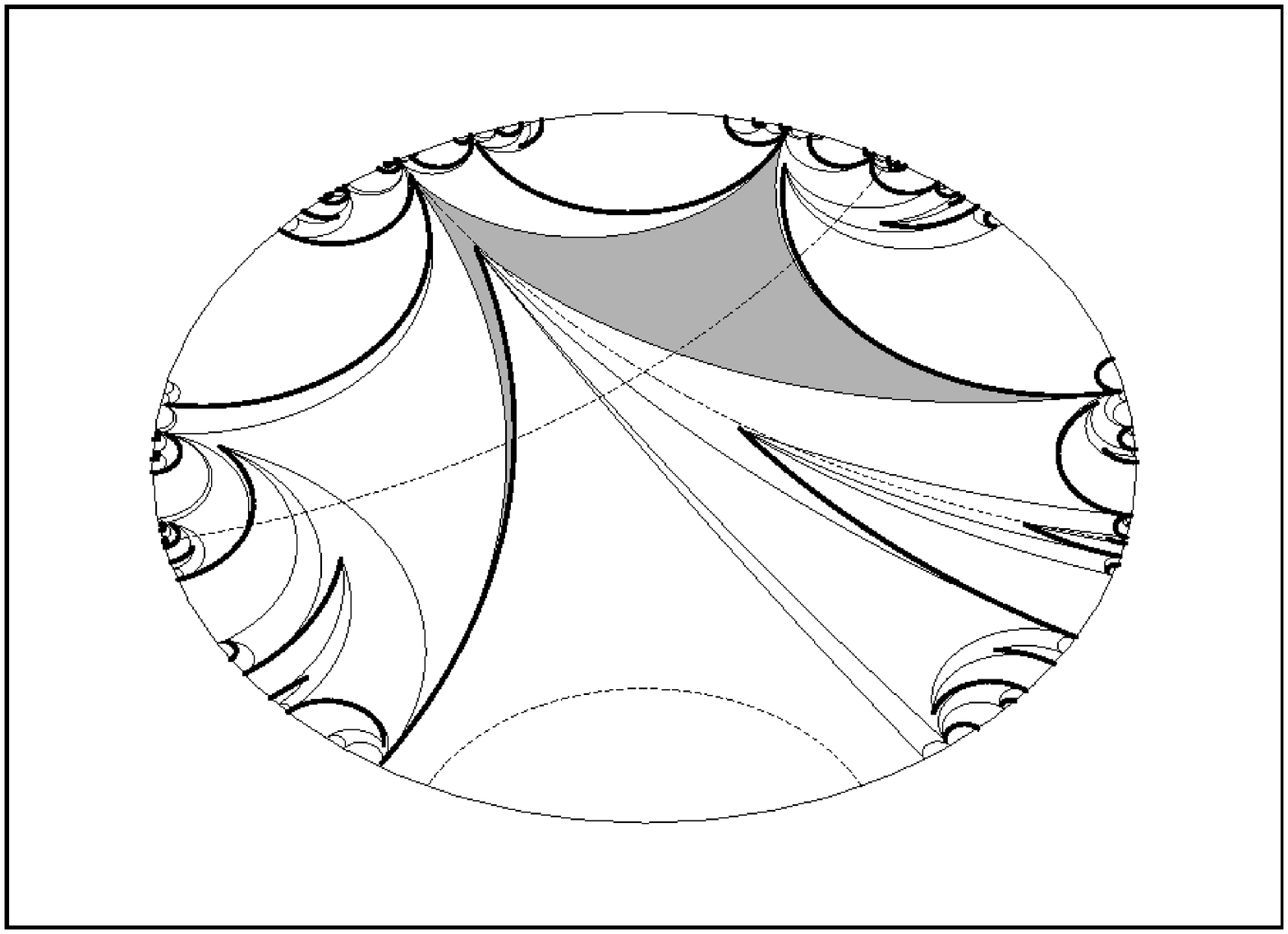}
\end{array}$
\caption{A well-chosen basepoint far inside $S_0$; corresponding partial developing map, with corner angle close to $2\pi$.} \label{fig:30}
\end{center}
\end{figure}

Given a hyperbolic surface with totally geodesic boundary, there is a fixed width $w$, depending only on the length of the boundary, called the \emph{collar width}, such that the set of all points within $w$ of the boundary forms a topological cylinder called the \emph{collar}. (See e.g. \cite{Buser} for details.) Given any point inside the collar, we may take a boundary-parallel geodesic arc from that point to itself, which lies in the collar, and cut along it, truncating the surface. The collar region develops to a convex region in the convex core of $\rho$ consisting of points at distance $\leq w$ from $\Axis[g^{-1},h^{-1}]$. Choosing $p$ in the convex core within a distance of $w$ of $\Axis[g^{-1},h^{-1}]$ gives $\Pent(g,h;p)$ simple bounding an embedded disc. That is, choice of $p$ need not be judicious in this region.

When $\rho$ is not discrete, we cannot think of a cone-manifold structure on $S$ as a complete structure which has been truncated or extended; the developing map is in general not a tessellation. But similar non-rigidity exists: perturbing $p$ gives non-isometric structures on $S$ with holonomy $\rho$. Section \ref{sec:-2_to_2} presents a detailed example, when $-2<\Tr[g,h]<2$. Proposition \ref{elliptic_construction} shows that perturbing $p$ from certain directions from $\Fix [g^{-1},h^{-1}] = r$ produces a desirable punctured torus; and taking $p = r$ can be considered the holonomy of a hyperbolic cone-manifold structure on a (non-punctured) torus $T$ with a single cone point $s$ with a quadrilateral fundamental domain. Figure \ref{fig:31} illustrates such a developing map.

In fact, the relationship between $T$ and $S$ is explicit in our construction. Let the cone point $s$ on $T$ have angle $\varphi$; $\varphi$ is $2\pi$ minus the area of the quadrilateral fundamental domain, which by a limiting version of proposition \ref{prop:twist_area} is $\Twist([g^{-1},h^{-1}],r)$; so $\varphi = 2\pi - \Twist([g^{-1},h^{-1}],r)$ as in section \ref{sec:-2_to_2}. When $\varphi < \pi$, i.e. $\Twist([g^{-1},h^{-1}],r) \in (\pi,2\pi)$, then for $q$ on $T$ close to $s$, there is a geodesic arc from $q$ to itself travelling around $s$; cutting along this arc and removing the piece containing $s$ gives a punctured torus $S$ with corner angle $\theta$. This is precisely the geometric effect of perturbing our quadrilateral into a pentagon in the first case of section \ref{sec:-2_to_2}, and as we saw in this case $\theta \in (\pi,2\pi)$. On the other hand, when $\varphi > \pi$, i.e. $\Twist([g^{-1}],h^{-1}],r) \in (0,\pi)$, then for $q$ on $T$ close to $s$ we may slice $T$ along the geodesic segment $qs$ and then glue on an isosceles triangle, with two sides of length $qs$, so that the third side becomes the boundary of the punctured torus $S$. This is the geometric effect of perturbing our quadrilateral into a pentagon in the second case of section \ref{sec:-2_to_2}; the developing map swallows the point $r$, and $\theta \in (2\pi, 3\pi)$.

The cases $\Tr[g,h]=2$ and $\Tr[g,h]>2$, as we have seen, are somewhat more complicated, with no nice underlying punctured or closed surface; but similar non-rigidity, through perturbation of pentagons, still exists.

\begin{figure}[tbh]
\begin{center}
$\begin{array}{c}
\includegraphics[width=6.6 cm, height=6.6 cm, angle=-90]{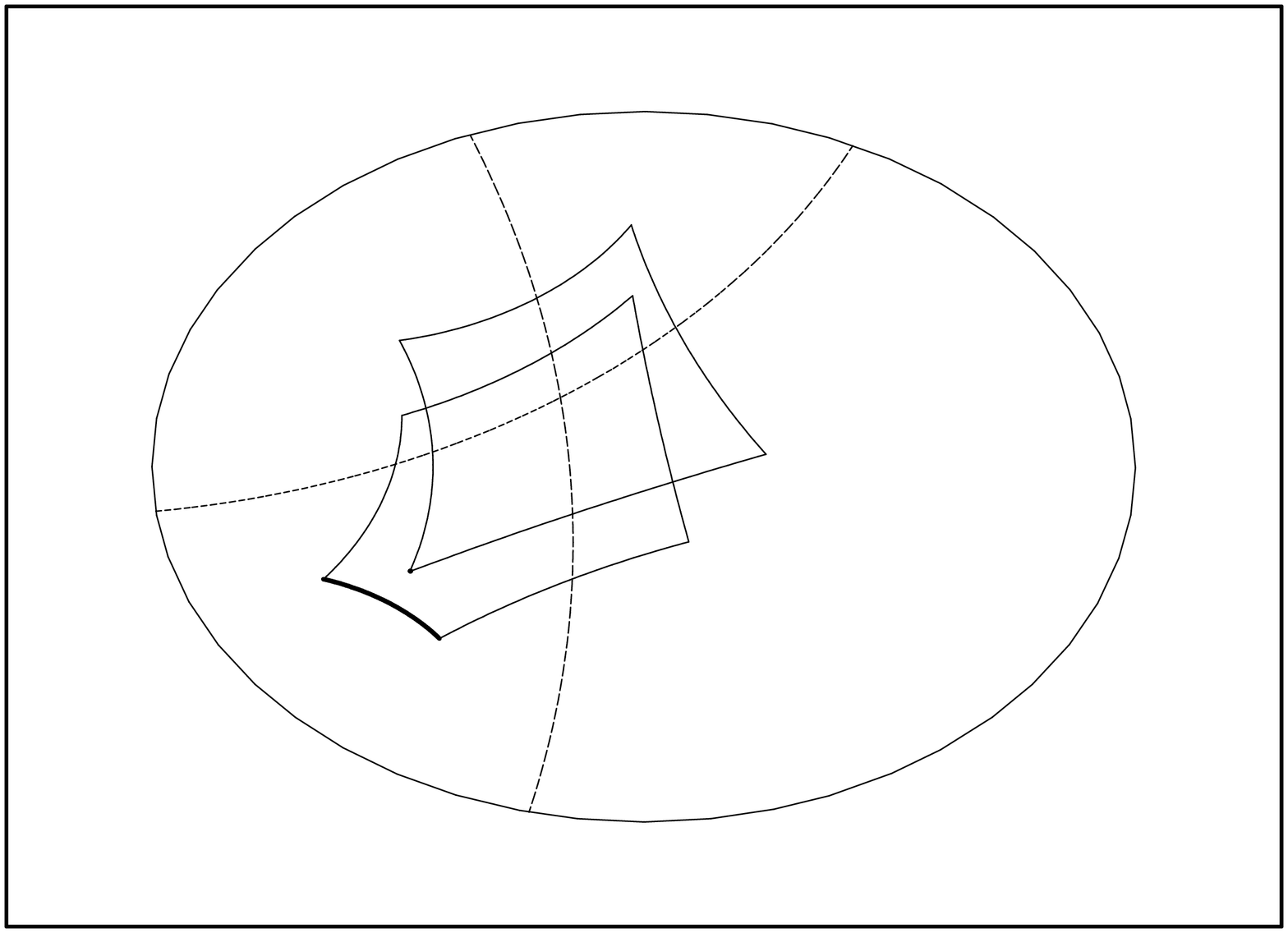}
\end{array}$
$\begin{array}{c}
\includegraphics[width=6.6 cm, height=6.6 cm, angle=-90]{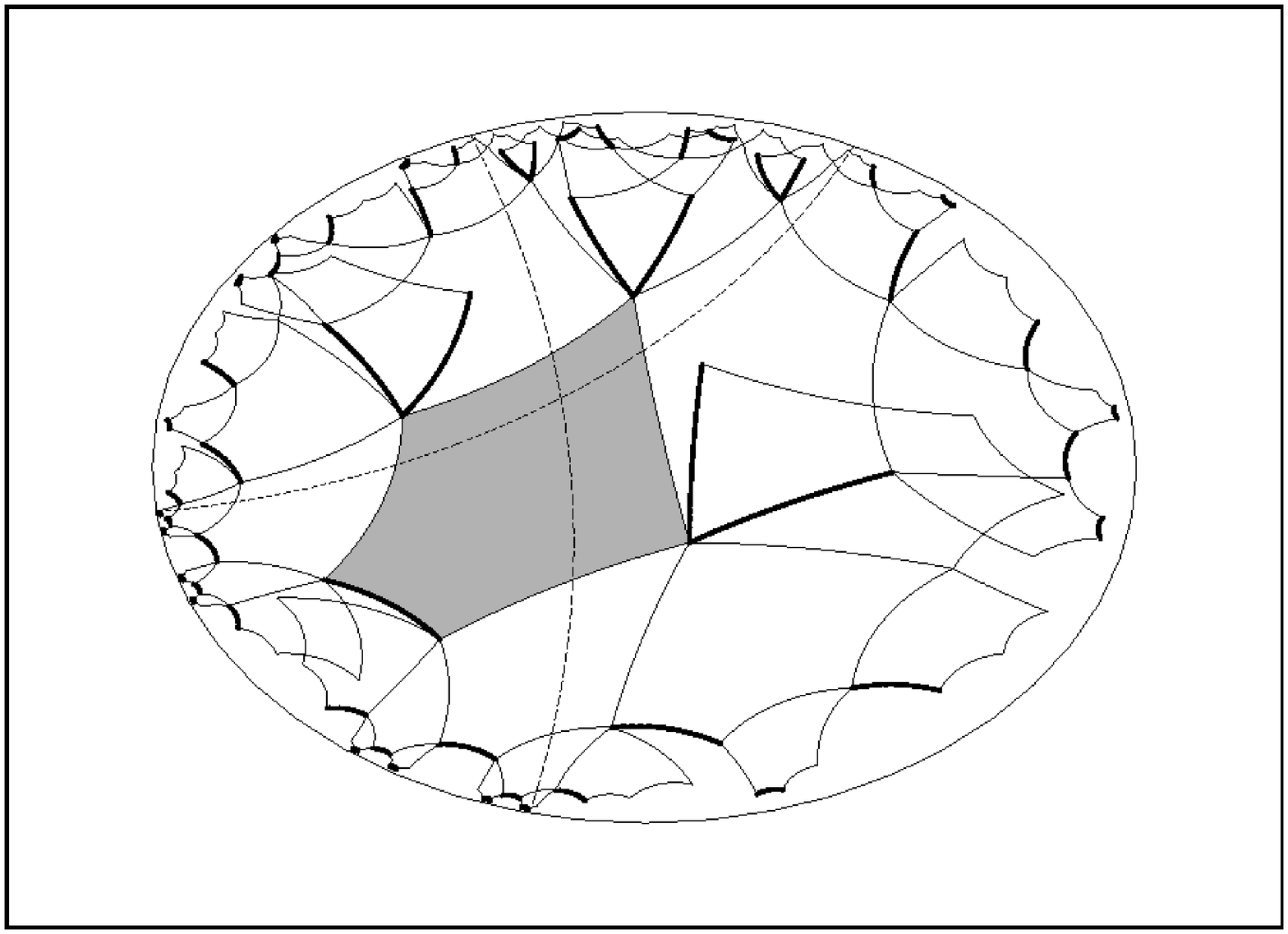}
\end{array}$
\caption{Fundamental domains for $T$ and $S$ and partial developing map for $S$; in this case $\theta > 2 \pi$.} \label{fig:31}
\end{center}
\end{figure}

\begin{figure}[tbh]
\centering
\includegraphics[scale=0.3]{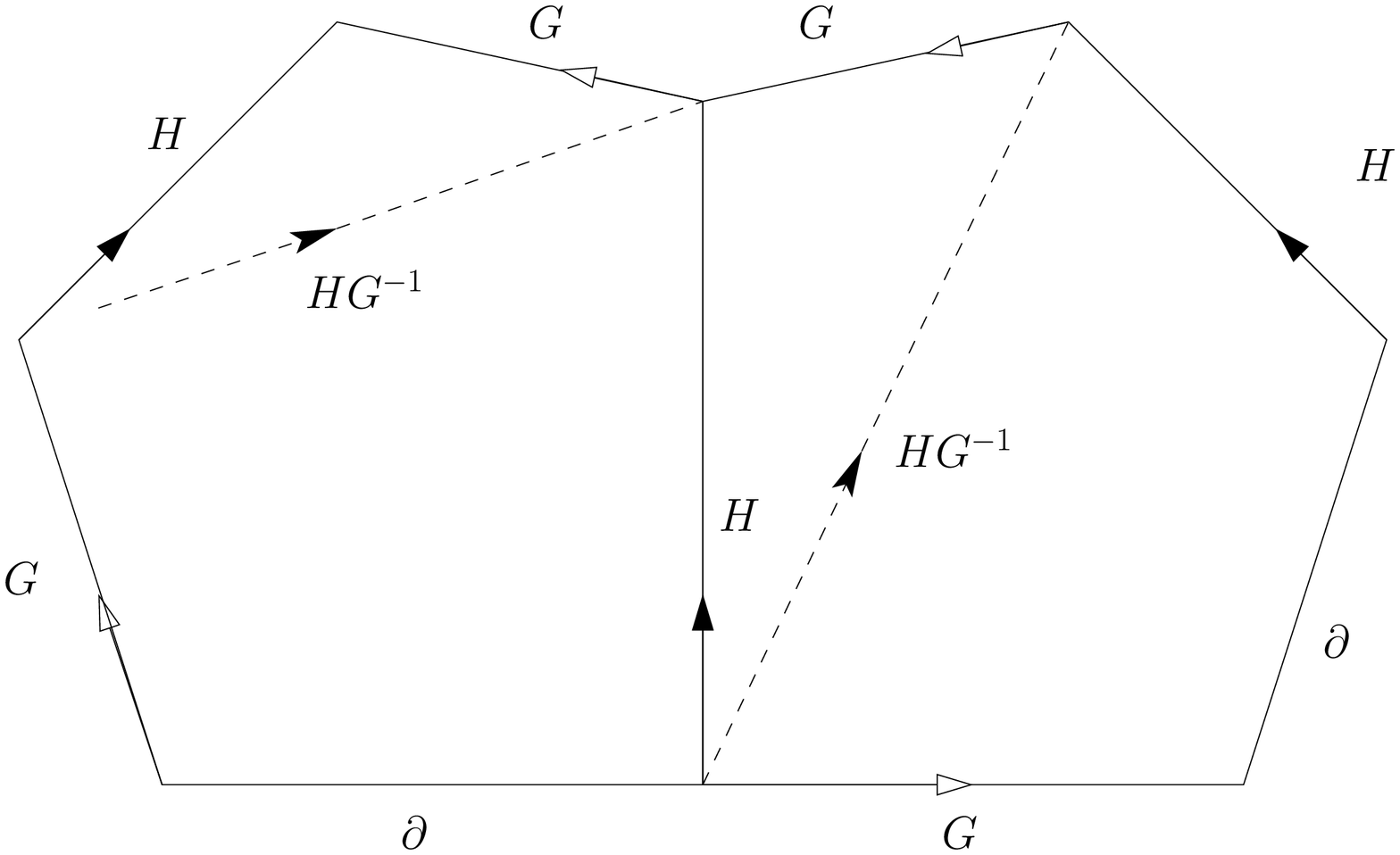}
\caption{A change of basis corresponding to a simple cut and paste
of pentagons.} \label{fig:32}
\end{figure}

The other aspect of non-rigidity is the choice of basis. If we change basis in $\pi_1(S, q)$ from $G,H$ to $G',H'$, we obtain another pentagon $\Pent(g',h';p')$ describing a hyperbolic cone-manifold structure on $S$. The change of basis $(G,H) \mapsto (G, HG^{-1})$ (figure \ref{fig:32}) for instance has a simple geometric interpretation as cutting our pentagon along a diagonal and regluing, arising from a Dehn twist in $S$. Any basis change has such a cut-and-paste interpretation. Our constructions have given, starting from $\rho$ alone, judicious changes of basis and judicious choices of basepoint demonstrating explicit hyperbolic cone-manifold structures.

\section*{Acknowledgements}
This paper forms one of several papers arising from the author's Masters thesis \cite{Mathews05}, completed at the University of Melbourne under Craig Hodgson, whose advice and suggestions have been highly valuable. It was completed during the author's postdoctoral fellowship at the Universit\'{e} de Nantes, supported by a grant ``Floer Power'' from the ANR. Thanks also to the referee for an extensive list of comments.

\addcontentsline{toc}{section}{References}

\small

\bibliography{danbib}

\providecommand{\bysame}{\leavevmode\hbox to3em{\hrulefill}\thinspace}
\providecommand{\MR}{\relax\ifhmode\unskip\space\fi MR }
\providecommand{\MRhref}[2]{%
  \href{http://www.ams.org/mathscinet-getitem?mr=#1}{#2}
}
\providecommand{\href}[2]{#2}
\begin{thebibliography}{10}

\bibitem{Bowditch_McShane}
B.~H. Bowditch, \emph{A proof of {M}c{S}hane's identity via {M}arkoff triples},
  Bull. London Math. Soc. \textbf{28} (1996), no.~1, 73--78. \MR{MR1356829
  (96i:58137)}

\bibitem{Bowditch96}
\bysame, \emph{Markoff triples and quasi-{F}uchsian groups}, Proc. London Math.
  Soc. (3) \textbf{77} (1998), no.~3, 697--736. \MR{MR1643429 (99f:57014)}

\bibitem{BH}
Martin~R. Bridson and Andr{\'e} Haefliger, \emph{Metric spaces of non-positive
  curvature}, Grundlehren der Mathematischen Wissenschaften [Fundamental
  Principles of Mathematical Sciences], vol. 319, Springer-Verlag, Berlin,
  1999. \MR{MR1744486 (2000k:53038)}

\bibitem{Buser}
Peter Buser, \emph{Geometry and spectra of compact {R}iemann surfaces},
  Progress in Mathematics, vol. 106, Birkh\"auser Boston Inc., Boston, MA,
  1992. \MR{MR1183224 (93g:58149)}

\bibitem{CCGLS}
D.~Cooper, M.~Culler, H.~Gillet, D.~D. Long, and P.~B. Shalen, \emph{Plane
  curves associated to character varieties of {$3$}-manifolds}, Invent. Math.
  \textbf{118} (1994), no.~1, 47--84. \MR{MR1288467 (95g:57029)}

\bibitem{CHK}
Daryl Cooper, Craig~D. Hodgson, and Steven~P. Kerckhoff,
  \emph{Three-dimensional orbifolds and cone-manifolds}, MSJ Memoirs, vol.~5,
  Mathematical Society of Japan, Tokyo, 2000, With a postface by Sadayoshi
  Kojima. \MR{MR1778789 (2002c:57027)}

\bibitem{Culler-Shalen}
Marc Culler and Peter~B. Shalen, \emph{Varieties of group representations and
  splittings of {$3$}-manifolds}, Ann. of Math. (2) \textbf{117} (1983), no.~1,
  109--146. \MR{MR683804 (84k:57005)}

\bibitem{Dunfield99}
Nathan~M. Dunfield, \emph{Cyclic surgery, degrees of maps of character curves,
  and volume rigidity for hyperbolic manifolds}, Invent. Math. \textbf{136}
  (1999), no.~3, 623--657. \MR{MR1695208 (2000d:57022)}

\bibitem{Eisenbud-Hirsch-Neumann}
David Eisenbud, Ulrich Hirsch, and Walter Neumann, \emph{Transverse foliations
  of {S}eifert bundles and self-homeomorphism of the circle}, Comment. Math.
  Helv. \textbf{56} (1981), no.~4, 638--660. \MR{MR656217 (83j:57016)}

\bibitem{Francaviglia}
Stefano Francaviglia, \emph{Hyperbolicity equations for cusped 3-manifolds and
  volume-rigidity of representations}, Ph.D. thesis, Scuola Normale Superiore
  Pisa, 2003.

\bibitem{Fricke}
R.~Fricke, \emph{\"{U}ber die theorie der automorphen modulgrupper}, Nachr.
  Akad. Wiss. G\"{o}ttingen (1896), 91--101.

\bibitem{Fricke_Klein}
Robert Fricke and Felix Klein, \emph{Vorlesungen \"uber die {T}heorie der
  automorphen {F}unktionen. {B}and 1: {D}ie gruppentheoretischen {G}rundlagen.
  {B}and {II}: {D}ie funktionentheoretischen {A}usf\"uhrungen und die
  {A}ndwendungen}, Bibliotheca Mathematica Teubneriana, B\"ande 3, vol.~4,
  Johnson Reprint Corp., New York, 1965. \MR{MR0183872 (32 \#1348)}

\bibitem{Gallo_Kapovich_Marden}
Daniel Gallo, Michael Kapovich, and Albert Marden, \emph{The monodromy groups
  of {S}chwarzian equations on closed {R}iemann surfaces}, Ann. of Math. (2)
  \textbf{151} (2000), no.~2, 625--704. \MR{MR1765706 (2002j:57029)}

\bibitem{Gilman_Maskit}
J.~Gilman and B.~Maskit, \emph{An algorithm for {$2$}-generator {F}uchsian
  groups}, Michigan Math. J. \textbf{38} (1991), no.~1, 13--32. \MR{MR1091506
  (92f:30062)}

\bibitem{Goldman_thesis}
William~M. Goldman, \emph{Discontinuous groups and the euler class}, Ph.D.
  thesis, Berkeley, 1980.

\bibitem{Goldman84}
\bysame, \emph{The symplectic nature of fundamental groups of surfaces}, Adv.
  in Math. \textbf{54} (1984), no.~2, 200--225. \MR{MR762512 (86i:32042)}

\bibitem{Goldman88}
\bysame, \emph{Topological components of spaces of representations}, Invent.
  Math. \textbf{93} (1988), no.~3, 557--607. \MR{MR952283 (89m:57001)}

\bibitem{Goldman03}
\bysame, \emph{The modular group action on real {${\rm SL}(2)$}-characters of a
  one-holed torus}, Geom. Topol. \textbf{7} (2003), 443--486 (electronic).
  \MR{MR2026539 (2004k:57001)}

\bibitem{Hoste-Shanahan}
Jim Hoste and Patrick~D. Shanahan, \emph{Trace fields of twist knots}, J. Knot
  Theory Ramifications \textbf{10} (2001), no.~4, 625--639. \MR{MR1831680
  (2002b:57012)}

\bibitem{Leleu}
Xavier Leleu, \emph{G\'{e}om\'{e}tries de courbure constante des
  3-vari\'{e}t\'{e}s et vari\'{e}t\'{e}s de caract\`{e}res de
  repr\'{e}sentations dans $sl_2(\mathbb{C})$}, Ph.D. thesis, Universit\'{e} de
  Provence, Marseille, 2000.

\bibitem{Lyndon_Schupp}
Roger~C. Lyndon and Paul~E. Schupp, \emph{Combinatorial group theory}, Classics
  in Mathematics, Springer-Verlag, Berlin, 2001, Reprint of the 1977 edition.
  \MR{MR1812024 (2001i:20064)}

\bibitem{Maclachlan_Reid}
Colin Maclachlan and Alan~W. Reid, \emph{The arithmetic of hyperbolic
  3-manifolds}, Graduate Texts in Mathematics, vol. 219, Springer-Verlag, New
  York, 2003. \MR{MR1937957 (2004i:57021)}

\bibitem{Magnus80}
Wilhelm Magnus, \emph{Rings of {F}ricke characters and automorphism groups of
  free groups}, Math. Z. \textbf{170} (1980), no.~1, 91--103. \MR{MR558891
  (81a:20043)}

\bibitem{Magnus_Karrass_Solitar}
Wilhelm Magnus, Abraham Karrass, and Donald Solitar, \emph{Combinatorial group
  theory}, second ed., Dover Publications Inc., Mineola, NY, 2004,
  Presentations of groups in terms of generators and relations. \MR{MR2109550
  (2005h:20052)}

\bibitem{Matelski}
J.~Peter Matelski, \emph{The classification of discrete {$2$}-generator
  subgroups of {${\rm PSL}(2,\,{\bf R})$}}, Israel J. Math. \textbf{42} (1982),
  no.~4, 309--317. \MR{MR682315 (84d:10029)}

\bibitem{Mathews03}
Daniel Mathews, \emph{Mahler's unfinished symphony: Etudes in knots, algebra
  and geometry}, honours project, University of Melbourne, 2003. Available at
  the author's website, {\small{\url{http://www.danielmathews.info}}}, 2003.

\bibitem{Mathews05}
\bysame, \emph{From algebra to geometry: A hyperbolic odyssey; the construction
  of geometric cone-manifold structures with prescribed holonomy}, Masters
  thesis, University of Melbourne, 2005. Available at the author's website,
  {\small{\url{http://www.danielmathews.info}}}, 2005.

\bibitem{Me10MScPaper2}
\bysame, \emph{Hyperbolic cone-manifold structures with prescribed holonomy
  {II}: higher genus}, {\small{\url{http://arxiv.org/abs/1006.5384}}}, 2010.

\bibitem{Me10MScPaper0}
\bysame, \emph{The hyperbolic meaning of the {M}ilnor--{W}ood inequality},
  {\small{\url{http://arxiv.org/abs/1006.5403}}}, 2010.

\bibitem{Milnor}
John Milnor, \emph{On the existence of a connection with curvature zero},
  Comment. Math. Helv. \textbf{32} (1958), 215--223. \MR{MR0095518 (20 \#2020)}

\bibitem{Nielsen18}
J.~Nielsen, \emph{Die {I}somorphismen der allgemeinen, unendlichen {G}ruppe mit
  zwei {E}rzeugenden}, Math. Ann. \textbf{78} (1964), no.~1, 385--397.
  \MR{MR1511907}

\bibitem{Nielsen27}
Jakob Nielsen, \emph{Untersuchungen zur {T}opologie der geschlossenen
  zweiseitigen {F}l\"achen}, Acta Math. \textbf{50} (1927), no.~1, 189--358.
  \MR{MR1555256}

\bibitem{Stillwell}
J.~Stillwell, \emph{The {D}ehn-{N}ielsen theorem}, Papers on Group Theory and
  Topology by Max Dehn, Springer-Verlag, Berlin, 1988.

\bibitem{Tan94}
Ser~Peow Tan, \emph{Branched {$\mathbf C{\rm P}^1$}-structures on surfaces with
  prescribed real holonomy}, Math. Ann. \textbf{300} (1994), no.~4, 649--667.
  \MR{MR1314740 (96m:57024)}

\bibitem{Thurston_notes}
William~P. Thurston, \emph{The geometry and topology of 3-manifolds},
  Mimeographed notes, 1979.

\bibitem{Thurston_book}
\bysame, \emph{Three-dimensional geometry and topology. {V}ol. 1}, Princeton
  Mathematical Series, vol.~35, Princeton University Press, Princeton, NJ,
  1997, Edited by Silvio Levy. \MR{MR1435975 (97m:57016)}

\bibitem{Troyanov}
Marc Troyanov, \emph{Prescribing curvature on compact surfaces with conical
  singularities}, Trans. Amer. Math. Soc. \textbf{324} (1991), no.~2, 793--821.
  \MR{MR1005085 (91h:53059)}

\bibitem{Wood}
John~W. Wood, \emph{Bundles with totally disconnected structure group},
  Comment. Math. Helv. \textbf{46} (1971), 257--273. \MR{MR0293655 (45 \#2732)}

\end{thebibliography}
\bibliographystyle{amsplain}

\end{document}